\documentclass[a4paper,reqno, 11pt]{amsart}  
\usepackage[DIV=12, oneside]{typearea}
\usepackage[utf8]{inputenc}
\usepackage[T1]{fontenc}
\usepackage[english]{babel}
\usepackage[centertags]{amsmath}
\usepackage{amstext,amssymb,amsopn,amsthm}
\usepackage{mathrsfs}
\usepackage{dsfont}
\usepackage{bbm}
\usepackage{thmtools}
\usepackage{graphicx}
\usepackage[titletoc,title]{appendix}

\usepackage[backgroundcolor=white, bordercolor=blue,
linecolor=blue]{todonotes}
\usepackage[colorlinks=true, linkcolor=black, citecolor=black]{hyperref}
\usepackage{enumitem}
\setlist[enumerate]{itemsep=0mm}

\addto\extrasenglish{}
\addto\extrasenglish{}
\addto\extrasenglish{}

\theoremstyle{plain}
\declaretheorem[title=Theorem, parent=section]{theorem}
\declaretheorem[title=Lemma,sibling=theorem]{lemma}
\declaretheorem[title=Proposition,sibling=theorem]{proposition}

\theoremstyle{definition}
\declaretheorem[title=Definition,sibling=theorem]{definition}

\declaretheorem[title=Remark, numbered=no]{remark*}
\declaretheorem[title=Example, sibling=theorem]{example}
\declaretheorem[title=Assumption, numbered=no]{assumption*}

\numberwithin{equation}{section}

\newcommand{\N}{\mathds{N}}
\newcommand{\R}{\mathds{R}}

\def\hmath$#1${\texorpdfstring{{\rmfamily\textit{#1}}}{#1}}

\newcommand{\cD}{\mathcal{D}}

\newcommand{\cE}{\mathcal{E}}
\newcommand{\cEs}{\cE^{K_s}}
\newcommand{\cEa}{\cE^{K_a}}

\newcommand{\eps}{\varepsilon}

\newcommand{\BIGOP}[1]
{
\mathop{\mathchoice%
{\raise-0.22em\hbox{\huge $#1$}}%
{\raise-0.05em\hbox{\Large $#1$}}{\hbox{\large $#1$}}{#1}}}

\def\Xint#1{\mathchoice
   {\XXint\displaystyle\textstyle{#1}}%
   {\XXint\textstyle\scriptstyle{#1}}%
   {\XXint\scriptstyle\scriptscriptstyle{#1}}%
   {\XXint\scriptscriptstyle\scriptscriptstyle{#1}}%
   \!\int}
\def\XXint#1#2#3{{\setbox0=\hbox{$#1{#2#3}{\int}$}
     \vcenter{\hbox{$#2#3$}}\kern-.5\wd0}}

\def\dashint{\Xint-}
\newcommand{\BIGboxplus}{\mathop{\mathchoice%
{\raise-0.35em\hbox{\huge $\boxplus$}}%
{\raise-0.15em\hbox{\Large $\boxplus$}}{\hbox{\large $\boxplus$}}{\boxplus}}}

\DeclareMathOperator{\diam}{diam}

\DeclareMathOperator{\supp}{supp}

\DeclareMathOperator*{\osc}{osc}

\DeclareMathOperator{\pv}{p.v.}
\newcommand{\U}{\widetilde{u}}

\renewcommand{\d}{\textnormal{d}}

\begin{document}
\allowdisplaybreaks
 \title{Nonlocal operators related to nonsymmetric forms I: {H}\"{o}lder estimates}

\author{Moritz Kassmann}
\author{Marvin Weidner}

\address{Fakult\"{a}t f\"{u}r Mathematik\\Universit\"{a}t Bielefeld\\Postfach 100131\\D-33501 Bielefeld}
\email{moritz.kassmann@uni-bielefeld.de}
\urladdr{https://www.uni-bielefeld.de/math/kassmann}

\address{Fakult\"{a}t f\"{u}r Mathematik\\Universit\"{a}t Bielefeld\\Postfach 100131\\D-33501 Bielefeld}
\email{mweidner@math.uni-bielefeld.de}

\keywords{nonlocal operator, energy form, nonsymmetric, 
regularity, weak Harnack inequality}

\thanks{Moritz Kassmann and Marvin Weidner gratefully acknowledge financial support by the German Research Foundation (SFB 1283 - 317210226 resp. GRK 2235 - 282638148). Moritz Kassmann wholeheartedly thanks the Institut de Math\'{e}matiques de Toulouse for hospitality during the academic year 2020/21.}

\subjclass[2010]{47G20, 35B65, 31B05, 60J75, 35K90}

\allowdisplaybreaks

\begin{abstract}
The aim of this article is to develop the regularity theory for parabolic equations driven by nonlocal operators associated with nonsymmetric forms. H\"older regularity and weak Harnack inequalities are proved using extensions of recently established nonlocal energy methods. We are able to connect the theory of nonsymmetric nonlocal operators with the important results of Aronson-Serrin in the local linear case. This connection is exemplified by nonlocal-to-local convergence results identifying the limiting class of operators as second order differential operators with drift terms.
\end{abstract}

\allowdisplaybreaks

\maketitle

\section{Introduction}  
\label{sec:intro}
The aim of this article is to prove regularity properties for nonlocal operators related to nonsymmetric bilinear forms. Such operators are determined by jumping kernels $K : \R^d \times \R^d \to [0,\infty]$ which might be nonsymmetric. The corresponding operators are of the form
\begin{align}
\label{eq:op}
-L u (x) = 2\pv\int_{\R^d} (u(x)-u(y))K(x,y) \d y.
\end{align}
The operator $L$ is associated with the nonsymmetric bilinear form
\begin{align}\label{eq:energy}
\begin{split}
\cE(u,v) &= 2\int_{\R^d} \int_{\R^d} (u(x)-u(y))v(x) K(x,y)  \d y \d x \\
&= \int_{\R^d} \int_{\R^d}  \big(u(x)-u(y)\big) \big(v(x)-v(y)\big) K_s(x,y) \d y \d x  \\
&\quad + \int_{\R^d} \int_{\R^d}  \big(u(x)-u(y)\big)\big(v(x)+v(y)\big) K_a(x,y)  \d y \d x  \,,
\end{split}
\end{align}
where $K_s$ and $K_a$ are the symmetric respectively the antisymmetric part of $K$. In the last 15 years, a lot of research has been devoted to the symmetric case, i.e., when $K_a=0$. As we explain in this work, treating the antisymmetric part requires a refinement of the proofs of the Caccioppoli estimates that have been developed for the symmetric case. 

In the symmetric case, energy form approaches have been developed in order to establish regularity properties such as H\"older regularity, local boundedness or the validity of weak and full Harnack inequalities for weak solutions to nonlocal elliptic and parabolic equations. Let us comment on this approach and refer to \autoref{subsec:literature} for research in related settings. On the one hand, a Moser iteration scheme for elliptic equations driven by operators related to the fractional Laplacian was developed in \cite{Kas09}. This method has been refined in \cite{DyKa20} and extended to singular anisotropic jumping measures in \cite{ChKa20}, \cite{ChKi21}. A parabolic version was provided in \cite{FeKa13} and further modified in \cite{KaSc14}, \cite{CKW19}. On the other hand, \cite{DKP14} and \cite{DKP16} adopted the De Giorgi method in order to prove regularity estimates and Harnack inequalities for nonlinear operators related to the fractional $p$-Laplacian. Their ideas were extended in \cite{Coz17} and \cite{CKW21}, \cite{CK21b}, \cite{Ok21}, \cite{BOS21} to problems with more general nonlinearities and growth conditions. Parabolic problems were considered, e.g., in \cite{CCV11}, \cite{Str19a}, \cite{Kim20}, \cite{DZZ21}. In contrast to the aforementioned methods, the symmetric case has been successfully studied with the help of the corresponding Markov jump process, too. See \cite{BaLe02b} and \cite{ChKu03} for first results and \cite{CKW20} for further references to this approach.

In this work, we extend some of the aforementioned approaches to the nonsymmetric case. The results can be seen as analogs of the results by Stampacchia \cite{Sta65}, and Aronson-Serrin \cite{ArSe67} in the linear case. The space-time integrability of lower order terms considered in this article is in align with the theory for local equations. One class of examples of nonsymmetric jumping kernels $K$ that we consider is given by 
\begin{align}\label{eq:protex0}
K(x,y) = g(x,y)\vert x-y \vert^{-d-\alpha},
\end{align}
where $g : \R^d \times \R^d \to [0,\infty)$ is a suitable nonsymmetric function and $\alpha \in (0,2)$. We discuss more general examples in \autoref{sec:examples}.

\subsection{Main results}
Let $\Omega \subset \R^d$ be an open set. This article focuses on nonlocal Moser iteration which yields a weak Harnack inequality for nonnegative weak supersolutions and an interior H\"older regularity estimate for weak solutions $u$ to 
\begin{align}
\label{PDE}\tag{$\text{PDE}$}
\partial_t u - L u = f ~~\text{in}~I_R(t_0) \times B_{2R} \subset \R^{d+1},
\end{align}
where $B_{2R} \subset \Omega$ is some ball, $I_R(t_0) = (t_0 - R^{\alpha},t_0 + R^{\alpha}) \subset \R$, $t_0 \in \R$, $0 < R \le 1$, $f \in L^{\infty}(I_R(t_0) \times B_{2R})$ and $L$ is defined as in \eqref{eq:op}.

Now, we state our first main result. Our assumptions will be introduced and explained in the sequel.

\begin{theorem}
\label{thm:mainthmPDE}
Assume that \eqref{K1}, \eqref{K2}, \eqref{cutoff}, \eqref{Poinc} and \eqref{Sob} hold true for some $\alpha \in (0,2)$, $\theta \in [\frac{d}{\alpha},\infty]$.
\vspace{-0.2cm}
\begin{itemize}
\item[(i)] (weak Harnack inequality): Then there is $c > 0$ such that for every $0 < R \le 1$, and every nonnegative,  weak supersolution $u$ to \eqref{PDE} in $I_R(t_0) \times B_{2R}$:
\begin{align}
\label{eq:wHI}
\inf_{\left(t_0 + R^{\alpha} - (\frac{R}{2})^{\alpha}, t_0 + R^{\alpha}\right) \times B_{\frac{R}{2}}} u \ge c \left( \dashint_{\left(t_0 - R^{\alpha}, t_0 - R^{\alpha} + (\frac{R}{2})^{\alpha}\right) \times B_{\frac{R}{2}}} u(t,x) \d x \d t - R^{\alpha} \Vert f \Vert_{L^{\infty}} \right).
\end{align}

\item[(ii)] (H\"older regularity estimate): Assume \eqref{cutoff2}. Then there are $c > 0$ and $\gamma \in (0,1)$ such that for every $0 < R \le 1$ and every weak solution $u$ to \eqref{PDE} in $I_R(t_0) \times B_{2R}$ with $f \equiv 0$:
\begin{align}
\label{eq:HR}
\vert u(t,x) - u(s,y) \vert \le c \Vert u \Vert_{L^{\infty}(I_R(t_0) \times \R^d)} \left( \frac{\vert x-y\vert + \vert t-s \vert^{1/\alpha}}{R}\right)^{\gamma}
\end{align}
for almost every $(t,x),(s,y) \in I_{R/2}(t_0) \times B_{R}$.
\end{itemize}
\end{theorem}

Analogous results are established for weak (super)-solutions $u$ to
 \begin{align}
\label{PDEdual}\tag{$\widehat{\text{PDE}}$}
\partial_t u - \widehat{L} u = f ~~\text{in}~I_R(t_0) \times B_{2R} \subset \R^{d+1},
\end{align}
where $\widehat{L}$ is the dual operator associated with $L$. We refer to \autoref{sec:prelim} for the precise definition of $\widehat{L}$ and the weak solution concept.

\begin{theorem}
\label{thm:mainthmPDEdual}
Assume that \eqref{K1glob}, \eqref{K2}, \eqref{cutoff}, \eqref{Poinc} and \eqref{Sob} hold true for some $\alpha \in (0,2)$, $\theta \in (\frac{d}{\alpha},\infty]$. 
\vspace{-0.2cm}
\begin{itemize}
\item[(i)] Then there is $c > 0$ such that for every $0 < R \le 1$, and every nonnegative, weak supersolution $u$ to \eqref{PDEdual} in $I_R(t_0) \times B_{2R}$, the weak Harnack inequality \eqref{eq:wHI} holds.

\item[(ii)] Assume \eqref{cutoff2dual}. Then there are $c > 0$ and $\gamma \in (0,1)$ such that for every $0 < R \le 1$ and every weak solution $u$ to \eqref{PDEdual} in $I_R(t_0) \times B_{2R}$, the H\"older regularity estimate \eqref{eq:HR} holds.
\end{itemize}
\end{theorem}

\begin{remark*}
	\begin{itemize}
		\item[(i)] The constants in \autoref{thm:mainthmPDE}, \autoref{thm:mainthmPDEdual} only depend on $d,\alpha,\theta$, as well as the constants from \eqref{K1}, \eqref{K1glob}, \eqref{K2}, \eqref{cutoff}, \eqref{Poinc}, \eqref{Sob}, \eqref{cutoff2}, \eqref{cutoff2dual}.
	\item[(ii)] The exclusion of the case $\theta = \frac{d}{\alpha}$ in \autoref{thm:mainthmPDEdual} is in align with the theory of second-order partial differential operators, see Equation (7.2) in \cite{Sta65}.
	\item[(iii)] \autoref{thm:mainthmPDE} and \autoref{thm:mainthmPDEdual} are robust in the sense that the constants do not explode as $\alpha \nearrow 2$, if the constants in \eqref{K1}, \eqref{K2}, \eqref{Poinc}, \eqref{Sob}, \eqref{cutoff} and \eqref{cutoff2} are independent of $\alpha$. 
\end{itemize}
\end{remark*}

\begin{remark*}
\begin{itemize}
\item[(i)] \autoref{thm:mainthmPDE}, \autoref{thm:mainthmPDEdual} remain valid for (super)-solutions $u$ to \eqref{PDE}, resp. \eqref{PDEdual}, where $f \in L^{\infty}(I_R(t_0);L^{\Theta}(B_{2R}))$ for some $\Theta \in ( \frac{d}{\alpha}, \infty)$, with only marginal manipulations in the proofs. We exclude more general source terms in this work.
\item[(ii)] It is possible to extend the aforementioned results to equations on $I_{R}(t_0) \times B_{2R}$ with $R \le \overline{R}$ for any $\overline{R} > 0$ by assuming \eqref{K1}, \eqref{K2}, \eqref{cutoff}, \eqref{Poinc}, \eqref{Sob} for balls with larger radii. However, the constants in \eqref{eq:wHI}, \eqref{eq:HR} would depend on $\overline{R}$ and blow up as $\overline{R} \nearrow \infty$ due to the lack of symmetry.
\item[(iii)] \autoref{thm:mainthmPDE} remains valid if \eqref{K1} is replaced by the assumption that there is $C > 0$ such that for every $0 < r \le 1$ and $B_{2r} \subset \Omega$, $v \in L^2(B_{2r})$:
\begin{align}
\label{eq:Kato}
\left\Vert\frac{|K_a(\cdot,y)|^2}{J(\cdot,y)} \d y \right\Vert_{\mathcal{K}^{d,\alpha}(B_{2r})} \le C, \qquad [v]^2_{H^{\alpha/2}(B_{2r})} + \cE^J_{B_{2r}}(v,v) \le C \cE^{K_s}_{B_{2r}}(v,v).
\end{align}
Here, $\mathcal{K}^{d,\alpha}(B_{2r})$ denotes the nonlocal Kato-class associated with $(-\Delta)^{\alpha/2}$, see \autoref{sec:prelim}.
\end{itemize}
\end{remark*}

As explained above, it is helpful to take into consideration the regularity results for partial differential equations of second order in order to understand the scope of \autoref{thm:mainthmPDE} and \autoref{thm:mainthmPDEdual}. A concrete relation to this theory is given by \autoref{thm:Mosco}, where we study the  limit behavior of the corresponding bilinear forms, as $\alpha \nearrow 2$.

\subsection{Discussion of assumptions}
Let us analyze the effect of a nonsymmetric jumping kernel in more detail in order to motivate our main assumptions and to explain our approach. Throughout this section, it is instructive to think of $K$ being defined as in \eqref{eq:protex0}. Explicit examples of admissible jumping kernels are provided and discussed in \autoref{sec:examples}.\\
We decompose $K = K_s + K_a$ into its symmetric part $K_s$ and its antisymmetric part $K_a$, where
\begin{align*}
K_s(x,y) = \frac{K(x,y)+K(y,x)}{2}, ~~ K_a(x,y) = \frac{K(x,y)-K(y,x)}{2}, ~~x,y \in \R^d.
\end{align*}
Note that by construction $K_s(x,y) = K_s(y,x)$ and $K_a(x,y) = -K_a(y,x)$. Since $K(x,y) \ge 0$ it holds that $K_a$ is dominated by $K_s$ in the following sense:
\begin{align}
\label{eq:KaKs}
\vert K_a(x,y) \vert \le K_s(x,y).
\end{align}

By making use of the decomposition of $K$ and the symmetry (resp. antisymmetry) of $K_s$ and $K_a$, we can rewrite $\cE$ as follows:
\begin{align*}
\begin{split}
\cE(u,v) &= 2\int_{\R^d} \int_{\R^d} (u(x)-u(y))v(x) K_s(x,y) \d y \d x +2\int_{\R^d} \int_{\R^d} (u(x)-u(y))v(x) K_a(x,y) \d y \d x\\
&=: \cEs(u,v) + \cEa(u,v),
\end{split}
\end{align*}
where
\begin{align*}
&\cEs(u,v) = \int_{\R^d} \int_{\R^d} (u(x)-u(y))(v(x)-v(y)) K_s(x,y) \d y \d x,\\
&\cEa(u,v) = \int_{\R^d} \int_{\R^d} (u(x)-u(y))(v(x)+v(y)) K_a(x,y) \d y \d x.
\end{align*}

The novelty caused by the absence of symmetry therefore lies in the existence of the second summand $\cEa$ which is of different shape compared to $\cEs$. Such decomposition of the jumping kernel is standard in the literature concerned with nonsymmetric energy forms (see \cite{FuUe12}, \cite{ScWa15}, \cite{FKV15}).
However, note that for any possibly nonsymmetric kernel $K$ it holds:
\begin{align}
\int_{\R^d} \int_{\R^d} (u(x)-u(y))(v(x)-v(y)) K(x,y) \d y \d x = \cEs(u,v).
\end{align}

There is a vast amount of research concerning symmetric energy forms of type $\cEs$ and corresponding function spaces. Throughout this article we will assume that $K_s$ satisfies the L\'evy-integrability condition
\begin{align}
\label{eq:Levy}
\left(x \mapsto \int_{\R^d} \left(\vert x-y \vert^2 \wedge 1 \right) K_s(x,y) \d y  \right) \in L^1_{loc}(\R^d).
\end{align}
In order to control the nonsymmetric part $\cE^{K_a}$, we need to impose suitable conditions on the jumping kernel $K$. To motivate our main assumption, we recall the following result:

\begin{proposition}[\cite{ScWa15} Theorem 1.1]
\label{prop:SchillingWang}
Assume that $K_s$ satisfies \eqref{eq:Levy} and that
\begin{align}
\label{eq:K1glob}
\sup_{x \in \R^d} \int_{\R^d} \frac{\vert K_a(x,y)\vert^2}{K_s(x,y)} \d y < \infty.
\end{align}
Then $\cE(u,v)$ is well-defined for $u,v \in C^{lip}_c(\R^d)$ and there exists a domain $\mathcal{F} \subset L^2(\R^d)$ such that the pair $(\cE,\mathcal{F})$ is a regular lower bounded semi-Dirichlet form on $L^2(\R^d)$.
\end{proposition}

Note that although it is desirable from a conceptual point of view, we do not necessarily require the energy form $\cE$ to be a regular lower bounded semi-Dirichlet form in order to give a reasonable definition of a weak solution or to prove regularity estimates. However, it is crucial for our purposes that $\cE(u,v)$ is well-defined, provided that $\cEs(u,v) < \infty$, see \autoref{lemma:welldef}. 

We introduce the following notation, given a set $M \subset \R^d \times \R^d$, and a jumping kernel $K$:
\begin{align*}
\cE_M^K(u,v) := \iint_{M} (u(x)-u(y))v(x) K(x,y) \d x \d y.
\end{align*}
Analogously, we define $\cEs_M, \cEa_M$. If $M := B_r \times B_r$ for a ball $B_r \subset \R^d$, we write $\cE^K_{B_r} = \cE^K_{B_r \times B_r}$.

\pagebreak[3]
Let $\alpha \in (0,2)$. The following condition will be our main assumption on the nonsymmetric part of the jumping kernel. It is a considerably more general assumption compared to \eqref{eq:K1glob}.
\begin{assumption*}[\textbf{K1}]
\label{ass:K1}
Let $J : \R^d \times \R^d \to [0,\infty]$ be a symmetric jumping kernel and $\theta \in [\frac{d}{\alpha},\infty]$.
\begin{itemize}
\item $K$ satisfies \eqref{K1} if there is $C > 0$ such that for every ball $B_{2r} \subset \Omega$ with $r \le 1$:
\begin{align}
\label{K1}\tag{$\text{K1}_{loc}$}
\left\Vert \int_{B_{2r}} \frac{\vert K_a(\cdot,y)\vert^2}{J(\cdot,y)} \d y  \right\Vert_{L^{\theta}(B_{2r})} \le C, \qquad \cE_{B_{2r}}^{J}(v,v) \le C \cE_{B_{2r}}^{K_s}(v,v), ~~ \forall v \in L^2(B_{2r}).
\end{align}
\item $K$ satisfies \eqref{K1glob} if $J$ satisfies \eqref{cutoff} and there is $C > 0$ such that for every ball $B_{2r} \subset \Omega$ with $r \le 1$:
\begin{align}
\label{K1glob}\tag{$\text{K1}_{glob}$}
\left\Vert \int_{\R^d} \frac{\vert K_a(\cdot,y)\vert^2}{J(\cdot,y)} \d y  \right\Vert_{L^{\theta}(\R^d)} \le C, \qquad \cE_{B_{2r}}^{J}(v,v) \le C \cE_{B_{2r}}^{K_s}(v,v), ~~ \forall v \in L^2(B_{2r}).
\end{align}
\end{itemize}
\end{assumption*}

We will see later that in the simplest case, \eqref{K1} and \eqref{K1glob} hold true with $J = K_s$. However, allowing for general symmetric kernels $J$ gives rise to a significantly larger class of operators, as we will discuss in \autoref{sec:examples}.

To get a feeling for \eqref{K1}, it is important to observe that $K_a$ must possess a singularity on the diagonal that is of lower order compared to $J$. Indeed, if $J(x,y) \asymp \vert x-y \vert^{-d-\alpha}$, it must be that $K_a(x,y) \asymp \vert x-y \vert^{-d-\beta}$ for some $0 < \beta < \alpha/2$ close to the diagonal in order for $y \mapsto \frac{K_a^2(x,y)}{J(x,y)}$ to be integrable close to $x$. Consequently, we regard $\cEa$ as a lower order term, dominated by $\cEs$. In this respect, allowing for nonsymmetric jumping kernels in the framework of regularity theory - with a well-behaved symmetric part and an antisymmetric part satisfying an integrability assumption as \eqref{K1} - can be considered as an extension of the existing nonlocal theory to nonlocal diffusion operators that are perturbed by nonlocal lower order terms.

\begin{example}
\label{ex:intro}
Let $K$ be as in \eqref{eq:protex0}, i.e. $K(x,y) = g(x,y)\vert x-y\vert^{-d-\alpha}$. We see that
\begin{align*}
2K_a(x,y) = (g(x,y)-g(y,x))\vert x-y \vert^{-d-\alpha}, ~~2 K_s(x,y) = (g(x,y)+g(y,x)) \vert x-y \vert^{-d-\alpha}.
\end{align*}
Let us discuss assumption \eqref{K1} for the particular choice
\begin{align*}
g(x,y) = 1 + (V(x)-V(y)), \qquad \text{where } V : \R^d \to \R, ~~\vert V(x)-V(y)\vert \le 1,~~\forall x,y \in \R^d.
\end{align*}
In that case $K_s(x,y) \asymp \vert x-y \vert^{-d-\alpha}$ and
\begin{align*}
K_a(x,y) = (V(x)-V(y))\vert x-y \vert^{-d-\alpha} \asymp \vert x-y \vert^{-d-\beta}, ~~ \text{if } V \in C^{0,\alpha - \beta}(\R^d) \,.
\end{align*}
\eqref{K1} holds true with $J = K_s$ if $\beta \le \alpha/2$. In this respect, $\cEa$ becomes a lower order term through the prescription of regularity for $V$. We observe that the linear operator $L$ corresponding to $K$ is given by
\begin{align}
\label{eq:protexop}
-L u(x) = (-\Delta)^{\alpha/2}u(x) + \Gamma^{(\alpha)}(u,V)(x),
\end{align}
where $\Gamma^{(\alpha)}(u,V)$ denotes the nonlocal carr\'e du champ defined by
\begin{align*}
\Gamma^{(\alpha)}(u,V)(x) = \int_{\R^d} (u(x)-u(y))(V(x)-V(y))\vert x-y \vert^{-d-\alpha} \d y.
\end{align*}
By comparison with its local counterpart $\Gamma^{(2)}(u,V)(x) = \left(\nabla u(x) , \nabla V(x)\right)$, one can consider the operator \eqref{eq:protexop} as a model example of a nonlocal diffusion operator with a nonlocal drift term. Let us make a related observation that extends a well-known fact for local operators. For a positive function $\rho$, the operator 
\begin{align} \label{eq:frac-diff}
(-\Delta)^{\alpha/2}u(x) - \frac{1}{\rho}\Gamma^{(\alpha)}(u,\rho)(x),
\end{align}
is a symmetric operator in $L^2(\R^d; \rho(x) \d x)$ for $0 < \alpha \leq 2$. 
\end{example}

As \autoref{ex:intro} indicates, the regularity results in this work can be seen as a nonlocal extension of the De Giorgi-Nash-Moser-theory for second order divergence form operators with a drift:
\begin{align}
\label{eq:localop}
\mathcal{L} u = \partial_i(a_{i,j} \partial_j u)+ b_i\partial_i u, \qquad \text{resp. } \widehat{\mathcal{L}} u = \partial_i(a_{i,j} \partial_j u - b_i u),
\end{align}
see, e.g., \cite{Sta65}, \cite{ArSe67}, \cite{GiTr01}. In \autoref{sec:approx}, we justify this viewpoint by providing an approximation result for $\alpha \nearrow 2$ (see \autoref{thm:Mosco}). We refer to \autoref{sec:examples}, where \autoref{ex:intro} is extended and further assumptions on the coefficient $g$ with respect to \eqref{K1} are discussed.

\begin{remark*}
\begin{enumerate}
\item[(i)] Clearly, when \eqref{K1} is satisfied for some $\theta$, it also holds true for every $\theta' \le \theta$.
\item[(ii)] The restriction to $\theta \ge d/\alpha$ in \eqref{K1} is natural in the light of the classical fractional Sobolev embedding, see \eqref{Sob} and \autoref{lemma:welldef}.
\item[(iii)] Weak Harnack inequalities and H\"older estimates were proved for weak (super)-solutions to
\begin{align*}
\partial_t u - \mathcal{L} u = f, ~~ \text{ in } I_R(t_0) \times B_{2R}
\end{align*}
where, $\mathcal{L}$ is as in \eqref{eq:localop}, respectively the corresponding elliptic equation, using De Giorgi-Nash-Moser theory  (see, e.g., \cite{Sta65}, \cite{ArSe67}, \cite{GiTr01}) under the condition $|b|^2 \in L^{\theta}(\Omega)$ for some $\theta \in [\frac{d}{2},\infty]$ and $a_{i,j}$ bounded, uniformly elliptic. For weak (super)-solutions to
\begin{align*}
\partial_t u - \widehat{\mathcal{L}} u = f, ~~ \text{ in } I_R(t_0) \times B_{2R},
\end{align*}
corresponding estimates were proved for $|b|^2 \in L^{\theta}(\Omega)$ with $\theta \in (\frac{d}{2},\infty]$.\\ 
In this sense, the range of $\theta$ in this work is in align with the theory of second-order operators in divergence form.
\end{enumerate}
\end{remark*}

In addition to \eqref{K1}, respectively \eqref{K1glob}, we impose the following assumption on $K$:

\begin{assumption*}[\textbf{K2}]
\label{ass:K2}
There exist $C > 0$, $D < 1$ and a symmetric jumping kernel $j : \R^d \times \R^d \to [0,\infty]$ such that for every ball $B_{2r} \subset \Omega$ with $r \le 1$ and every $v \in L^2(B_{2r})$ with $\cE_{B_{2r}}^{K_s}(v,v) < \infty$:
\begin{align*}
\label{K2}\tag{$\text{K2}$}
K(x,y) \ge (1-D)j(x,y), ~~ \forall x,y \in B_{2r}, \qquad \cE_{B_{2r}}^{K_s}(v,v) &\le C \cE_{B_{2r}}^{j}(v,v).
\end{align*}
\end{assumption*}

We consider \eqref{K2} to be an ellipticity assumption on $K$ since it ensures that the symmetric kernel $K_s - |K_a|$ is locally coercive with respect to $\cEs$. In fact, it turns out that if $K_s(x,y) \gtrsim |x-y|^{-d-\alpha}$, assumption \eqref{K1} is already sufficient for \eqref{K2} to hold true on balls $B_{2r} \subset 5^{-n}\Omega$ for some $n \in \N$, see \autoref{prop:K1impliesK2} and the following discussion.

\begin{remark*}
\begin{itemize}
\item[(i)] Variants of assumptions \eqref{K1}, \eqref{K1glob}, \eqref{K2} have already appeared in the literature, see \cite{ScWa15}, \cite{FKV15}, \cite{DaTo20}.
\item[(ii)] \eqref{K1}, \eqref{K2} are localized properties since only $x,y \in \Omega$ are taken into account. This turns out to be sufficient for proving interior regularity estimates for $L$.
\end{itemize}
\end{remark*}

The following three assumptions only affect the symmetric part $K_s$. They are standard in the regularity theory of nonlocal operators related to energy forms, appear in different variations (see \cite{DyKa20}, \cite{FeKa13} and the comments below) and include for example $K_s(x,y) \asymp \vert x-y\vert^{-d-\alpha}$ but also more general kernels, see \autoref{sec:examples}. Let $\alpha \in (0,2)$, as before:

\begin{assumption*}[\textbf{Cutoff}]
\label{ass:Cutoff}
There is $c > 0$ such that for every $0 < \rho \le r \le 1$,  $z \in \Omega$ such that $B_{r+\rho}(z) \subset \Omega$ there is a radially decreasing function $\tau = \tau_{z,r,\rho}$ centered at $z \in \R^d$ with $\supp(\tau) \subset \overline{B_{r+\rho}(z)}$, $0 \le \tau \le 1$, $\tau \equiv 1$ on $B_r(z)$ and $\vert \nabla \tau \vert \le 3\rho^{-1}$:
\begin{align}
\label{cutoff}\tag{$\text{Cutoff}$}
\sup_{x \in B_{r+\rho}(z)} \Gamma^{K_s}(\tau,\tau)(x) \le c \rho^{-\alpha},
\end{align}
where $\Gamma^{K_s}(\tau,\tau)(x) := \int_{\R^d} (\tau(x)-\tau(y))^2 K_s(x,y) \d y$ is the carr\'e du champ associated with $\cEs$.
\end{assumption*}

\begin{assumption*}[\textbf{Poinc}]
\label{ass:Poinc}
There is $c > 0$ such that for every ball $B_r \subset \Omega$ with $0 < r \le 1$ and $v \in L^2(B_r)$:
\begin{align}
\label{Poinc}\tag{$\text{Poinc}$}
\int_{B_r} \left(v(x) - [v]_{B_r}\right)^2 \d x \le c r^{\alpha} \cEs_{B_r}(v,v),
\end{align}
where $[v]_{B_r} = \dashint_{B_r} v(x) \d x$.
\end{assumption*}

\begin{assumption*}[\textbf{Sob}]
\label{ass:Sob}
There is $c > 0$ such that for every ball $B_{r+\rho} \subset \Omega$ with $0 < \rho \le r \le 1$ and $v \in L^2(B_{r+\rho})$:
\begin{align}
\label{Sob}\tag{$\text{Sob}$}
\Vert v^2 \Vert_{L^{\frac{d.}{d-\alpha}}(B_r)} \le c\cEs_{B_{r+\rho}}(v,v) + c \rho^{-\alpha}\Vert v^2\Vert_{L^{1}(B_{r+\rho})}.
\end{align}
\end{assumption*}

We collect a few comments on the aforementioned assumptions.

\begin{remark*}
\begin{itemize}
\item[(i)] \eqref{cutoff} guarantees the existence of suitable cutoff functions $\tau$. A sufficient condition for \eqref{cutoff} to hold true is (see \cite{CKW19}): There is $c > 0$ such that for every $0 < \zeta \le \rho \le r \le 1$, $z\in \R^d$ with $B_{r+\rho}(z) \subset \Omega$:
\begin{align}
\label{eq:suffcutoff}
\sup_{x \in B_{r+\rho}(z)}\left( \int_{\R^d \setminus B_{\zeta}(x)} K_s(x,y) \d y \right) \le c \zeta^{-\alpha}.
\end{align}
\item[(ii)] Instead of \eqref{Sob} it is equally sufficient for our purposes to assume a global Sobolev inequality: There exists $c > 0$ such that for every $v \in L^{\frac{2d}{d-\alpha}}(\R^d)$:
\begin{align}
\label{eq:globalSob}
\Vert v^2\Vert_{L^{\frac{d}{d-\alpha}}(\R^d)} \le c\cEs(v,v).
\end{align}
\item[(iii)] \eqref{eq:globalSob} and \eqref{cutoff} together already imply the localized Sobolev inequality \eqref{Sob}.
\item[(iv)] \eqref{Sob} and \eqref{Poinc} both follow if one assumes coercivity of $\cE^{K_s}$ on small scales, i.e., that there is $c > 0$ such that for every ball $B_r \subset \R^d$, $0 < r \le 1$:
\begin{align}
\label{eq:coercivity}
\cE^{K_s}_{B_{r}}(v,v) \ge c [v]^{2}_{H^{\alpha/2}(B_r)}, ~~ \forall v \in L^2(B_{r}).
\end{align}
A sufficient condition for \eqref{eq:coercivity} is given in \cite{ChSi20}.
\end{itemize}
\end{remark*}

The final assumption is a condition on the decay of the jumping kernel $K$ at infinity.

\begin{assumption*}[$\mathbf{\infty}$\textbf{-Tail}]
\label{ass:cutoff2}
There are $c,\sigma > 0$ such that for every ball $B_{2r} \subset \Omega$, $0 < r \le 1$ and every $A > 1$ with $A r \ge 1$:
\begin{align}
\label{cutoff2}\tag{$\infty\text{-Tail}$}
\sup_{x \in B_{2r}}\left( \int_{\R^d \setminus B_{A r}(x)} K(x,y) \d y \right) &\le c (Ar)^{-\sigma},\\
\label{cutoff2dual}\tag{$\widehat{\infty\text{-Tail}}$}
\sup_{x \in B_{2r}}\left( \int_{\R^d \setminus B_{A r}(x)} K(y,x) \d y \right) &\le c (Ar)^{-\sigma}.
\end{align}
\end{assumption*}

\begin{remark*}
\begin{itemize}
\item[(i)] \eqref{cutoff2}, \eqref{cutoff2dual} hold with $\sigma = \alpha$ if \eqref{eq:suffcutoff} is satisfied for every $\zeta > 0$. These assumptions appear only in the proofs of \autoref{lemma:ioinf}, \autoref{lemma:ioinfdual} and are reminiscent of condition $(1-16)$ in \cite{DyKa20}. 
\item[(ii)] It is important to allow for $\sigma < \alpha$ in \eqref{cutoff2} and \eqref{cutoff2dual} since the operators under consideration might have nonlocal drifts of lower order, see \autoref{prop:nldriftoncones}.
\end{itemize}

\end{remark*}

\subsection{Time-dependent jumping kernels}
\label{sec:introtime}

Let $I \subset \R$ be an open interval. In this section, we discuss how the main results \autoref{thm:mainthmPDE} and \autoref{thm:mainthmPDEdual} can be extended to nonsymmetric nonlocal operators of type \eqref{eq:op} with time-inhomogeneous jumping kernels $k : I \times \R^d \times \R^d$. The proof of the main result in this context, see \autoref{thm:mainthm-time}, is provided in \autoref{sec:time}.\\
Given $k$, we define the associated nonlocal operator and bilinear form via
\begin{align}
\label{eq:op-time}
-L_t u (t,x) &= 2\pv\int_{\R^d} (u(t,x)-u(t,y))k(t;x,y)\d y,\\
\label{eq:energy-time}
\cE^{k(t)}_B(u,v) &= \int_{B} \int_{B} (u(x) - u(y))v(x) k(t;x,y) \d y \d x, ~~ B \subset \R^d.
\end{align}
Moreover, we define $\widehat{L_t}$ and $\widehat{\cE}^{k(t)}$ as the corresponding dual operator and bilinear form.
As before, we decompose $k(t) = k_s(t) + k_a(t)$ into its symmetric and antisymmetric part.\\
A common assumption on the symmetric part $k_s$ is that there exist a symmetric jumping kernel $K_s : \R^d \times \R^d \to [0,\infty]$, and a measurable function $a : I \times \R^d \times \R^d \to [\lambda,\Lambda]$, given $0 < \lambda \le \lambda$, which satisfies $a(t;x,y) = a(t;y,x)$ for every $t \in I$, $x,y \in \R^d$, such that
\begin{align*}
k_s(t;x,y) := a(t,x,y) K_s(x,y).
\end{align*}
Note that the specific structure of $k_s$ is not restrictive, since it is equivalent to
\begin{align}
\label{eq:symmetric-time-dependence}\tag{$k_s \asymp$}
\lambda K_s(x,y) \le k_s(t;x,y) \le \Lambda K_s(x,y), ~~ t \in I,~ x,y \in \R^d,
\end{align}
upon defining $a(t;x,y) :=k_s(t;x,y)/K_s(x,y)$.

In fact, H\"older regularity estimates, as well as weak Harnack inequalities as in \autoref{thm:mainthmPDE}, can be established for symmetric time-dependent jumping kernels $k_s$ satisfying \eqref{eq:symmetric-time-dependence} via the same proof, up to some straightforward modifications, see \cite{FeKa13}.\\
By contrast, the case of time-dependent jumping kernels is of particular interest due to the lack of symmetry.  In fact, assuming boundedness in $t$ of the constant from \eqref{K1} is not necessary and would not be in align with the theory for local operators, see \cite{ArSe67}.\\
Given $\theta,\mu \in (1,\infty]$, we will work under the following assumption, which can be interpreted as a time-inhomogeneous analog to \eqref{K1}:

\begin{assumption*}[$\textbf{K1}^{\mathbf{t}}$]
\label{ass:K1time}
Let $J : \R^d \times \R^d \to [0,\infty]$ be a symmetric jumping kernel.
\begin{itemize}
\item Assume that there is $C > 0$ such that for every ball $B_{2r} \subset \Omega$ and any interval $I_r \subset I$ with $r \le 1$:
\begin{align}
\label{eq:K1time}\tag{$\text{K1}_{loc}^t$}
\left\Vert \int_{B_{2r}} \frac{\vert k_a(\cdot;\cdot,y)\vert^2}{J(\cdot,y)} \d y  \right\Vert_{L^{\mu,\theta}_{t,x}(I_r \times B_{2r})} \le C, \qquad \cE_{B_{2r}}^{J}(v,v) \le C \cE_{B_{2r}}^{K_s}(v,v), ~~ \forall v \in L^2(B_{2r}).
\end{align}
\item Assume that there is $C > 0$ such that for every ball $B_{2r} \subset \Omega$ and any interval $I_r \subset I$ with $r \le 1$:
\begin{align}
\label{eq:K1globtime}\tag{$\text{K1}_{glob}^t$}
\left\Vert \int_{\R^d} \frac{\vert k_a(\cdot;\cdot,y)\vert^2}{J(\cdot,y)} \d y  \right\Vert_{L^{\mu,\theta}_{t,x}(I_r \times \R^d)} \le C, \qquad \cE_{B_{2r}}^{J}(v,v) \le C \cE_{B_{2r}}^{K_s}(v,v), ~~ \forall v \in L^2(B_{2r}).
\end{align}
\end{itemize}
\end{assumption*}

Here, $\Vert v \Vert_{L^{\mu,\theta}_{t,x}(I \times B)} := \Vert v \Vert_{L^{\mu}(I;L^{\theta}(B))}$. Moreover, we introduce the following compatibility condition for $\mu,\theta$, which is in align with (3) in \cite{ArSe67} in the linear case, see also \cite{LSU68}:
\begin{align}
\label{eq:compatibility}\tag{CP}
\frac{d}{\alpha\theta} + \frac{1}{\mu} &\le 1 , ~~\theta \in (d/\alpha,\infty],\\
\label{eq:compatibilitydual}\tag{$\widehat{\text{CP}}$}
\frac{d}{\alpha\theta} + \frac{1}{\mu} &< 1.
\end{align}

The following condition is an analog to \eqref{K2}:

\begin{assumption*}[$\textbf{K2}^{\mathbf{t}}$]
\label{ass:K2time}
There exist $C > 0$, $D < 1$ and a symmetric jumping kernel $j$ such that for every ball $B_{2r} \subset \Omega$ and every interval $I_r \subset I$ with $r \le 1$ and every $v \in L^2(B_{2r})$ with $\cE_{B_{2r}}^{K_s}(v,v) < \infty$:
\begin{align}
\label{eq:K2time}\tag{$\text{K2}^t$}
k(t;x,y) \ge (1-D)j(x,y), ~~ \forall t \in I_r, x,y \in B_{2r}, \qquad \cE_{B_{2r}}^{K_s}(v,v) &\le C \cE_{B_{2r}}^{j}(v,v).
\end{align}
\end{assumption*}

Finally, we introduce the following time-dependent analog to \eqref{cutoff2} and \eqref{cutoff2dual}:

\begin{assumption*}[$\mathbf{\infty}\textbf{-Tail}^{\mathbf{t}}$]
\label{ass:cutoff2time}
There are $c,\sigma > 0$ such that for every ball $B_{2r} \subset \Omega$ and any interval $I_r \subset I$ with $0 < r \le 1$ and every $A > 1$ with $A r \ge 1$:
\begin{align}
\label{cutoff2time}\tag{$\infty\text{-Tail}^t$}
\sup_{(t,x) \in I_r \times B_{2r}}\left( \int_{\R^d \setminus B_{A r}(x)} k(t;x,y) \d y \right) &\le c (Ar)^{-\sigma},\\
\label{cutoff2dualtime}\tag{$\widehat{\infty\text{-Tail}^t}$}
\sup_{(t,x) \in I_r \times B_{2r}}\left( \int_{\R^d \setminus B_{A r}(x)} k(t;y,x) \d y \right) &\le c (Ar)^{-\sigma}.
\end{align}
\end{assumption*}

Now, we can formulate counterparts of \autoref{thm:mainthmPDE} and \autoref{thm:mainthmPDEdual} for time-inhomogeneous jumping kernels. Note that in the time-inhomogeneous context, \eqref{PDE} and \eqref{PDEdual} have to be understood with $L$ replaced by $L_t$, respectively $\widehat{L}$ replaced by $\widehat{L_t}$.

\begin{theorem}
\label{thm:mainthm-time}
Assume that there is a symmetric jumping kernel $K_s$ such that \eqref{eq:symmetric-time-dependence} holds true and $K_s$ satisfies \eqref{cutoff}, \eqref{Poinc} and \eqref{Sob} for some $\alpha \in (0,2)$. Moreover, assume \eqref{eq:K2time}.
\begin{itemize}
\item[(i)] Assume \eqref{eq:K1time} holds for $(\mu,\theta)$ satisfying \eqref{eq:compatibility}. Then, any nonnegative, weak supersolution $u$ to \eqref{PDE} in $I_R(t_0) \times B_{2R}$ satisfies the weak Harnack inequality \eqref{eq:wHI}.\\
Moreover, if $k$ satisfies \eqref{cutoff2time}, then any weak solution $u$ to \eqref{PDE} in $I_R(t_0) \times B_{2R}$ with $f \equiv 0$ satisfies the H\"older regularity estimate \eqref{eq:HR}.
\item[(ii)] Assume \eqref{eq:K1globtime} holds for $(\mu,\theta)$ satisfying \eqref{eq:compatibilitydual}. Then, any nonnegative, weak supersolution $u$ to \eqref{PDEdual} in $I_R(t_0) \times B_{2R}$ satisfies the weak Harnack inequality \eqref{eq:wHI}.\\
Moreover, if $k$ satisfies \eqref{cutoff2dualtime}, then any weak solution $u$ to \eqref{PDEdual} in $I_R(t_0) \times B_{2R}$ with $f \equiv 0$ satisfies the H\"older regularity estimate \eqref{eq:HR}.
\end{itemize}

\end{theorem}

\begin{remark*}
These results are in align with those of Aronson and Serrin in the linear case, see \cite{ArSe67} and Ladyzhenskaya-Solonnikov-Ural'tceva, see \cite{LSU68}. Note that even in the local case, it is still unknown, whether the H\"older estimate holds true in the limit case, when \eqref{eq:compatibility} holds true with equality, e.g, if $\theta = \frac{d}{\alpha}$ and $\mu = \infty$. 
\end{remark*}

\pagebreak[1]
\subsection{Strategy of proof}
\label{sec:introstrategy}
The main challenge in this article are adequate Caccioppoli-type estimates, which take into account the special structure of nonsymmetric energy forms.
Let us motivate the idea behind our approach by recalling how to establish the following Caccioppoli-type estimate in the elliptic H\"older regularity program for second order operators with a drift term:
\begin{align}
\label{eq:lochelpgoal}
\int_{B_{r+\rho}} \tau^2(x)\vert \nabla \log u (x)\vert^2 \d x \le c \int_{B_{r+\rho}} (\tau^2(x) + \vert \nabla \tau (x)\vert^2) \d x.
\end{align}
Here $\tau \in C_c^1(\R^d)$ is a cutoff function between two balls $B_r$ and $B_{r+\rho} \subset B_{2r}$, and $u \ge 0$ solves 
\begin{align}
\label{eq:lochelp1}
\cD(u,\phi) := \int_{\R^d} (a(x)\nabla u(x) , \nabla \phi(x)) \d x + \int_{\R^d} (b(x) ,\nabla u (x))\phi(x) \d x \le 0, ~~\forall \phi \in H^1_0(B_{2r}),
\end{align}
with a uniformly positive definite and bounded diffusion matrix $a$ and a bounded drift coefficient $b$. In fact, \eqref{eq:lochelpgoal} implies an energy estimate for $\log u$ wherefrom it can be deduced that $u$ is of bounded mean oscillation, allowing to connect negative and positive exponents via the John-Nirenberg lemma (see \cite{GiTr01}). Estimate \eqref{eq:lochelpgoal} is proved by testing $\eqref{eq:lochelp1}$ with $\phi = -\tau^2 u^{-1}$, where $u \ge \eps$ for some $\eps > 0$. Then, by product rule and the assumptions on $a$:
\begin{align*}
c_1 \int \tau^2(x) \vert \nabla \log u (x)\vert^2 \d x \le \cD^{a}(u,-\tau^2 u^{-1}) +  c_2 \int (\nabla \log u(x) , \nabla \tau(x))\tau(x) \d x,
\end{align*}
where we set $\cD^{a}(u,\phi) = \int (a\nabla u,\nabla \phi)$. By Young's inequality:
\begin{align}
\label{eq:lochelp2}
c_3\int \tau^2(x) \vert \nabla \log u (x)\vert^2 \d x \le \cD^{a}(u,-\tau^2 u^{-1}) + c_4\int \vert\nabla \tau(x)\vert^2\d x.
\end{align}
In the case $b \equiv 0$, i.e. without lower order term, \eqref{eq:lochelp2}  already implies \eqref{eq:lochelpgoal} since $u$ satisfies \eqref{eq:lochelp1}. Otherwise, one can estimate the drift contribution $\cD^{b}(u,\phi) = \int (b,\nabla u)\phi$ as follows
\begin{align}
\label{eq:lochelp3}
- \cD^{b}(u,-\tau^2 u^{-1}) \le \frac{c_3}{2} \int \tau^2(x) \vert \nabla \log u (x)\vert^2 \d x + c_5 \int \tau^2(x) \d x,
\end{align}
and deduce \eqref{eq:lochelpgoal} by combination of \eqref{eq:lochelp2}, \eqref{eq:lochelp3} and absorption of $\frac{c_3}{2} \int \tau^2(x) \vert \nabla \log u (x)\vert^2 \d x$.
Due to the locality of the underlying equation, \eqref{eq:lochelp3} is a simple consequence of the chain rule and Young's inequality, since $b$ is bounded.\\
A key problem in the nonsymmetric nonlocal case is to find nonlocal counterparts of the chain rule, which allow to absorb the nonsymmetric nonlocal contributions into the weighted symmetric energy of $\log u$. 
A nonlocal analog to \eqref{eq:lochelp3} is an estimate of the form (see \eqref{eq:loguJa}):
\begin{align}
\label{eq:locnl1}
\begin{split}
&\iint_{B_{r+\rho}^2} (u(y) - u(x))(u^{-1}(x) \wedge u^{-1}(y))(\tau^2(x) \wedge \tau^2(y)) K_a(x,y) \d y \d x\\
&\quad\le \frac{1}{2} \iint_{B_{r+\rho}^2} (\log u(x) - \log u(y))^2 (\tau^2(x) \wedge \tau^2(y)) K_s(x,y) \d y \d x + c \int_{B_{r+\rho}}\tau^2(x)\d x,
\end{split}
\end{align}
which can be proved with the help of \eqref{K1}. This estimate matches the following well-known inequality for symmetric forms (see \eqref{eq:loguIsIa}, \eqref{eq:loguJs}), which is a nonlocal version of \eqref{eq:lochelp2}:
\begin{align}
\label{eq:locnl2}
\begin{split}
\iint_{B_{r+\rho}^2} & (\log u(x) - \log u(y))^2 (\tau^2(x) \wedge \tau^2(y)) K_s(x,y) \d y \d x\\
&\le \cEs_{B_{r+\rho}}(u,-\tau^2 u^{-1}) + c\int_{ B_{r+\rho}}\Gamma^{(\alpha)}(\tau,\tau)(x)\d x.
\end{split}
\end{align}
The integral on the left hand side is a nonlocal analog to $\int \tau^2 \vert \nabla \log u\vert^2$. When $u$ satisfies $\cE(u,\phi) \le 0$, the estimates \eqref{eq:locnl1}, \eqref{eq:locnl2} can be combined as in the local case. However,  since the quantity on the left hand side of \eqref{eq:locnl1} is not equal to $\cE^{K_a}_{B_{r+\rho}}(u,-\tau^2 u^{-1})$, an additional step is required in order to obtain
\begin{align*}
\iint_{B_{r+\rho}^2}(\log u(x) - \log u(y))^2 (\tau^2(x) \wedge \tau^2(y)) K_s(x,y) \d y \d x \le c\int_{ B_{r+\rho}}(\tau^2(x) + \Gamma^{(\alpha)}(\tau,\tau)(x))\d x,
\end{align*}
which is a nonlocal counterpart of \eqref{eq:lochelpgoal}. The above arguments are presented in full detail in the proof of \autoref{lemma:MSgen2}.\\
Due to the lack of the chain rule, establishing Caccioppoli-type estimates for nonsymmetric kernels requires a careful treatment of the terms involved. In particular, an optimization of the existing methods for symmetric kernels (see \cite{Kas09}, \cite{FeKa13}, \cite{DyKa20}) is needed in order to make possible the absorption of all nonsymmetric terms.

Another difficulty comes from the consideration of the dual energy form $\widehat{\cE}$. In contrast to \eqref{PDE}, the family of supersolutions to \eqref{PDEdual} is not invariant under addition of constants. This property is at the core of the proof of oscillation decay, which yields \autoref{thm:mainthmPDEdual} (ii). In order to enforce such invariance, we introduce the generalized dual equation \eqref{PDEdualext} and prove a weak parabolic Harnack inequality for supersolutions to \eqref{PDEdualext}, see \autoref{thm:wHIdualext}. This requires proving Caccioppoli-type estimates also for the dual form $\widehat{\cE}$. Note that \autoref{thm:mainthmPDEdual} (i) can be proved without considering \eqref{PDEdualext}.

\subsection{Related literature}\label{subsec:literature}
As we have explained above, regularity properties of solutions to nonlocal equations governed by \emph{nonsymmetric} bilinear forms have not yet been studied systematically. H\"older regularity estimates are proved in \cite{ImSi20} for a large class of kinetic integro-differential equations using an adaptation of De Giorgi's method. These equations include nonlocal operators of type \eqref{eq:op}, \eqref{eq:protex0} with nonsymmetric kernels whose antisymmetric part satisfies certain cancellation conditions:
\begin{align*}
\sup_{x \in \R^d} \left\vert \pv \int_{B_r(x)}K_a(x,y) \d y \right\vert \le c , \quad \sup_{x \in \R^d} \left\vert \pv \int_{B_r(x)}K_a(x,y)(y-x) \d y \right\vert \le c r^{1-\alpha} \quad \forall r \in (0,1) \,.
\end{align*}
On the one hand, this setup includes kernels with an integrable antisymmetric part, but it also includes some kernels whose antisymmetric part might be of the same order as the symmetric part if $\alpha \in (0,1)$. On the other hand, \cite{ImSi20} does not cover a large and natural class of nonlocal drifts that are covered by our methods. In particular, the cancellation condition implies that the dual operator is also of the form \eqref{eq:op} plus a bounded killing term. Such a representation formula for the dual operator is rather restrictive and is not necessarily true in our setup. 

In the beginning of the introduction, we have already commented on energy form approaches to  H\"older regularity, local boundedness and Harnack inequalities in the \emph{symmetric}  case developed in \cite{Kas09, CCV11, FeKa13, DKP14, KaSc14, DKP16, Coz17, CKW19, Str19a, Kim20, DyKa20, ChKa20, DZZ21}. We refer to these articles for a detailed account of results and technical challenges in this case.

It is important to mention that there have been a lot of related research activities for nonlocal operators $\mathcal{L}$ that are not related to bilinear forms. Let us make a digression to explain some results for these nonlocal \emph{non-divergence form} operators. This naming relates to the fact that $\mathcal{L} u$ can be evaluated pointwise for smooth functions $u$ as in the case 
\begin{align}\label{eq:L-nondiv}
\mathcal{L} u (x) = \int_{\R^d} \left[ u(x+h) - u(x) - \mathbbm{1}_{B_1}(h) \left(\nabla u(x) , h\right)\right] K(x,x+h) \d h \,,
\end{align}
where symmetry in the form of $K(x,y) =  K(y,x)$ is not assumed. 
 \cite{Sil06} provides a surprisingly elementary proof of H\"older regularity for elliptic equations. Fully nonlinear elliptic equations are treated in \cite{CaSi09} under the additional symmetry assumption $K(x,x+h) = K(x,x-h)$ for all $x,h \in \R^d$. Main results include H\"older regularity estimates and a Harnack inequality. The results are robust in the same way as the results of this work, see the remark after \autoref{thm:mainthmPDEdual}.  Results of \cite{CaSi09} have been extended for nonsymmetric kernels in \cite{ChD12}, \cite{KiLe13b}. Analogous results for parabolic equations were established in \cite{Sil11}, \cite{ChD14}, \cite{ScSi16}. Note that operators $\mathcal{L}$ of the form \eqref{eq:L-nondiv} come along with first order local drift terms. If we set
\begin{align*}
J(x,h) = K(x,x+h), \quad J_e(x,h) = \tfrac{1}{2}(J(x,h)+J(x,-h)), \quad J_o(x,h) = \tfrac{1}{2}(J(x,h)-J(x,-h)) \,,
\end{align*}
then 
\begin{align*}
\begin{split}
\mathcal{L} u (x) &= \int_{\R^d} (u(x+h)+u(x-h)-2u(x)) J_e(x,h)\d h\\
&\quad + \int_{\R^d} (u(x+h)-u(x-h)) J_o(x,h)\d h + \left( \int_{B_1}J_o(x,h)\d h\right)\cdot\nabla u(x).
\end{split}
\end{align*}
Formally, there is no obstacle to study equations involving $\mathcal{L}$ for $J_e(x,h) \asymp \vert h \vert^{-d-\alpha}$ with $\alpha \le 1$ and $J_o \neq 0$, although the order of the first summand does not necessarily dominate the order of the drift term, see \cite{SiSn16} for a related counterexample. \cite{Sil11}, \cite{ChDa16}, \cite{ScSi16} establish H\"older regularity of solutions to certain parabolic equations under certain conditions even when $\alpha \le 1$ and a drift term of type $\left(b(x) , \nabla u(x)\right)$ is present.

Finally, let us mention results of different nature that are related to nonsymmetric operators in the sense of this article. Nonsymmetric quadratic forms of jump-type were investigated in \cite{FuUe12}. Assumptions on $K$ are given such that $\cE$ is a regular lower bounded semi-Dirichlet form and existence of an associated Hunt process follows. These results were extended in \cite{ScWa15}. In \cite{Uem14}, a class of degenerate nonsymmetric diffusion processes with jumps was investigated. Concerning the actual objects of study, the closest to our article is \cite{FKV15}, where existence and uniqueness of solutions to the elliptic and parabolic Dirichlet problem for linear nonlocal operators with nonsymmetric jumping kernels, as well as an elliptic weak maximum principle were proved. In \cite{Dav20}, the author established parabolic comparison principles for nonlinear nonlocal equations with a nonlocal drift. A nonlocal analog to Donsker and Varadhan's inverse problem for elliptic nonlocal operators with nonlocal drift reminiscent of \eqref{eq:protexop} was derived in \cite{DaTo20}. Let us also mention the recent article \cite{DRSV20} where interior and boundary regularity properties are studied for translation invariant operators, i.e., generators of nonsymmetric L\'evy processes.

\subsection{Outline}
This article is structured as follows: In \autoref{sec:prelim} the weak solution concept is introduced and some auxiliary results are provided.\\
\autoref{sec:Caccioppoli} is devoted to the proof of several Caccioppoli-type estimates suitable to the nonsymmetric structure of the nonlocal forms under consideration. A reader who is only interested in the proof of \autoref{thm:mainthmPDE} might only consult \autoref{lemma:MSgen} and \autoref{lemma:MSgen2}.\\
The proofs of weak parabolic Harnack inequalities, see \autoref{thm:mainthmPDE}(i) and \autoref{thm:mainthmPDEdual}(i), are given in \autoref{sec:wpHI}. In \autoref{sec:proofs}, we prove H\"older regularity estimates, see \autoref{thm:mainthmPDE}(ii), \autoref{thm:mainthmPDEdual}(ii). In \autoref{sec:approx} we investigate convergence of nonsymmetric nonlocal forms to local forms with a drift term in the sense of Mosco-Hino. Several examples are presented in \autoref{sec:examples}.

\pagebreak[3]
\section{Preliminaries}
\label{sec:prelim}

In this section we explain the notion of a weak (super/sub)-solution to \eqref{PDE}, resp. \eqref{PDEdual} and establish several auxiliary results which will be of use in the following chapters.

Let us fix $\alpha \in (0,2)$ and $\theta \in [\frac{d}{\alpha},\infty]$ for the remainder of this article. In the following all constants are assumed to depend only on $d,\alpha,\theta$ and the constants in \eqref{K1}, \eqref{K1glob}, \eqref{K2}, \eqref{cutoff}, \eqref{Poinc}, \eqref{Sob} and \eqref{cutoff2} if not explicitly mentioned.

Given a kernel $K$ and the associated operator $L$, see \eqref{eq:op}, and its bilinear form $\cE$, see \eqref{eq:energy}, we define the dual bilinear form $\widehat{\cE}$ via $\widehat{\cE}(u,v) = \cE(v,u)$.
The form $\widehat{\cE}$ is associated with $\widehat{L}$, the dual operator of $L$ defined via
\begin{align}
\cE(u,v) = (L u , v) = (\widehat{L} v , u) = \widehat{\cE}(v,u).
\end{align}

\subsection{Weak solution concept}

For the study of solutions to \eqref{PDE} and \eqref{PDEdual} to be based on a solid mathematical framework we introduce the following function spaces for $\Omega \subset \R^d$:
\begin{align*}
V(\Omega|\R^d) &= \left\lbrace v : \R^d \to \R : v \mid_{\Omega} \in L^2(\Omega) : (v(x)-v(y))K_s^{1/2}(x,y) \in L^2(\Omega \times \R^d)\right\rbrace,\\
H_{\Omega}(\R^d) &= \left\lbrace v \in V(\R^d|\R^d) : v = 0 \text{ on } \R^d \setminus \Omega \right\rbrace
\end{align*}
equipped with the norms
\begin{align*}
\Vert v \Vert_{V(\Omega|\R^d)}^2 &= \Vert v \Vert_{L^2(\Omega)}^2 + \int_{\Omega} \int_{\R^d}(v(x)-v(y))^2K_s(x,y) \d y \d x,\\
\Vert v \Vert_{H_{\Omega}(\R^d)}^2 &= \Vert v \Vert_{L^2(\R^d)}^2 + \cEs(v,v).
\end{align*}

We emphasize that both spaces are completely determined by the symmetric part of the jumping kernel $K_s$. When $\cEs \asymp [\cdot]^2_{H^{\alpha/2}}$, the notations $V(\Omega|\R^d) = V^{\alpha}(\Omega | \R^d)$ and $H_{\Omega}(\R^d) = H^{\alpha/2}_{\Omega}(\R^d)$ are standard.

The following definition provides the solution concept of this article. In order to prove H\"older estimates for solutions to \eqref{PDEdual}, we need to consider the more general equation \eqref{PDEdualext}.

\begin{definition}
Let $\Omega \subset \R^d$ be a bounded domain, $I \subset \R$ a finite interval and $f \in L^{\infty}(I \times \Omega)$. 
\begin{itemize}
\item[(i)] We say that $u \in L^2_{loc}(I;V(\Omega|\R^d))$ is a supersolution to \eqref{PDE} in $I \times \Omega$ if the weak $L^2(\Omega)$-derivative $\partial_t u$ exists, $\partial_t u \in L^1_{loc}(I;L^2(\Omega))$ and 
\begin{align}
\label{eq:supercal}
(\partial_t u(t),\phi) + \cE(u(t),\phi) \le (f(t),\phi), \quad \forall t \in I,~ \forall \phi \in H_{\Omega}(\R^d) \text{ with } \phi \le 0.
\end{align}
$u$ is called a subsolution if \eqref{eq:supercal} holds true for every $\phi \ge 0$. $u$ is called a solution, if it is a supersolution and a subsolution.

\item[(ii)] We say that $u \in L^2_{loc}(I;V(\Omega|\R^d) \cap L^{2\theta'}(\R^d))$ is a  supersolution to \eqref{PDEdualext} in $I \times \Omega$ for some $d \in L^{\infty}(I \times \R^d)$ if $\partial_t u \in L^1_{loc}(I;L^2(\Omega))$ and 
\begin{align}
\label{PDEdualext}\tag{$\widehat{\text{PDE}}_d$}
(\partial_t u(t),\phi) + \widehat{\cE}(u(t),\phi) + \widehat{\cE}^{K_a}(d(t),\phi) \le (f(t) , \phi)
\end{align}
for every $t \in I$ and every $\phi \in H_{\Omega}(\R^d)$ with $\phi \le 0$. We call $u$ a supersolution to \eqref{PDEdual} if $u$ is a supersolution to \eqref{PDEdualext} with $d \equiv 0$.
\end{itemize}
\end{definition}

Let us point out that the solution concept also makes sense under much weaker assumptions on $u$ without any change in the proofs being needed (see \cite{FeKa13}). In particular, one can drop the condition that the weak time derivative $\partial_t u$ exists. However, we restrict ourselves to this stronger notion of a supersolution in order for the presentation to be less technical. \\
We only consider solutions on special time-space cylinders $I_R(t_0) \times B_{2R}$, where $B_{2R} \subset \Omega$ is a ball, $I_R(t_0) = (t_0 - R^{\alpha},t_0 + R^{\alpha})$, $0 < R \le 1$, $t_0 \in \R$, and $\Omega \subset \R^d$ is a fixed open set. Moreover:
\begin{align*}
I^{\ominus}_R(t_0) := (t_0 - R^{\alpha}, t_0), ~~ I^{\oplus}_R(t_0) := (t_0, t_0 + R^{\alpha}).
\end{align*}

The following lemma ensures that the expressions in \eqref{eq:supercal} are well-defined. For symmetric forms such result is immediate from the definition of the function spaces and H\"older's inequality. 

\begin{lemma}
\label{lemma:welldef}
Let $0 < \rho \le r \le 1$, $B_{2r} \subset \Omega$.
\vspace{-0.2cm}
\begin{itemize}
\item[(i)] Assume that one of the following is true:
\begin{itemize}
\item \eqref{K1} holds true with $\theta = \infty$,
\item \eqref{K1} holds with $\theta \in [\frac{d}{\alpha}, \infty)$ and \eqref{Sob}   holds true.
\end{itemize} 
Then $\cE(u,\phi)$ is well-defined for $u \in V(B_{r+\rho}|\R^d)$, $\phi \in H_{B_{r+\rho}}(\R^d)$.
\item[(ii)] Assume that \eqref{K1glob} holds true with $\theta \in [\frac{d}{\alpha}, \infty]$.\\
Then $\widehat{\cE}(u,\phi)$ is well-defined for $u \in V(B_{r+\rho}|\R^d) \cap L^{2\theta'}(\R^d)$ and $\phi \in H_{B_{r+\frac{\rho}{2}}}(\R^d)$.
\item[(iii)] Assume that \eqref{K1glob} holds true with $\theta \in [\frac{d}{\alpha}, \infty]$ and that \eqref{cutoff} is satisfied.\\
Then $\widehat{\cE}^{K_a}(d,\phi)$ is well-defined for $d \in L^{\infty}(\R^d)$ and $\phi \in H_{B_{r+\frac{\rho}{2}}}(\R^d)$.
\end{itemize}
\end{lemma}

\begin{proof}
It is well-known that $\cEs(u,\phi) < \infty$ for $u \in V(B_{r+\rho}|\R^d)$, $\phi \in H_{B_{r+\rho}}(\R^d)$. For the antisymmetric part, it holds by H\"older's and Young's inequality:
\begin{align*}
\cEa_{B_{r+\rho}}(u,\phi) &= \int_{B_{r+\rho}} \int_{B_{r+\rho}} (u(x)-u(y))\phi(x) K_a(x,y) \d y \d x\\
&\le \frac{1}{2}\cE^J_{B_{r+\rho}}(u,u) + \frac{1}{2}\int_{B_{r+\rho}} \int_{B_{r+\rho}} \phi^2(x) \frac{\vert K_a(x,y)\vert^2}{J(x,y)} \d y \d x.
\end{align*}
The first summand is finite by definition of $V(B_{r+\rho}|\R^d)$ and \eqref{K1}. In case $\theta = \infty$, the second summand is finite by \eqref{K1} and if $\theta \in [\frac{d}{\alpha},\infty)$, by H\"older's inequality:
\begin{align*}
\int_{B_{r+\rho}} \phi^2(x) \left(\int_{B_{r+\rho}}\frac{\vert K_a(x,y)\vert^2}{J(x,y)} \d y\right) \d x &\le \Vert \phi^2 \Vert_{L^{\theta'}(B_{r+\rho})} \left\Vert \int_{B_{r+\rho}}\frac{\vert K_a(\cdot,y)\vert^2}{J(\cdot,y)} \d y \right\Vert_{L^{\theta}(B_{r+\rho})}\\
&\le C\vert B_{r+\rho}\vert^{p} \Vert \phi^2 \Vert_{L^{\frac{d}{d-\alpha}}(B_{r+\rho})},
\end{align*}
where $\frac{1}{p} = \left( \frac{d}{(d-\alpha)\theta'}\right)'$. Thus, we conclude from \eqref{Sob} that also in this case the second term is finite. Moreover, we estimate using \eqref{eq:KaKs}
\begin{align*}
&\cE^{K_a}_{(B_{r+\rho} \times B_{r+\rho})^c} (u,\phi) = \int_{B_{r+\rho}} \int_{B_{r+\rho}^c} (u(x)-u(y))\phi(x) K_a(x,y) \d y \d x\\
&\le \frac{1}{2}\int_{B_{r+\rho}} \int_{B_{r+\rho}^c} (u(x)-u(y))^2 K_s(x,y)\d y \d x + \frac{1}{2}\int_{B_{r+\rho}} \int_{B_{r+\rho}^c} \phi^2(x) K_s(x,y) \d y \d x\\
&\le \frac{1}{2}\int_{B_{r+\rho}} \int_{B_{r+\rho}^c} (u(x)-u(y))^2 K_s(x,y)\d y \d x + \frac{1}{2}\int_{B_{r+\rho}} \int_{B_{r+\rho}^c} (\phi(x) - \phi(y))^2 K_s(x,y) \d y \d x.
\end{align*}
Again, the first summand is finite by definition of $V(B_{r+\rho}|\R^d)$. The second summand is finite since $\phi \in H_{B_{r+\rho}}(\R^d)$, proving (i). In order to prove (ii), we compute
\begin{align*}
\widehat{\cE}^{K_a}(u,\phi) &= \int_{\R^d} \int_{\R^d} (\phi(x)-\phi(y))u(x) K_a(x,y) \d y \d x\\
&\le \frac{1}{2}\cE^J(\phi,\phi) + \frac{1}{2}\int_{\R^d}  u^2(x) \left(\int_{\R^d} \frac{\vert K_a(x,y)\vert^2}{J(x,y)} \d y \right) \d x.
\end{align*}
The second summand is finite as a direct consequence of H\"older's inequality and the assumption $u \in L^{2\theta'}(\R^d)$. To prove finiteness of the first summand, we estimate using \eqref{K1glob}
\begin{align*}
\cE^J(\phi,\phi) &= \cE^J_{B_{r+\rho}}(\phi,\phi) + \cE^J_{(B_{r+\rho} \times B_{r+\rho)^c}}(\phi,\phi)\\
&\le \cE^{K_s}_{B_{r+\rho}}(\phi,\phi) + \int_{B_{r+\frac{\rho}{2}}} \phi^2(x) \left(\int_{\R^d \setminus B_{\frac{\rho}{2}}(x)} J(x,y) \d y\right) \d x < \infty
\end{align*}
where we applied \eqref{cutoff} with $\tau = \tau_{x,\frac{\rho}{4},\frac{\rho}{4}}$ in the last step and used that $\phi \in H_{B_{r+\frac{\rho}{2}}}(\R^d)$.\\
Finally, we prove (iii). We decompose the domain of integration into three parts: $\R^d \times \R^d = (B_{r+\rho} \times B_{r+\rho}) \cup (B_{r+\rho} \times B_{r+\rho}^c) \cup (B_{r+\rho}^c \times B_{r+\rho})$ and compute:
\begin{align*}
\int_{B_{r+\rho}} \int_{B_{r+\rho}} |\phi(x) - \phi(y)||K_a(x,y)| \d y \d x &\le \cE^{J}_{B_{r+\rho}}(\phi,\phi) + \int_{B_{r+\rho}} \left(\int_{B_{r+\rho}} \frac{|K_a(x,y)|^2}{J(x,y)} \d y \right) \d x.
\end{align*}
The right hand side is finite by $\phi \in H_{B_{r+\frac{\rho}{2}}}(\R^d)$ and \eqref{K1glob}. Next, we estimate using \eqref{eq:KaKs}:
\begin{align*}
\int_{B_{r+\rho}} \int_{B_{r+\rho}^c} |\phi(x) - \phi(y)||K_a(x,y)| \d y \d x &= \int_{B_{r+\frac{\rho}{2}}} |\phi(x)| \left( \int_{B_{r+\rho}^c}|K_a(x,y)| \d y \right) \d x\\
&\le \Vert \phi \Vert_{L^2(B_{r+\frac{\rho}{2}})} \left\Vert \int_{B_{\frac{\rho}{2}}(\cdot)^c} K_s(\cdot,y) \d y\right\Vert_{L^{2}(B_{r + \frac{\rho}{2}})}.
\end{align*}
The right hand side of this quantity is finite by definition of $H_{B_{r + \frac{\rho}{2}}}(\R^d)$ and by \eqref{cutoff}, applied with $\tau = \tau_{\cdot, r+ \frac{\rho}{2}, \frac{\rho}{2}}$. The proof of finiteness for the integral over $(B_{r+\rho}^c \times B_{r+\rho})$ follows by the same arguments.
\end{proof}

\subsection{Auxiliary results}

We list several results for later use in the proofs of \autoref{sec:Caccioppoli}.
 
The following lemma is a direct consequence of the symmetry properties of $K_s$ and $K_a$:

\begin{lemma}
\label{lemma:symmlemma}
Let $u,v \in L^2(\R^d)$, $f : \R^d \to \R$ be measurable and $A,B \subset \R^d$. Then
\begin{align*}
\cE^{K_s}_{A \times B}(u,v) &= 2\iint_{A \times B \cap \{ f(y) > f(x) \}} (u(x) - u(y))(v(x) - v(y)) K_s(x,y) \d y \d x,\\
\cE^{K_a}_{A \times B}(u,v) &= 2\iint_{A \times B \cap \{ f(y) > f(x) \}} (u(x) - u(y))(v(x) + v(y)) K_a(x,y) \d y \d x.
\end{align*}
\end{lemma}

\begin{proof}
The first identity follows by the following computation:
\begin{align*}
&\iint_{A \times B \cap \{ f(x) > f(y) \}} (u(x) - u(y))(v(x) - v(y)) K_s(x,y) \d y \d x\\
&= \iint_{A \times B \cap \{ f(y) > f(x) \}} (u(y) - u(x))(v(y) - v(x)) K_s(y,x) \d y \d x\\
&= \iint_{A \times B \cap \{ f(y) > f(x) \}} (u(x) - u(y))(v(x) - v(y)) K_s(x,y) \d y \d x.
\end{align*}
The second identity follows by the following computation:
\begin{align*}
&\iint_{A \times B \cap \{ f(x) > f(y) \}} (u(x) - u(y))(v(x) + v(y)) K_a(x,y) \d y \d x\\
&= \iint_{A \times B \cap \{ f(y) > f(x) \}} (u(y) - u(x))(v(y) + v(x)) K_a(y,x) \d y \d x\\
&= \iint_{A \times B \cap \{ f(y) > f(x) \}} (u(x) - u(y))(v(x) + v(y)) K_a(x,y) \d y \d x.
\end{align*}
\end{proof}

As we explained in \autoref{sec:introstrategy}, the proof of Caccioppoli-type estimates for nonsymmetric forms relies on the fact that lower order terms can be absorbed into symmetric energies related to the leading order term. This argument requires some control over the terms which are not absorbed. In our setup, this is guaranteed by assumption \eqref{K1}, as the following lemma demonstrates:

\begin{lemma}
\begin{itemize}
\item[(i)] Assume that \eqref{K1} holds true for some $\theta \in [\frac{d}{\alpha},\infty]$. Moreover, assume \eqref{Sob} if $\theta < \infty$. Then, there exists $c_1 > 0$ such that for every $\delta > 0$ there is $C(\delta) > 0$ such that for every $v \in L^2(B_{r+\rho})$ with $\supp(v) \subset B_{r + \frac{\rho}{2}}$, and every ball $B_{2r} \subset \Omega$ with $0 < \rho \le r \le 1$ it holds
\begin{align}
\label{eq:K1consequence}
\int_{B_{r+\rho}} v^2(x) \left(\int_{B_{r+\rho}} \frac{|K_a(x,y)|^2}{J(x,y)} \d y\right) \d x \le \delta \cE^{K_s}_{B_{r+\rho}}(v,v) + c_1(C(\delta) + \delta \rho^{-\alpha}) \Vert v^2 \Vert_{L^1(B_{r+\rho})}.
\end{align}
Moreover, if $\theta \in (\frac{d}{\alpha},\infty]$, the constant $C(\delta)$ has the following form:
\begin{align}
\label{eq:quantifiedK1consequence}
C(\delta) = \begin{cases}
\Vert W \Vert_{L^{\infty}(B_{r+\rho})}, &\theta = \infty,\\
\delta^{\frac{d}{d-\theta \alpha}} \Vert W \Vert_{L^{\theta}(B_{r+\rho})}^{\frac{\theta \alpha}{\theta\alpha - d}}, &\theta \in (\frac{d}{\alpha},\infty),
\end{cases}
\text{ where } W(x) := \int_{B_{r+\rho}} \frac{|K_a(x,y)|^2}{J(x,y)} \d y.
\end{align}

\item[(ii)] Assume that \eqref{K1glob} holds true for some $\theta \in [\frac{d}{\alpha},\infty]$. Moreover, assume \eqref{Sob} if $\theta < \infty$.
then \eqref{eq:K1consequence} holds true with $\left(\int_{\R^d} \frac{|K_a(x,y)|^2}{J(x,y)} \d y\right)$ instead of $\left(\int_{B_{r+\rho}} \frac{|K_a(x,y)|^2}{J(x,y)} \d y\right)$.
\end{itemize}
\end{lemma}

\begin{proof}
For $\theta = \infty$, \eqref{eq:K1consequence} directly follows from H\"older's inequality. Let now $\theta \in [\frac{d}{\alpha},\infty)$.
Writing $W(x) = \int_{B_{r+\rho}} \frac{|K_a(x,y)|^2}{J(x,y)} \d y =: W_1(x) + W_2(x)$, where $W_1(x) := W(x) \mathbbm{1}_{\{ |W(x)| > M\}}$, we observe that for every $\delta > 0$ there exists $M > 0$ such that $\Vert W_1 \Vert_{L^{\frac{d}{\alpha}}(B_{r + \frac{\rho}{2}})} < \delta$. Thus, by \eqref{Sob}:
\begin{align*}
\int_{B_{r+\rho}} v^2(x) W(x) \d x &\le \int_{B_{r + \frac{\rho}{2}}} v^2(x) W_1(x) \d x + \int_{B_{r + \frac{\rho}{2}}} v^2(x) W_2(x) \d x\\
&\le \delta \Vert v^2 \Vert_{L^{\frac{d}{d-\alpha}}(B_{r+\rho})} + M(\delta) \Vert v^2 \Vert_{L^1(B_{r + \frac{\rho}{2}})}\\
&\le c\delta \cE^{K_s}_{B_{r+\rho}}(v,v) + (c \delta \rho^{-\alpha} + M(\delta))\Vert v^2\Vert_{L^1(B_{r+\rho})},
\end{align*}
as desired, where we also used $\supp(v) \subset B_{r + \frac{\rho}{2}}$. Moreover, in case $\theta \in (\frac{d}{\alpha} , \infty)$, we compute
\begin{align*}
\Vert W_1 \Vert_{L^{\frac{d}{\alpha}}(B_{r + \frac{\rho}{2}})} &\le 2 \Vert W \Vert_{L^{\theta}(B_{r + \frac{\rho}{2}})} |\{ W \ge M \}|^{\frac{\alpha}{d} - \frac{1}{\theta}} \le 2 \Vert W \Vert_{L^{\theta}(B_{r + \frac{\rho}{2}})} \left( \frac{\Vert W \Vert_{L^{\theta}(B_{r + \frac{\rho}{2}})}}{M} \right)^{\theta \left( \frac{\alpha}{d} - \frac{1}{\theta} \right)}\\
&= 2 \Vert W \Vert_{L^{\theta}(B_{r + \frac{\rho}{2}})}^{\frac{\theta \alpha}{d}} M^{1- \frac{\theta \alpha}{d}}.
\end{align*}
Thus, in order to prove \eqref{eq:quantifiedK1consequence}, we have to choose $M > \left(\frac{\delta}{2}\right)^{\frac{d}{d-\theta \alpha}} \Vert W \Vert_{L^{\theta}(B_{r + \frac{\rho}{2}})}^{\frac{\theta \alpha}{\theta \alpha - d}}$.\\
The proof of $(ii)$ follows by the same arguments, after redefining $W(x) = \int_{\R^d} \frac{|K_a(x,y)|^2}{J(x,y)} \d y$.
\end{proof}

\begin{remark*}
Note that \eqref{eq:K1consequence} also holds true in case \eqref{eq:Kato} holds true, i.e., there is $C > 0$ such that for every $0 < r \le 1$ and $B_{2r} \subset \Omega$, $v \in L^2(B_{2r})$:
\begin{align}
\left\Vert\frac{|K_a(\cdot,y)|^2}{J(\cdot,y)} \d y \right\Vert_{\mathcal{K}^{d,\alpha}(B_{2r})} \le C, \qquad [v]^2_{H^{\alpha/2}(B_{2r})} + \cE^J_{B_{2r}}(v,v) \le C \cE^{K_s}_{B_{2r}}(v,v).
\end{align}
Here, $\mathcal{K}^{d,\alpha}(B_{r+\rho})$ denotes the Kato class associated to $(-\Delta)^{\alpha/2}$, i.e., the family of measurable functions $W : B_{r+\rho} \to [0,\infty]$:
\begin{align*}
\lim_{\eps \searrow 0} \sup_{z \in B_{r+\rho}} \int_{B_{\eps}(z) \cap B_{r+\rho}} |z-x|^{\alpha-d} W(x) \d x = 0,
\end{align*}
This observation shows, that one could prove \autoref{thm:mainthmPDE} also under \eqref{eq:Kato}, instead of \eqref{K1}. 
\end{remark*}

We state a classical algebraic inequality which will become useful later.

\begin{lemma}
\label{lemma:logusaviour}
Let $\tau_1,\tau_2 \ge 0$. Then for every $\delta \in (0,1)$ it holds
\begin{align}
\tau_1^2 \vee \tau_2^2 \le (1+\delta) (\tau_1^2 \wedge \tau_2^2) + (1+\delta)\delta^{-1}(\tau_2-\tau_1)^2
\end{align}
\end{lemma}

\begin{proof}
Without loss of generality assume that $\tau_1 < \tau_2$. Then  by Young's inequality
\begin{align*}
\tau_2^2 = (\tau_2 - \tau_1 + \tau_1)^2 = \tau_1^2 + (\tau_2 - \tau_1)^2 + 2\tau_1(\tau_2-\tau_1) \le (1+\delta)\tau_1^2 + (1+\delta^{-1})(\tau_2 - \tau_1)^2.
\end{align*}
\end{proof}

The following result states that comparability of energy forms is preserved under addition of weight functions $\tau^2(x) \wedge \tau^2(y)$:

\begin{lemma}
\label{lemma:improvedK2help}
Let $j,J : \R^d \times \R^d \to [0,\infty]$ be symmetric. Assume that there exist $c_1, c_2 > 0$ such that for every $0 < \rho \le r \le 1$ and every $v \in L^2(B_{r+\rho})$ it holds:
\begin{align}
\label{eq:threekernelcomp}
c_1 \cE^{J}_{B_{r+\rho}}(v,v) \le \cE^{K_s}_{B_{r+\rho}}(v,v) \le c_2 \cE^{j}_{B_{r+\rho}}(v,v).
\end{align}
Then, it holds for every $0 < \rho \le r \le 1$, $v \in L^2(B_{r+\rho})$ and $\tau = \tau_{r,\rho}$
\begin{align*}
c_1 \int_{B_{r+\rho}} \int_{B_{r+\rho}} &(v(x) - v(y))^2 (\tau^2(x) \wedge \tau^2(y)) J(x,y) \d y \d x\\
&\le \int_{B_{r+\rho}} \int_{B_{r+\rho}} (v(x) - v(y))^2 (\tau^2(x) \wedge \tau^2(y)) K_s(x,y) \d y \d x\\
&\le c_2 \int_{B_{r+\rho}} \int_{B_{r+\rho}} (v(x) - v(y))^2 (\tau^2(x) \wedge \tau^2(y)) j(x,y) \d y \d x
\end{align*}
\end{lemma}

\begin{proof}
Since $\tau = \tau_{r,\rho}$ is radially decreasing, there exists a Borel measure $\nu$ on $(r,r+\rho]$ such that $\tau^2(x) = \int_{r}^{r+\rho} \mathbbm{1}_{B_t}(x) \nu(\d t)$, see \cite{DyKa13}. An elementary computation shows
\begin{align}
\int_{B_{r+\rho}} \int_{B_{r+\rho}} (v(x) - v(y))^2 (\tau^2(x) \wedge \tau^2(y)) J(x,y) \d y \d x = \int_{r}^{r+\rho} \cE_{B_t}^{J}(v,v) \nu(\d t).
\end{align}
Similar identities hold true for $j,K_s$. Thus, the desired result directly follows from \eqref{eq:threekernelcomp}.
\end{proof}

\subsection{(K1) implies (K2)}

Let $\Omega \subset \R^d$ be bounded. The goal of this section is to show that assumption \eqref{K1} implies \eqref{K2} under suitable additional assumptions on $K_s$. We will prove the following statement:

\begin{proposition} 
\label{prop:K1impliesK2}
Assume that there is $c > 0$ such that for every ball $B_{2r} \subset \Omega$ and every $r \le 1$:
\begin{align}
\label{eq:K2suff}
K_s(x,y) \ge c |x-y|^{-d-\alpha}, ~~ \forall x,y \in B_{2r}.
\end{align}
Assume that \eqref{K1} holds true with $\theta = \infty$ and some jumping kernel $J$ satisfying $J(x,y) \le |x-y|^{-d-\alpha}$. Then there is $n \in \N$ such that \eqref{K2} holds true for every ball $B_{2r} \subset \R^d$ with $r \le 1$ and $B_{5^n r} \subset \Omega$, i.e. there exists a symmetric jumping kernel $j : \R^d \times \R^d \to [0,\infty]$ and $c > 1$ such that 
\begin{align*}
K(x,y) \ge c^{-1} j(x,y),~~ \forall x,y \in B_{2r}, \qquad \cE^j_{B_{2r}}(u,u) \ge c \cE^{K_s}_{B_{2r}}(u,u), ~~ \forall v \in L^2(B_{2r}),
\end{align*}
for every $B_{2r} \subset 5^{-n}\Omega$,.
\end{proposition}

A consequence of \autoref{prop:K1impliesK2} is that for any jumping kernel $K$ satisfying the assumptions of \autoref{prop:K1impliesK2}, assumption \eqref{K2} holds true for all balls which are far from the boundary of $\Omega$, relative to their radius. By close inspection of the proofs of \autoref{thm:mainthmPDE}, \autoref{thm:mainthmPDEdual}, it becomes apparent that this weaker version of \eqref{K2} suffices to also prove the H\"older estimate and the weak Harnack inequality on such balls. By a covering argument, it is then possible to deduce the aforementioned estimates for any ball $B_{2R} \subset \Omega$.

The main ingredient for the proof of \autoref{prop:K1impliesK2} is the following result from \cite{ChSi20}:

\begin{lemma}[\cite{ChSi20}]
\label{lemma:ChSi}
Let $j : \R^d \times \R^d \to [0,\infty]$ be symmetric.
Assume that there are $\lambda > 0$, $\mu \in (0,1)$ such that for every ball $B \subset \Omega$ and every $x \in B$:
\begin{align}
\label{eq:CSass}
|\{y \in B : j(x,y) \ge \lambda |x-y|^{-d-\alpha}\}| \ge \mu |B|.
\end{align}
Then, there exist $c > 0$, $n \in \N$ such that for every ball $B_r \subset \Omega$ with $B_{5^{n}r} \subset \Omega$ and every $u \in L^2(B_r)$:
\begin{align*}
\cE_{B_r}^j(u,u) \ge c [u]^2_{H^{\alpha/2}(B_r)}.
\end{align*}
\end{lemma}

\begin{proof}
Lemma 5.1 in \cite{ChSi20} yields that for every ball $B_{r} \subset \Omega$ with $B_{5^n r} \subset \Omega$, and every $u \in L^2(B_{5^n r})$ it holds:
\begin{align*}
\cE_{B_{5^n r}}^j(u,u) \ge c_1 [u]^2_{H^{\alpha/2}(B_r)}
\end{align*}
for some $c_1 > 0$ and $n \in \N$. Thus, by an argument based on Whitney decomposition, as in Lemma 6.13 in \cite{DyKa20}, it follows that 
\begin{align*}
\cE_{B_{r}}^j(u,u) \ge c_2 [u]^2_{H^{\alpha/2}(B_r)}
\end{align*}
for every $B_r \subset B$, where $B \subset \Omega$ can be any ball with the property $5^n B \subset \Omega$. Here, $c_2 > 0$ is a constant depending only on $c_1,d,\alpha,n$. In particular, $c_2$ does not depend on $B$, which implies the desired result.
\end{proof}

The following lemma contains a sufficient condition for \eqref{K2}:

\begin{lemma}
\label{lemma:goodimplesK2}
Assume that there is $c > 0$ such that for every ball $B_{2r} \subset \Omega$ and every $r \le 1$:
\begin{align*}
K_s(x,y) \ge c |x-y|^{-d-\alpha}, ~~ \forall x,y \in B_{2r}.
\end{align*}
Assume there are $D \in (0,1)$ and $\mu \in (0,1)$ such that for every ball $B \subset \Omega$ and every $x \in B$:
\begin{align}
\label{eq:good}
|\{ y \in B : |K_a(x,y)| \le D K_s(x,y) \}| \ge \mu |B|.
\end{align}
Then there is $n \in \N$ such that \eqref{K2} holds true for every ball $B_{2r} \subset \R^d$ with $r \le 1$ and $B_{5^n r} \subset \Omega$.
\end{lemma}

\begin{proof}
We define 
\begin{align*}
j(x,y) = K_s(x,y) \mathbbm{1}_{\{ K(x,y) \ge (1-D)K_s(x,y) \}}.
\end{align*}
Then by definition, $K \ge (1-D) j$. Note that for $x,y$ such that $|K_a(x,y)| \le D K_s(x,y)$, it holds that 
\begin{align*}
K(x,y) \ge (1-D) K_s(x,y) \ge c(1-D)|x-y|^{-d-\alpha}.
\end{align*}
Therefore, $j$ satisfies \eqref{eq:CSass}. Consequently, \autoref{lemma:ChSi} implies that $\cE^j_{B_{r}}(u,u) \ge c \cE^{K_s}_{B_r}(u,u)$, for every $B_{2r} \subset \R^d$ with $B_{5^n r} \subset \Omega$, as desired. 
\end{proof}

\begin{lemma}
\label{lemma:K1impliesgood}
Assume that for every $B_{2r} \subset \Omega$ with $r \le 1$:
\begin{align}
\label{eq:K1alphastable}
\left\Vert\int_{B_{2r}} \frac{|K_a(\cdot,y)|^2}{|\cdot -y|^{-d-\alpha}} \d y \right\Vert_{L^{\infty}(B_{2r})} \le C.
\end{align}
Then, for every $D > 0$ there is $\mu \in (0,1)$ depending on $\diam(\Omega)$ such that for every ball $B \subset \Omega$ and every $x \in B$:
\begin{align*}
|\{ y \in B : |K_a(x,y)| \le D |x-y|^{-d-\alpha} \}| \ge \mu |B|.
\end{align*}
\end{lemma}

\begin{proof}
Assume for contradiction that there is $D > 0$ such that for every $\mu \in (0,1)$ there exist a ball $B \subset \Omega$ and $x \in B$ such that 
\begin{align}
\label{eq:CSapplcontrary}
|\{ y \in B : |K_a(x,y)| \ge D |x-y|^{-d-\alpha} \}| \ge (1-\mu) |B|.
\end{align}
We will prove that \eqref{eq:K1alphastable} fails under this assumption. Indeed, under \eqref{eq:CSapplcontrary} and by boundedness of $\Omega$, for every $\eps > 0$ it is possible to find a ball $B \subset \Omega$ and $x \in B$ with $B_{\eps}(x) \subset B$ such that
\begin{align*}
|M| := |\{ y \in B_{\eps}(x) : |K_a(x,y)| \ge D|x-y|^{-d-\alpha} \}| \ge \frac{1}{4} |B_\eps(x)|,
\end{align*}
by choosing $\mu > 0$ small enough. As a consequence,
\begin{align*}
\int_{B_{\eps}(x)} \frac{|K_a(x,y)|^2}{|x-y|^{-d-\alpha}} \d y \ge c_1 \eps^{d} \dashint_{M} \frac{|K_a(x,y)|^2}{|x-y|^{-d-\alpha}} \d y \ge c_2 \eps^{d} \dashint_{M} |x-y|^{-d-\alpha} \d y \ge c_3 \eps^{-\alpha},
\end{align*}
where $c_1,c_2,c_3 > 0$ depend only on $d,D$.
Taking $\eps \searrow 0$ implies the desired result.
\end{proof}

We are now in the position to prove \autoref{prop:K1impliesK2}:

\begin{proof}[Proof of \autoref{prop:K1impliesK2}]
By assumption, there exists $C > 0$ such that \eqref{eq:K1alphastable} is satisfied. By \autoref{lemma:K1impliesgood} and \eqref{eq:K2suff}, there exists $\mu \in (0,1)$ such that 
\begin{align*}
|\{ y \in B : |K_a(x,y)| \le \frac{1}{2} K_s(x,y) \}| \ge \mu |B|,
\end{align*}
i.e. \eqref{eq:good} holds true with $D = \frac{1}{2}$. Then, the desired result follows from \autoref{lemma:goodimplesK2}.
\end{proof}

\section{Caccioppoli-type estimates for nonsymmetric forms}
\label{sec:Caccioppoli}

The goal of this section is to establish Caccioppoli-type estimates that are suitable to perform Moser iteration for nonnegative supersolutions to \eqref{PDE}, \eqref{PDEdual} and \eqref{PDEdualext}. Caccioppoli estimates for local equations are a simple consequence of the chain rule. Therefore, the main task in our setup is to find suitable algebraic estimates which work as a nonlocal replacement of the chain rule. Although nonlocal Caccioppoli estimates of different kinds have already been established for equations with symmetric jumping kernels, see \cite{Kas09}, \cite{Coz17}, \cite{DyKa20}, \cite{DKP16}, \cite{ChKi21}, it is necessary to refine the existing proofs in order to allow for jumping kernels with an additional antisymmetric part. In order to understand the main technical ideas, it is instructive to look at the proof of \autoref{lemma:MSgen}, which is later applied to solutions of \eqref{PDE}. The details get more complicated when considering \eqref{PDEdual}.

In the sequel, we present proofs of nonlocal Caccioppoli estimates in a very general framework. Such estimates are derived by testing the respective equation with a test function of the form $\phi = -\tau^2(u+\eps)^{-p}$ for $p > 0$, where $u$ is a nonnegative supersolution to the equation, $\tau$ is a suitable cutoff-function and $\eps > 0$. Note that, formally speaking, we do not prove Caccioppoli estimates but estimates for $\cE(u, \phi)$, so that it only remains to apply the supersolution property for $u$. 

The section is structured as follows: In \autoref{subsec:aux-ineq} we derive substitutes for the chain rule, which we later apply to power functions $t \mapsto t^{-p}$ for $p > 0$. Since \autoref{lemma:genaux} provides estimates for general functions $g$ resp. $G$, we can treat the cases $p\in (0,1)$ and $p > 1$ simultaneously in the subsequent results of this section. In \autoref{subsec:basic-forms} we concentrate on the equation \eqref{PDE} resulting in estimates of $\cE(u,-\tau^{2}(u + \eps)^{-p})$ in \autoref{lemma:MSgen} for $p \ne 1$ and corresponding estimates for $p=1$ in \autoref{lemma:MSgen2}. \autoref{subsec:dual-forms} deals with \eqref{PDEdual} and \eqref{PDEdualext}. First, we derive estimates, analogous to the lemmas from \autoref{subsec:basic-forms}, for  $\widehat{\cE}(u,-\tau^{2}(u + \eps)^{-p})$, see \autoref{lemma:MSgendual} and \autoref{lemma:MSgendual2}. Finally, we treat the extra term $\widehat{\cE}^{K_a}(d,-\tau^{2}(u + \eps)^{-p})$ arising in \eqref{PDEdualext}, see \autoref{lemma:MSgendualext} and \autoref{lemma:MSgendualext2}.

\subsection{Substitutes for the chain rule}
\label{subsec:aux-ineq}

\begin{lemma}
\label{lemma:genaux}
Let $g : (0,\infty) \to (0,\infty)$ be continuously differentiable. Assume that $g$ is decreasing. Let $G : (0,\infty) \to \R$ be a function satisfying the relation $G'(t) = (-g'(t))^{1/2}$. Then, for every $s,t > 0$:
\begin{align}
\label{eq:genaux1}
(s-t)(g(t)-g(s)) &\ge (G(t) - G(s))^2,\\
\label{eq:genaux2}
\frac{|t-s|}{|G(t) - G(s)|} &\le \frac{1}{G'(t)} \vee \frac{1}{G'(s)}, \qquad \frac{(g(t) \wedge g(s))|t-s|}{|G(t) - G(s)|} \le \frac{g(t)}{G'(t)} \vee \frac{g(s)}{G'(s)},\\
\label{eq:genaux3}
\frac{|g(t) - g(s)|}{|G(t) - G(s)|} &\le G'(t) \vee G'(s), \qquad \frac{(t \wedge s)|g(t) - g(s)|}{|G(t) - G(s)|} \le tG'(t) \vee sG'(s).
\end{align}
\end{lemma}

\begin{proof}
Let us compute, using that $G'(t) = (-g'(t))^{1/2}$ and that $G' \ge 0$ and with the help of Jensen's inequality:
\begin{align*}
(s-t)(g(t)-g(s)) &= (t-s)\int_s^t -g'(\tau) \d \tau = (t-s)\int_s^t G'(\tau)^2 \d \tau\\
&\ge \left(\int_s^t G'(\tau) \d \tau \right)^2 = (G(t) - G(s))^2,
\end{align*}
which proves \eqref{eq:genaux1}. 
\eqref{eq:genaux3} is a consequence of Cauchy's mean value theorem and the fact that $\frac{|g'|}{|G'|} = \frac{(-g')}{(-g')^{1/2}} = G'$ by definition of $G$ and since $g$ is decreasing. The second assertions in \eqref{eq:genaux2}, respectively \eqref{eq:genaux3} follow by multiplication with $(g(t) \wedge g(s))$, respectively $(t \wedge s)$.
\end{proof}

\begin{lemma}
Let $G : (0,\infty) \to \R$. Then for any $\tau_1,\tau_2 \ge 0$ and $t,s > 0$:
\begin{align}
\label{eq:genaux4}
(\tau_1^2 \wedge \tau_2^2)|G(t) - G(s)|^2 &\ge \frac{1}{2} |\tau_1 G(t) - \tau_2 G(s)|^2 - (\tau_1 - \tau_2)^2 (G^2(t) \vee G^2(s)),\\
\label{eq:genaux5}
(\tau_1^2 \vee \tau_2^2)|G(t) - G(s)|^2 &\le 2 |\tau_1 G(t) - \tau_2 G(s)|^2 +2 (\tau_1 - \tau_2)^2 (G^2(t) \vee G^2(s)).
\end{align}
\end{lemma}

\begin{proof}
The proof of \eqref{eq:genaux4} and \eqref{eq:genaux5} is straightforward. Let us assume without loss of generality that $\tau_1 < \tau_2$. Then
\begin{align*}
\tau_1 G(t) - \tau_2 G(s) &= \tau_1 (G(t) - G(s)) + G(s)(\tau_1 - \tau_2), \\
\tau_2 (G(t) - G(s)) &= \tau_1 G(t) - \tau_2 G(s) + (\tau_2 - \tau_1)G(t)
\end{align*}
and \eqref{eq:genaux4} and \eqref{eq:genaux5} follow by taking absolute value in the above lines. 
\end{proof}

In the sequel, we will work with the following specific functions $g,G$. Let us define for $p > 0$:
\begin{align}
\label{eq:gGdef}
g(t) = t^{-p}, \qquad
G(t) = \begin{cases}
\frac{2\sqrt{p}}{-p+1} t^{\frac{-p+1}{2}},& ~~ p \neq 1,\\
\log(t),& ~~ p = 1.
\end{cases}
\end{align}

It is easy to check that $g,G$ satisfy the assumptions of \autoref{lemma:genaux}, i.e. $g$ is decreasing and it holds $G'(t) = (-g'(t))^{1/2}$. Moreover, $g$ is convex, which implies that $G'$ is decreasing.\\
The following lemma is a direct consequence of the definition of $g,G$ and lists several identities involving $g,G$ for the readers' convenience.

\begin{lemma}
\label{lemma:gGidentities}
For every $p > 0$ and $t \in (0,\infty]$:
\begin{alignat*}{2}
\frac{g(t)}{G'(t)} &= \begin{cases}
\frac{1}{\sqrt{p}}t^{\frac{-p+1}{2}},& ~~ p \neq 1,\\
1,& ~~ p = 1,
\end{cases}
&G'(t) &= \begin{cases}
\sqrt{p} t^{\frac{-p-1}{2}},& ~~ p \neq 1,\\
\frac{1}{t},&~~ p = 1,
\end{cases}\\
t g(t) &= \begin{cases}
t^{-p+1}&, ~~ p \neq 1,\\
1&,~~ p = 1,
\end{cases},
&G^2(t) &= \begin{cases}
\frac{4p}{(p-1)^2}t^{-p+1}&, ~~ p \neq 1,\\
\log(t)^2&,~~ p = 1.
\end{cases},\\
tG'(t) &= \begin{cases}
\sqrt{p} t^{\frac{-p+1}{2}},& ~~ p \neq 1,\\
1,&~~ p = 1,
\end{cases} \qquad
&t^{1/2}g(t)^{1/2} &= \begin{cases}
t^{\frac{-p+1}{2}}&, ~~ p \neq 1,\\
1&,~~ p = 1.
\end{cases},\\
\end{alignat*}
\end{lemma}

The case $p=1$ can be regarded as a special case due to the appearance of the logarithm. In particular, we will avoid using \eqref{eq:genaux4}, \eqref{eq:genaux5} in this case, and apply the following algebraic estimate instead:

\begin{lemma}
Let $\tau_1,\tau_2 \ge 0$ and $t,s > 0$. Then the following holds true:
\begin{align}
\label{eq:logaux4}
(\tau_1^2 \wedge \tau_2^2) |\log t - \log s|^2 &\ge \frac{1}{2} (\tau_1^2 \wedge \tau_2^2) \left(\log \frac{t}{\tau_1} - \log \frac{s}{\tau_2} \right)^2 - (\tau_1 - \tau_2)^2.
\end{align}
\end{lemma}

\begin{proof}
We assume $\tau_1 < \tau_2$ without loss of generality and compute
\begin{align*}
\tau_1\left( \log \frac{t}{\tau_1} - \log \frac{s}{\tau_2} \right) &= \tau_1 (\log   t - \log s) + \tau_1 (\log \tau_2 - \log \tau_1)\\
&= \tau_1 (\log t - \log s) + \tau_1 \dashint_{\tau_1}^{\tau_2} \frac{1}{\tau} \d \tau.
\end{align*}
Taking absolute value,
\begin{align*}
\tau_1\left| \log \frac{t}{\tau_1} - \log \frac{s}{\tau_2} \right|
&\le \tau_1 |\log t - \log s| + \tau_1 \dashint_{\tau_1}^{\tau_2} \frac{1}{\tau} \d \tau \\
&\le \tau_1 |\log t - \log s| + (\tau_2 - \tau_1).
\end{align*}
This proves the desired result.
\end{proof}

\subsection{Basic bilinear forms}\label{subsec:basic-forms}

In this section, we prove nonlocal Caccioppoli-type estimates which will be used to derive the weak parabolic Harnack inequality for supersolutions to \eqref{PDE}. We start with an estimate arising from testing the equation with $-\tau^2 (u+\eps)^{-p}$ for some $p \neq 1$. 

The following lemma merges Lemma 4.3 in \cite{DyKa20} and Lemma 3.3 in \cite{FeKa13} into one statement and generalizes both results to nonlocal forms with nonsymmetric jumping kernels. 

\begin{lemma}\label{lemma:MSgen}
Assume that \eqref{K1} and \eqref{cutoff} hold true for some $\theta \in [\frac{d}{\alpha},\infty]$. Moreover, assume \eqref{Sob} if $\theta < \infty$. Then there exist $c_1,c_2 > 0$ such that for every $0 < \rho \le r \le 1$ and every nonnegative function $u \in V(B_{r+\rho}|\R^d)$, and every $p \ge 1 - \kappa^{-1}$, with $p \neq 1$, $\eps > 0$, it holds:
\begin{align*}
\cE^{K_s}_{B_{r+\rho}}(\tau \U^{\frac{-p+1}{2}},\tau \U^{\frac{-p+1}{2}})
&\le c_1|p-1| \cE(u,-\tau^{2}\U^{-p}) + c_2 \left(1 \vee |p-1| \right) \rho^{-\alpha} \Vert \U^{-p+1}\Vert_{L^{1}(B_{r+\rho})},
\end{align*}
where $B_{2r} \subset \Omega$, $\tau = \tau_{r,\frac{\rho}{2}}$, $\U = u + \eps$, and $\kappa = 1 + \frac{\alpha}{d}$.
\end{lemma}

\begin{remark*}
Note that one can replace the lower bound $p \ge 1 - \kappa^{-1}$ by any number strictly greater than zero for the above result to hold true. However, the constant $c$ of course depends on the lower bound for $p$.
\end{remark*}

\begin{proof}
Let us fix a pair of functions $g,G$ satisfying the assumptions of \autoref{lemma:genaux}. Later, we will fix $p \ge 1 - \kappa^{-1}$ with $p \neq 1$ and define $g,G$ as in \eqref{eq:gGdef}.
We define $M := \{(x,y) \in B_{r+\rho} \times B_{r+\rho} : g(\U(x)) > g(\U(y)) \}$. The proof is divided into several steps.

Step 1: First, we claim that there are $c_1,c_2 > 0$ such that for every $\delta_0 \in (0,1)$ and $\delta_1 > 0$:
\begin{align}
\label{eq:MSgenstep1}
\begin{split}
&\cE_{B_{r+\rho}}(u,-\tau^2 g(\U)) \ge c_1 \delta_0 \cE_{B_{r+\rho}}^{K_s}(\tau G(\U),\tau G(\U)) - \delta_1\cE^{K_s}_{B_{r+\rho}}(\tau g(\U) G'(\U)^{-1} , \tau g(\U) G'(\U)^{-1}) \\
\quad &-c_2\Bigg(\delta_0 C \Vert \tau^2 G(\U)^2 \Vert_{L^{1}(B_{r+\rho})} + \frac{\delta_0}{\rho^{\alpha}} \Vert G(\U)^2 \Vert_{L^{1}(B_{r+\rho})}   + \frac{\delta_0^{-1}  +\delta_0\delta_1}{\rho^{\alpha}} \left\Vert \left(\frac{g(\U)}{G'(\U)} \right)^2 \right\Vert_{L^{1}(B_{r+\rho})} \\
&+ C(\delta_0\delta_1)\left\Vert \tau^2\left(\frac{g(\U)}{G'(\U)} \right)^2 \right\Vert_{L^{1}(B_{r+\rho})}+ \rho^{-\alpha} \left\Vert \U g(\U) \right\Vert_{L^{1}(B_{r+\rho})} \Bigg),
\end{split}
\end{align}
where $C = C(1), C(\delta_0\delta_1) > 0$ denote respective constants from \eqref{eq:K1consequence}.

For the symmetric part, we compute using \autoref{lemma:symmlemma}:
\begin{align*}
\cE^{K_s}_{B_{r+\rho}}(u,-\tau^{2}g(\U)) &= 2 \iint_{M} (\U(y) - \U(x))(\tau^2(x)g(\U(x)) - \tau^2(y)g(\U(y))) K_s(x,y) \d y \d x\\
&=2 \iint_{M} (\U(y) - \U(x))(g(\U(x)) - g(\U(y)))\tau^2(x) K_s(x,y) \d y \d x\\\
& \quad + 2 \iint_{M} (\U(y) - \U(x))g(\U(y))(\tau^2(x) - \tau^2(y)) K_s(x,y) \d y \d x\\
&= I_s + J_s.
\end{align*}
For the nonsymmetric part, we compute using \autoref{lemma:symmlemma}:
\begin{align*}
\cE^{K_a}_{B_{r+\rho}}(u,-\tau^{2}g(\U)) &= 2 \iint_{M} (\U(y) - \U(x))(\tau^2(x)g(\U(x)) + \tau^2(y)g(\U(y))) K_a(x,y) \d y \d x\\
&=2 \iint_{M} (\U(y) - \U(x))(g(\U(x)) - g(\U(y)))\tau^2(x) K_a(x,y) \d y \d x\\\
& \quad + 2 \iint_{M} (\U(y) - \U(x))g(\U(y))(\tau^2(x) + \tau^2(y)) K_a(x,y) \d y \d x\\
&= I_a + J_a.
\end{align*}
By adding up $I_s + I_a$ and using \eqref{eq:genaux1}, \eqref{eq:genaux4}, \eqref{eq:KaKs}, \eqref{cutoff} and again \autoref{lemma:symmlemma}, we obtain that for any $\delta_0 \in (0,1)$:
\begin{align*}
I_s + I_a &= 2 \iint_{M} (\U(y) - \U(x))(g(\U(x)) - g(\U(y)))\tau^2(x) K(x,y) \d y \d x\\
&\ge 2\iint_{M} (G(\U(x)) - G(\U(y)))^2 ( \tau^2(x) \wedge \tau^2(y) )  K(x,y) \d y \d x\\
&\ge \delta_0 \iint_{M} (\tau(x) G(\U(x)) - \tau(y) G(\U(y)))^2  K(x,y) \d y \d x - 2\delta_0 \int_{B_{r+\rho}} G^2(\U(x)) \Gamma^{K_s}(\tau,\tau)(x) \d x\\
&\ge \frac{\delta_0}{2} \cE_{B_{r+\rho}}^{K_s}(\tau G(\U) , \tau G(\U)) - c \delta_0 \rho^{-\alpha} \Vert G(\U)^2\Vert_{L^{1}(B_{r+\rho})}\\
& \quad -\frac{\delta_0}{2} \int_{B_{r+\rho}} \int_{B_{r+\rho}} (\tau G(\U(x)) - \tau G(\U(y)))^2 |K_a(x,y)| \d y \d x.
\end{align*}
For the summand involving $K_a$, we find using \eqref{K1} and \eqref{eq:K1consequence} that for every $\delta > 0$  
\begin{align*}
&\int_{B_{r+\rho}} \int_{B_{r+\rho}} (\tau G(\U(x)) - \tau G(\U(y)))^2 |K_a(x,y)| \d y \d x \\
&\le \delta \cE_{B_{r+\rho}}^{J}(\tau G(\U) , \tau G(\U)) + c\delta^{-1} \int_{B_{r+\rho}} \tau^2(x) G^2(\U(x)) \left(\int_{B_{r+\rho}}  \frac{|K_a(x,y)|^2}{J(x,y)} \right) \d x\\
&\le c\delta \cE_{B_{r+\rho}}^{K_s}(\tau G(\U) , \tau G(\U)) + c[\delta^{-1}C(\delta^{2}) + \delta \rho^{-\alpha}] \Vert \tau^2 G(\U)^2 \Vert_{L^1(B_{r+\rho})}.
\end{align*}
Here, $C(\delta^{2}) > 0$ denotes the constant from \eqref{eq:K1consequence}.
Consequently, by choosing $\delta > 0$ small enough,
\begin{align}
I_s + I_a \ge \frac{\delta_0}{4} \cE_{B_{r+\rho}}^{K_s}(\tau G(\U) , \tau G(\U)) - c \delta_0 \left( C \Vert \tau^2 G(\U)^2 \Vert_{L^1(B_{r+\rho})} + \rho^{-\alpha} \Vert G(\U)^2 \Vert_{L^1(B_{r+\rho})}\right).
\end{align}

To estimate $J_s$, we use \eqref{eq:genaux2}, \eqref{eq:genaux5}, \eqref{cutoff} to prove that
\begin{align*}
J_s &\ge - \iint_{M} |G(\U(x)) - G(\U(y))| \left(\frac{g(\U(y))}{G'(\U(y))} \vee \frac{g(\U(y))}{G'(\U(y))} \right)|\tau(x) - \tau(y)| ( \tau(x) \vee \tau(y) ) K_s(x,y) \d y \d x\\
&\ge - \frac{\delta_0}{32} \iint_M  (G(\U(x)) - G(\U(y)))^2 ( \tau^2(x) \vee \tau^2(y) ) K_s(x,y) \d y \d x - \frac{c}{\delta_0} \rho^{-\alpha} \left\Vert \left(\frac{g(\U)}{G'(\U)} \right)^2 \right\Vert_{L^{1}(B_{r+\rho})}\\
&\ge - \frac{\delta_0}{16} \cE_{B_{r+\rho}}^{K_s}(\tau G(\U) , \tau G(\U)) - \frac{c}{\delta_0}\rho^{-\alpha} \left\Vert \left(\frac{g(\U)}{G'(\U)} \right)^2 \right\Vert_{L^{1}(B_{r+\rho})} - c \delta_0 \rho^{-\alpha} \Vert G(\U)^2 \Vert_{L^1(B_{r+\rho})}.
\end{align*}
Note that \eqref{eq:genaux2} is applicable since on $M$ it holds: $g(\U(y)) \le g(\U(x))$.
Next, we estimate $J_a$, and prove using \autoref{lemma:logusaviour}, \eqref{eq:KaKs}, \eqref{eq:genaux2}, \eqref{cutoff}, \eqref{eq:genaux5}, as well as the fact that $g(\U(y)) \le g(\U(x))$ on $M$, that for every $\delta_1 > 0$:
\begin{align*}
J_a &\ge -8\iint_{M} |\U(x) - \U(y)|g(\U(y)) ( \tau^2(x) \wedge \tau^2(y) ) |K_a(x,y)| \d y \d x\\
&\quad - 8\iint_{M} |\U(x) - \U(y)|g(\U(y)) (\tau(x) - \tau(y))^2 K_s(x,y) \d y \d x\\
&\ge - \frac{\delta_0}{c} \iint_M  (G(\U(x)) - G(\U(y)))^2 ( \tau^2(x) \wedge \tau^2(y) ) J(x,y) \d y \d x\\
&\quad - c \delta_0^{-1} \iint_M \left(\frac{g(\U(x))}{G'(\U(x))} \vee \frac{g(\U(y))}{G'(\U(y))} \right)^2 ( \tau^2(x) \wedge \tau^2(y) ) \frac{|K_a(x,y)|^2}{J(x,y)} \d y \d x\\
&\quad - 16 \int_{B_{r+\rho}} \U(x)g(\U(x)) \Gamma^{K_s}(\tau,\tau)(x) \d x,\\
&\ge - \frac{\delta_0}{16} \cE_{B_{r+\rho}}^{K_s}(\tau G(\U) , \tau G(\U)) - c \delta_0 \rho^{-\alpha} \Vert G(\U)^2 \Vert_{L^1(B_{r+\rho})} - c \rho^{-\alpha} \left\Vert \U g(\U) \right\Vert_{L^{1}(B_{r+\rho})}\\
&\quad -\delta_1\cE^{K_s}_{B_{r+\rho}}(\tau g(\U) G'(\U)^{-1} , \tau g(\U) G'(\U)^{-1}) -c[C(\delta_0\delta_1) + \delta_0\delta_1 \rho^{-\alpha}] \left\Vert \tau^2 \left(\frac{g(\U)}{G'(\U)} \right)^2 \right\Vert_{L^{1}(B_{r+\rho})},
\end{align*}
where $C(\delta_0\delta_1) > 0$ is the constant from \eqref{eq:K1consequence}, depending on the product $\delta_0\delta_1$, and we used \eqref{K1} and \eqref{eq:K1consequence} applied with $\delta := \delta_0\delta_1$ in the last step to estimate:
\begin{align*}
\delta_0^{-1}& \iint_M \left(\frac{g(\U(x))}{G'(\U(x))} \vee \frac{g(\U(y))}{G'(\U(y))} \right)^2 ( \tau^2(x) \wedge \tau^2(y) ) \frac{|K_a(x,y)|^2}{J(x,y)} \d y \d x\\
&\le \delta_1\cE^{K_s}_{B_{r+\rho}}(\tau g(\U) G'(\U)^{-1} , \tau g(\U) G'(\U)^{-1}) + c[C(\delta_0\delta_1) + \delta_0\delta_1 \rho^{-\alpha}] \left\Vert \tau^2 \left(\frac{g(\U)}{G'(\U)} \right)^2 \right\Vert_{L^{1}(B_{r+\rho})}
\end{align*}
and used \autoref{lemma:improvedK2help}, \eqref{K1}, \eqref{eq:genaux4} and \eqref{cutoff} to estimate
\begin{align}
\label{eq:absorbJtau}
\begin{split}
\delta_0\iint_M  &(G(\U(x)) - G(\U(y)))^2 ( \tau^2(x) \wedge \tau^2(y) ) J(x,y) \d y \d x \\
&\le c \delta_0 \int_{B_{r+\rho}}\int_{B_{r+\rho}} (G(\U(x)) - G(\U(y)))^2 ( \tau^2(x) \wedge \tau^2(y) ) K_s(x,y) \d y \d x\\
&\le c \delta_0 \cE_{B_{r+\rho}}^{K_s}(\tau G(\U) , \tau G(\U)) + c \delta_0 \rho^{-\alpha} \Vert G(\U)^2 \Vert_{L^1(B_{r+\rho})}.
\end{split}
\end{align}
Altogether, we obtain:
\begin{align*}
\cE_{B_{r+\rho}} & (u,-\tau^{2}g(\U)) \ge \left[\frac{\delta_0}{4} - \frac{\delta_0}{8}\right]\cE_{B_{r+\rho}}^{K_s}(\tau G(\U), \tau G(\U)) - \delta_1 \cE^{K_s}_{B_{r+\rho}}(\tau g(\U) G'(\U)^{-1} , \tau g(\U) G'(\U)^{-1}) \\
&- c\Bigg(\delta_0 C \Vert \tau^2 G(\U)^2 \Vert_{L^{1}(B_{r+\rho})} + \delta_0\rho^{-\alpha} \Vert G(\U)^2 \Vert_{L^{1}(B_{r+\rho})}   + [\delta_0^{-1} +\delta_0\delta_1]\rho^{-\alpha} \left\Vert \left(\frac{g(\U)}{G'(\U)} \right)^2 \right\Vert_{L^{1}(B_{r+\rho})} \\
&+ C(\delta_0\delta_1)\left\Vert \tau^2\left(\frac{g(\U)}{G'(\U)} \right)^2 \right\Vert_{L^{1}(B_{r+\rho})}+ \rho^{-\alpha} \left\Vert \U g(\U) \right\Vert_{L^{1}(B_{r+\rho})} \Bigg),
\end{align*}
The desired estimate \eqref{eq:MSgenstep1} follows.

Step 2: In addition, we claim
\begin{align}
\label{eq:MSgenstep2}
- \cE_{(B_{r+\rho} \times B_{r+\rho})^{c}}(u,-\tau^{2} g(\U)) \le c \rho^{-\alpha} \Vert \U g(\U) \Vert_{L^1(B_{r+\rho})}.
\end{align}
To see this, we compute using \eqref{eq:KaKs} and \eqref{cutoff}
\begin{align*}
- &\cE_{(B_{r+\rho} \times B_{r+\rho})^{c}}(u,-\tau^{2} g(\U)) = -2 \int_{(B_{r+\rho}\times B_{r+\rho})^{c}} (u(x)-u(y))(-\tau^{2}(x)g(\U(x))) K(x,y) \d y \d x\\
&= 2 \int_{B_{r+\rho}} \hspace*{-2ex} \tau^{2}(x)u(x)g(\U(x)) \left(\int_{B_{r+\rho}^{c}} \hspace*{-2ex} K(x,y) \d y \right) \d x -2 \int_{B_{r+\rho}} \hspace*{-2ex} \tau^{2}(x) g(\U(x)) \left(\int_{B_{r+\rho}^{c}} \hspace*{-2ex} u(y) K(x,y) \d y \right) \d x\\
&\le 2 \int_{B_{r+\rho}} \U(x) g(\U(x)) \Gamma^{K_s}(\tau,\tau)(x) \d x \le c \rho^{-\alpha} \Vert \U g(\U) \Vert_{L^1(B_{r+\rho})} ,
\end{align*}
using that $u,K \ge 0$, $u \le \U$.

Step 3: Observe that
\begin{align*}
\cE_{B_{r+\rho}}(u,-\tau^{2}g(\U)) &= \cE(u,-\tau^{2} g(\U)) - \cE_{(B_{r+\rho}\times B_{r+\rho})^{c}}(u,-\tau^{2} g(\U)).
\end{align*} 
Therefore, combining \eqref{eq:MSgenstep1}, \eqref{eq:MSgenstep2} yields 
\begin{align*}
&\cE^{K_s}_{B_{r+\rho}}(\tau G(\U),\tau G(\U)) \le c_1 \delta_0^{-1} \cE(u,-\tau^2 g(\U)) + c_1\delta_0^{-1}\delta_1\cE^{K_s}_{B_{r+\rho}}(\tau g(\U) G'(\U)^{-1} , \tau g(\U) G'(\U)^{-1})\\
&\qquad+c_2 \Bigg(C\Vert \tau^2 G(\U)^2 \Vert_{L^{1}(B_{r+\rho})} + \rho^{-\alpha}\Vert G(\U)^2 \Vert_{L^{1}(B_{r+\rho})} + [\delta_0^{-2}\rho^{-\alpha} +\delta_1 \rho^{-\alpha}]\left\Vert \left(\frac{g(\U)}{G'(\U)} \right)^2 \right\Vert_{L^{1}(B_{r+\rho})}\\
&\qquad\qquad+ \delta_0^{-1} C(\delta_0\delta_1)\left\Vert \tau^2\left(\frac{g(\U)}{G'(\U)} \right)^2 \right\Vert_{L^{1}(B_{r+\rho})}+ \delta_0^{-1}\rho^{-\alpha} \left\Vert \U g(\U) \right\Vert_{L^{1}(B_{r+\rho})} \Bigg),
\end{align*}
Finally, let us fix $p \ge 1 - \kappa^{-1}$ with $p \neq 1$ and define $g,G$ as in \eqref{eq:gGdef}. We apply \autoref{lemma:gGidentities} to deduce the desired result, choosing $\delta_0 = (1-\kappa^{-1})\frac{|p-1|}{p} \le 1$ and $\delta_1 = \frac{1}{8c_1} \delta_0^{-1} = \frac{(1-\kappa^{-1})p}{8c_1 |p-1|}$. Moreover, note that $\frac{|p-1|}{p} \le (1-\kappa^{-1})^{-1}$ by assumption on $p$. 
\end{proof}

Note that the aforementioned lemma does not use the equation, but only properties of the jumping kernel $K$ and the estimates \eqref{eq:genaux1}, \eqref{eq:genaux2}, \eqref{eq:genaux4} and \eqref{eq:genaux5}.

The following Caccioppoli-type estimate gives an estimate for the energy of the logarithm of a function $u$. It arises from testing the equation with $-\tau^2 (u+\eps)^{-1}$. Note that the  proof follows the general structure of the proof of \autoref{lemma:MSgen} but requires \eqref{K2}.

\begin{lemma} \label{lemma:MSgen2}
Assume that \eqref{K1}, \eqref{K2} and \eqref{cutoff} hold true for some $\theta \in [\frac{d}{\alpha},\infty]$. Moreover, assume \eqref{Sob} if $\theta < \infty$. Then there exist $c_1,c_2 > 0$ such that for every $0 < \rho \le r \le 1$ and every nonnegative function $u \in V(B_{r+\rho}|\R^d)$, and every $\eps > 0$, it holds:
\begin{align*}
c_1\int_{B_{r+\rho}} \int_{B_{r+\rho}} & (\tau^2(x) \wedge \tau^2(y)) \left(\log\frac{\U(x)}{\tau(x)}-\log\frac{\U(y)}{\tau(y)}\right)^2 K_s(x,y) \d y \d x \\ & \le \cE(u,-\tau^2 \U^{-1}) + c_2 \rho^{-\alpha}\vert B_{r+\rho}\vert,
\end{align*}
where $B_{2r} \subset \Omega$, $\tau = \tau_{r,\frac{\rho}{2}}$, and $\U = u + \eps$.
\end{lemma}

\begin{proof}
We will explain how to prove that there are $c_1, c_2 > 0$ such that
\begin{align}
\label{eq:genCacc2step1}
\begin{split}
\cE_{B_{r+\rho}}(u,-\tau^2 \U^{-1}) &\ge c_1 \int_{B_{r+\rho}}\int_{B_{r+\rho}} (\log \U(x) - \log \U(y))^2 (\tau^2(x) \wedge \tau^2(y))K_s(x,y) \d y \d x\\
&- c_2[C+\rho^{-\alpha}]|B_{r+\rho}|
\end{split}
\end{align}
Here, $C = C(1) > 0$ denotes the constant from \eqref{eq:K1consequence}.
This estimate can be seen as the counterpart to Step 1 in the proof of \autoref{lemma:MSgen} for $g,G$ as in \eqref{eq:gGdef} with $p = 1$.
Let us explain how to deduce the desired result from \eqref{eq:genCacc2step1}. First, observe that by the same argument as in Step 2 of the proof of \autoref{lemma:MSgen}, we have that
\begin{align*}
 - \cE_{(B_{r+\rho} \times B_{r+\rho})^{c}}(u,-\tau^{2} \U^{-1}) = - \cE_{(B_{r+\rho} \times B_{r+\rho})^{c}}(u,-\tau^{2} g(\U)) \le c \rho^{-\alpha} \Vert \U g(\U) \Vert_{L^1(B_{r+\rho})} = c \rho^{-\alpha} |B_{r+\rho}|,
\end{align*}
see \eqref{eq:MSgenstep2}. Thus, using \eqref{eq:logaux4} and \eqref{cutoff}, \eqref{eq:genCacc2step1} translates to
\begin{align*}
&\cE(u,-\tau^2 \U^{-1}) = \cE_{B_{r+\rho}}(u,-\tau^2 \U^{-1}) + \cE_{(B_{r+\rho} \times B_{r+\rho})^{c}}(u,-\tau^2 \U^{-1})\\
&\ge c_2\int_{B_{r+\rho}} \int_{B_{r+\rho}} (\tau^2(x) \wedge \tau^2(y)) \left(\log \U(x) -\log \U(y)\right)^2 K_s(x,y) \d y \d x - c_1[C+\rho^{-\alpha}]|B_{r+\rho}|\\
&\ge c_2\int_{B_{r+\rho}} \int_{B_{r+\rho}} (\tau^2(x) \wedge \tau^2(y)) \left(\log\frac{\U(x)}{\tau(x)}-\log\frac{\U(y)}{\tau(y)}\right)^2 K_s(x,y) \d y \d x - c_1\rho^{-\alpha}|B_{r+\rho}|.
\end{align*}
In the last step, we used that $\rho \le 1$.
It remains to prove \eqref{eq:genCacc2step1}. As in the proof of \autoref{lemma:MSgen}, we will first work in a general setting, using an abstract pair of functions $g,G$ satisfying the properties of \autoref{lemma:genaux}, before fixing $g,G$ as in \eqref{eq:gGdef} with $p = 1$ in order to deduce \eqref{eq:genCacc2step1}. 
\begin{align*}
\cE^{K_s}_{B_{r+\rho}}(u,-\tau^{2}g(\U)) = I_s + J_s, \qquad \cE^{K_a}_{B_{r+\rho}}(u,-\tau^{2}g(\U)) = I_a + J_a.
\end{align*}
To estimate $I_s + I_a$, we proceed as follows, using \eqref{eq:genaux1}, \eqref{K2}, \autoref{lemma:symmlemma} and \autoref{lemma:improvedK2help}:
\begin{align}
\label{eq:loguIsIa}
\begin{split}
I_s + I_a &= 2 \iint_{M} (\U(y) - \U(x))(g(\U(x)) - g(\U(y)))\tau^2(x) K(x,y) \d y \d x\\
&\ge 2\iint_{M} (G(\U(x)) - G(\U(y)))^2 (\tau^2(x) \wedge \tau^2(y) ) K(x,y) \d y \d x\\
&\ge (1-D)\int_{B_{r+\rho}}\int_{B_{r+\rho}} (G(\U(x)) - G(\U(y)))^2 (\tau^2(x) \wedge \tau^2(y) ) j(x,y) \d y \d x\\
&\ge B \int_{B_{r+\rho}} \int_{B_{r+\rho}} (G(\U(x)) - G(\U(y)))^2 (\tau^2(x) \wedge \tau^2(y) ) K_s(x,y) \d y \d x,
\end{split}
\end{align}
where $B > 0$ is a constant. To estimate $J_s$, we use \autoref{lemma:logusaviour}, \eqref{cutoff} and \eqref{eq:genaux2} to obtain that there is $c > 0$ such that for every $\delta > 0$:
\begin{align}
\label{eq:loguJs}
\begin{split}
J_s &\ge - 4\iint_{M} |\U(x) - \U(y)| g(\U(y)) |\tau(x) - \tau(y)| (\tau(x) \vee \tau(y)) K_s(x,y) \d y \d x\\
&\ge - \delta \iint_M  |\U(x) - \U(y)|^2 g^2(\U(y))( \tau^2(x) \vee \tau^2(y) ) K_s(x,y) \d y \d x - c\delta^{-1} \cE_{B_{r+\rho}}^{K_s}(\tau,\tau)\\
&\ge - 2\delta \iint_M |\U(x) - \U(y)|^2 g^2(\U(y)) ( \tau^2(x) \wedge \tau^2(y))K_s(x,y) \d y \d x\\
& \quad - 2\delta \iint_M  |\U(x) - \U(y)|^2 g^2(\U(y)) (\tau(x) - \tau(y))^2 K_s(x,y) \d y \d x - c\delta^{-1} \rho^{-\alpha} |B_{r+\rho}|\\
&\ge - 2\delta \int_{B_{r+\rho}}\int_{B_{r+\rho}} \hspace*{-2ex}(G(\U(x)) - G(\U(y)))^2 \left(\frac{g(\U(y))}{G'(\U(y))} \vee \frac{g(\U(y))}{G'(\U(y))} \right)^2( \tau^2(x) \wedge \tau^2(y))K_s(x,y) \d y \d x\\
& \quad - c\delta^{-1} \rho^{-\alpha} |B_{r+\rho}| - c \delta \rho^{-\alpha} \Vert \U^2 g^2(\U)\Vert_{L^1(B_{r+\rho})},
\end{split}
\end{align}
where we also used that $g(\U(y)) \le g(\U(x))$ on $M$ to apply \eqref{eq:genaux2}.\\
To estimate $J_a$, we observe that there exists $c > 0$ such that for every $\delta > 0$:
\begin{align*}
J_a &\ge - 4\iint_{M} |\U(y) - \U(x)|g(\U(y)) (\tau^2(x) \vee \tau^2(y)) |K_a(x,y)| \d y \d x\\
&\ge -\delta \iint_{M} |\U(x) - \U(y)|^2 g^2(\U(y)) (\tau^2(x) \vee \tau^2(y)) J(x,y) \d y \d x \\
& \quad - c \delta^{-1} \iint_{M} (\tau^2(x) \vee \tau^2(y))  \frac{|K_a(x,y)|^2}{J(x,y)} \d y \d x.
\end{align*}
Similar to the estimate for $J_s$, we obtain by \autoref{lemma:logusaviour}, \eqref{eq:genaux2}, \eqref{cutoff}, but also \eqref{K1} and \autoref{lemma:improvedK2help}:
\begin{align*}
\delta &\iint_{M} |\U(x) - \U(y)|^2 g^2(\U(y)) (\tau^2(x) \vee \tau^2(y)) J(x,y) \d y \d x\\
&\le 2\delta \iint_M (G(\U(x)) - G(\U(y)))^2 \left(\frac{g(\U(y))}{G'(\U(y))} \vee \frac{g(\U(y))}{G'(\U(y))} \right)^2( \tau^2(x) \wedge \tau^2(y))J(x,y) \d y \d x\\
& \quad - c \delta^{-1} \rho^{-\alpha} \Vert \U^2 g^2(\U)\Vert_{L^1(B_{r+\rho})}.
\end{align*}

Furthermore, by \eqref{K1}, \eqref{eq:K1consequence} applied with $\delta = 1$, and \eqref{cutoff}, we obtain
\begin{align*}
\iint_{M} \hspace*{-1ex}(\tau^2(x) \vee \tau^2(y)) \frac{|K_a(x,y)|^2}{J(x,y)} \d y \d x&\le  c \cE^{K_s}_{B_{r+\rho}}(\tau,\tau) + c[C +\rho^{-\alpha}] \int_{B_{r+\rho}} \hspace*{-3ex} \tau^2(x) \d x \le c[C +\rho^{-\alpha}] |B_{r+\rho}|,
\end{align*}
where $C = C(1) > 0$ denotes the constant from \eqref{eq:K1consequence}. Altogether, we have shown
\begin{align}
\label{eq:loguJa}
\begin{split}
J_a &\ge -2\delta \int_{B_{r+\rho}}\int_{B_{r+\rho}} \hspace*{-2ex}(G(\U(x)) - G(\U(y)))^2 \left(\frac{g(\U(y))}{G'(\U(y))} \vee \frac{g(\U(y))}{G'(\U(y))} \right)^2( \tau^2(x) \wedge \tau^2(y))J(x,y) \d y \d x\\
& \quad - c \delta^{-1}[C+\rho^{-\alpha}]|B_{r+\rho}|  -  c \delta \rho^{-\alpha} \Vert \U^2 g^2(\U)\Vert_{L^1(B_{r+\rho})},
\end{split}
\end{align}
and therefore, by combining the estimates for $I_s + I_a$, $J_s$ and $J_a$, choosing $g,G$ according to \eqref{eq:gGdef} for $p =1$ and applying \autoref{lemma:gGidentities}, we have proved \eqref{eq:genCacc2step1}, as desired after choosing $\delta > 0$ small enough. This concludes the proof.
\end{proof}

\subsection{Dual bilinear forms}\label{subsec:dual-forms}

The following Caccioppoli-type estimates are designed for the dual equation \eqref{PDEdual}. The general structure of the proofs in this section resembles the one from the previous section. However, particular care is required since the dual form does not satisfy 
\begin{align*}
\widehat{\cE}(u+\eps,v) = \widehat{\cE}(u,v).
\end{align*}

The following lemma is a counterpart to \autoref{lemma:MSgen} for the dual form.

\begin{lemma}
\label{lemma:MSgendual}
Assume that \eqref{K1glob} and \eqref{cutoff} hold true for some $\theta \in (\frac{d}{\alpha},\infty]$. Moreover, assume \eqref{Sob} if $\theta < \infty$. Then there are $c_1,c_2 > 0$, $\gamma > 1$ such that for every $0 < \rho \le r \le 1$ and every nonnegative function $u \in V(B_{r+\rho}|\R^d) \cap L^{2\theta'}(\R^d)$, and $p \ge 1 - \kappa^{-1}$ with $p \neq 1$, $\eps > 0$:
\begin{align*}
\cE^{K_s}_{B_{r+\rho}}(\tau \U^{\frac{-p+1}{2}},\tau \U^{\frac{-p+1}{2}}) \le c_1 |p-1| \widehat{\cE}(u,-\tau^{2}\U^{-p}) + c_2 \left(1 \vee p^{\gamma}\right)\rho^{-\alpha} \Vert \U^{-p+1}\Vert_{L^{1}(B_{r+\rho})},
\end{align*}
where $B_{2r} \subset \Omega$, $\tau = \tau_{r,\frac{\rho}{2}}$, and $\U = u + \eps$.
\end{lemma}

\begin{proof}
We fix a pair of functions $g,G$ satisfying the assumptions of \autoref{lemma:genaux}. Later, we will fix $p \ge 1 - \kappa^{-1}$ with $p \neq 1$ and define $g,G$ as in \eqref{eq:gGdef}. Let $M$ be as in the proof of \autoref{lemma:MSgen}. 

Step 1: We claim that there exist $c_1,c_2,c_3 > 0$ such that for every $\delta_0 \in (0,1)$, $\delta_1 , \delta_2 > 0$:
\begin{align}
\label{eq:dualMSgenstep1}
\begin{split}
&\widehat{\cE}_{B_{r+\rho}}(u,-\tau^2 g(\U)) \ge c_1 \delta_0 \cE^{K_s}_{B_{r+\rho}}(\tau G(\U),\tau G(\U))\\
&- c_2 \left(\delta_0^{-1} \delta_1 \cE^{K_s}_{B_{r+\rho}}(\tau \U G'(\U),\tau \U G'(\U)) + \delta_2 \cE^{K_s}_{B_{r+\rho}}(\tau g(\U)^{1/2}\U^{1/2}, \tau g(\U)^{1/2}\U^{1/2})\right)\\
&- c_3 \Bigg(\delta_0 C \Vert \tau^2 G(\U)^2 \Vert_{L^{1}(B_{r+\rho})} + \delta_0 \rho^{-\alpha} \Vert G(\U)^2 \Vert_{L^{1}(B_{r+\rho})} + \delta_0^{-1} \rho^{-\alpha} \left\Vert \left( \frac{g(\U)}{G'(\U)} \right)^2 \right\Vert_{L^1(B_{r+\rho})}\\
&+ \delta_0^{-1} \big[C(\delta_1) + \tfrac{\delta_1}{\rho^{\alpha}}\big] \Vert \tau^2 \U^2G'(\U)^2 \Vert_{L^{1}(B_{r+\rho})} + C(\delta_2)\Vert \tau^2 g(\U)\U \Vert_{L^{1}(B_{r+\rho})} + \tfrac{\delta_2 +1}{\rho^{\alpha}} \Vert g(\U)\U \Vert_{L^{1}(B_{r+\rho})}\Bigg).
\end{split}
\end{align}
Here, $C = C(1), C(\delta_1), C(\delta_2) > 0$ are respective constants from \eqref{eq:K1consequence}. We observe the following algebraic identity:
\begin{align*}
(a+b)(\tau_2^2 g(\widetilde{b}) - \tau_1^2 g(\widetilde{a})) = (b-a)(g(\widetilde{b}) -  g(\widetilde{a}))\tau_1^2 +2a(g(\widetilde{b})- g(\widetilde{a}))\tau_1^2 + (a+b) g(\widetilde{b}) (\tau_2^2 - \tau_1^2).
\end{align*}
We use \autoref{lemma:symmlemma} to estimate
\begin{align*}
\widehat{\cE}_{B_{r+\rho}}^{K_a}(u,-\tau^{2}g(\U)) &= 2\iint_M (u(x) + u(y)) (\tau^{2}(y)g(\U(y)) - \tau^{2}(x)g(\U(x))) K_a(x,y) \d y \d x\\
&= 2\iint_M (\U(y) - \U(x))(g(\U(y)) - g(\U(x)))\tau^2(x)  K_a(x,y) \d y \d x\\
& \quad + 4 \iint_M u(x) (g(\U(y)) - g(\U(x))) \tau^2(x) K_a(x,y) \d y \d x\\
& \quad +  4 \iint_M (u(x) + u(y)) g(\U(y)) (\tau^2(y) - \tau^2(x))  K_a(x,y) \d y \d x\\
&\ge 2\iint_M  (\U(y) - \U(x))(g(\U(x)) - g(\U(y)))\tau^2(x)  K_a(y,x) \d y \d x\\
& \quad - 4 \iint_M \U(x) |g(\U(x)) - g(\U(y))| \tau^2(x) |K_a(x,y)| \d y \d x\\
& \quad -  8 \iint_M g(\U(y))\U(y) |\tau^2(x) - \tau^2(y)|  |K_a(x,y)| \d y \d x\\
&= I_a + M_a + N_a,
\end{align*}
where we used that $g(\U(x)) \ge g(\U(y))$ on $M$, which implies that $\U(y) \ge \U(x)$ since $g$ is decreasing, as well as $u \le \U$.
As in the proof of \autoref{lemma:MSgen}, we can decompose 
\begin{align*}
\cE^{K_s}_{B_{r+\rho}}(u,-\tau^2 g(\U)) = I_s + J_s.
\end{align*}
The estimate of $J_s$ goes as in the proof of \autoref{lemma:MSgen}.
Moreover, using the same arguments as in the proof of \autoref{lemma:MSgen}, we estimate for any $\delta_0 \in (0,1)$:
\begin{align*}
I_s + I_a &= 2\iint_M (\U(y) - \U(x))(g(\U(x)) - g(\U(y)))\tau^2(x) K(y,x) \d y \d x\\
& \ge \frac{\delta_0}{4}\cE^{K_s}_{B_{r+\rho}}(\tau G(\U),\tau G(\U)) - c \delta_0 \left( C \Vert \tau^2 G(\U)^2 \Vert_{L^{1}(B_{r+\rho})} + \rho^{-\alpha}\Vert G(\U)^2 \Vert_{L^{1}(B_{r+\rho})} \right). 
\end{align*}

It remains to estimate $M_a$ and $N_a$:\\
For $M_a$, we obtain using, \eqref{eq:genaux3}, \eqref{eq:genaux5} and \eqref{cutoff} that for every $\delta_1 >0$:
\begin{align*}
M_a &\ge - 4 \iint_M \left|G(\U(x)) - G(\U(y)) \right| \left(\U(x) G'(\U(x)) \vee \U(y) G'(\U(y)) \right) \tau^2(x) |K_a(x,y)| \d y \d x\\
&\ge - \delta_0 \iint_M (G(\U(x)) - G(\U(y)) )^2 (\tau^2(y) \vee \tau^2(x))  J(x,y) \d y \d x\\
& \quad - c\delta_0^{-1} \iint_M \tau^2(x)\U^2(x) G'(\U(x))^2 \frac{|K_a(x,y)|^2}{J(x,y)} \d y \d x\\
&\ge - c \delta_0 \cE^{K_s}_{B_{r+\rho}}(\tau G(\U) , \tau G(\U)) - c \delta_0 \rho^{-\alpha} \Vert G(\U)^2 \Vert_{L^{1}(B_{r+\rho})}\\
& \quad - c\delta_0^{-1} \delta_1\cE^{K_s}_{B_{r+\rho}}(\tau \U G'(\U) , \tau \U G'(\U)) - c\delta_0^{-1} [C(\delta_1) + \delta_1 \rho^{-\alpha}]\Vert \tau^2 \U^2 G'(\U)^2 \Vert_{L^{1}(B_{r+\rho})},
\end{align*}

where we used that, by \eqref{K1glob} and \eqref{eq:quantifiedK1consequence} applied with some $\delta_1 > 0$,
\begin{align}
\label{eq:applyquantifiedK1consequence}
\begin{split}
\delta_0^{-1} \iint_M &\tau^2(x)\U^2(x) G'(\U(x))^2  \frac{|K_a(x,y)|^2}{J(x,y)} \d y \d x\\
&\le \delta_0^{-1} \delta_1\cE^{K_s}_{B_{r+\rho}}(\tau \U G'(\U) , \tau \U G'(\U)) + c\delta_0^{-1} [C(\delta_1) + \delta_1 \rho^{-\alpha}] \Vert \tau^2 \U^2 G'(\U)^2 \Vert_{L^{1}(B_{r+\rho})}.
\end{split}
\end{align}
Here, $C(\delta_1) > 0$ denotes the constant in \eqref{eq:quantifiedK1consequence}. Moreover, by \eqref{cutoff}, \eqref{eq:genaux5} and \eqref{K1glob}:
\begin{align*}
\delta_0 \iint_M & (G(\U(x)) - G(\U(y)) )^2 (\tau^2(y) \vee \tau^2(x)) J(x,y) \d y \d x\\
&\le 2\delta_0 \cE^{J}_{B_{r+\rho}}(\tau G(\U) , \tau G(\U)) + c \delta_0 \rho^{-\alpha}\Vert G(\U)^2 \Vert_{L^1(B_{r+\rho})}\\
&\le c\delta_0 \cE^{K_s}_{B_{r+\rho}}(\tau G(\U) , \tau G(\U)) + c \delta_0 \rho^{-\alpha}\Vert G(\U)^2 \Vert_{L^1(B_{r+\rho})}.
\end{align*}
For $N_a$, we compute using \autoref{lemma:logusaviour} and \eqref{eq:KaKs} that for every $\delta_2 > 0$:
\begin{align*}
N_a &= -8 \iint_M g(\U(y))\U(y)|\tau(x)-\tau(y)|(\tau(x) + \tau(y)) |K_a(x,y)| \d y \d x \\
&\ge -c \iint_M g(\U(y))\U(y)(\tau(x)-\tau(y))^2 K_s(x,y) \d y \d x\\
& \quad -c \iint_M g(\U(y))\U(y) (\tau(x) \wedge \tau(y))|\tau(x)-\tau(y)| |K_a(x,y)| \d y \d x\\
&\ge -c \rho^{-\alpha} \Vert g(\U)\U \Vert_{L^{1}(B_{r+\rho})}- c\int_{B_{r+\rho}} \int_{B_{r+\rho}} \tau^2(x)g(\U(x))\U(x) \frac{|K_a(x,y)|^2}{J(x,y)} \d y \d x\\
&\ge- c C(\delta_2) \Vert \tau^2 g(\U)\U \Vert_{L^{1}(B_{r+\rho})} -c (\delta_2 +1)\rho^{-\alpha} \Vert g(\U)\U \Vert_{L^{1}(B_{r+\rho})} \\
& \quad - c \delta_2 \cE^{K_s}_{B_{r+\rho}}(\tau g(\U)^{1/2}\U^{1/2}, \tau g(\U)^{1/2}\U^{1/2}),
\end{align*}
where we applied \eqref{cutoff} and used the same argument as in \eqref{eq:applyquantifiedK1consequence} to estimate the second summand in the last step. Here, $C(\delta_2) > 0$ is the constant from \eqref{eq:quantifiedK1consequence}.
Altogether, we have
\begin{align*}
&\widehat{\cE}_{B_{r+\rho}}(u,-\tau^2 g(\U)) \ge \left[ \frac{\delta_0}{4} - \frac{\delta_0}{8} \right] \cE^{K_s}_{B_{r+\rho}}(\tau G(\U),\tau G(\U))\\
&- \delta_0^{-1} \delta_1 \cE^{K_s}_{B_{r+\rho}}(\tau \U G'(\U),\tau \U G'(\U)) - c \delta_2 \cE^{K_s}_{B_{r+\rho}}(\tau g(\U)^{1/2}\U^{1/2}, \tau g(\U)^{1/2}\U^{1/2})\\
&- c \Bigg(\delta_0 C \Vert \tau^2 G(\U)^2 \Vert_{L^{1}(B_{r+\rho})} + \delta_0 \rho^{-\alpha} \Vert G(\U)^2 \Vert_{L^{1}(B_{r+\rho})} + \delta_0^{-1} \rho^{-\alpha} \left\Vert \left( \frac{g(\U)}{G'(\U)} \right)^2 \right\Vert_{L^1(B_{r+\rho})}\\
&+ \delta_0^{-1} [C(\delta_1) + \frac{\delta_1}{\rho^{\alpha}}] \Vert \tau^2 \U^2G'(\U)^2 \Vert_{L^{1}(B_{r+\rho})} + C(\delta_2)\Vert \tau^2 g(\U)\U \Vert_{L^{1}(B_{r+\rho})} + \frac{ \delta_2 +1}{\rho^{\alpha}} \Vert g(\U)\U \Vert_{L^{1}(B_{r+\rho})}\Bigg).
\end{align*}
This yields \eqref{eq:dualMSgenstep1}, as desired.

Step 2: Moreover, it holds
\begin{align}
\label{eq:dualMSgenstep2}
- \widehat{\cE}_{(B_{r+\rho}\times B_{r+\rho})^{c}}(u,-\tau^2 g(\U)) \le  c\rho^{-\alpha} \Vert \U g(\U) \Vert_{L^1(B_{r+\rho})}.
\end{align}
The proof works similar to the proof of Step 2 in \autoref{lemma:MSgen}.

Step 3: Combining \eqref{eq:dualMSgenstep1} and \eqref{eq:dualMSgenstep2}, we have proved that
\begin{align*}
&\cE^{K_s}_{B_{r+\rho}}(\tau G(\U),\tau G(\U)) \le c_1 \delta_0^{-1} \widehat{\cE}(u,-\tau^{2}g(\U))\\
&\qquad+ c_2 \delta_0^{-2} \delta_1 \cE^{K_s}_{B_{r+\rho}}(\tau \U G'(\U),\tau \U G'(\U)) + c_2 \delta_0^{-1} \delta_2 \cE^{K_s}_{B_{r+\rho}}(\tau g(\U)^{1/2}\U^{1/2}, \tau g(\U)^{1/2}\U^{1/2})\\
&\qquad+ c_3 \Bigg(C \Vert \tau^2 G(\U)^2 \Vert_{L^{1}(B_{r+\rho})} + \rho^{-\alpha} \Vert G(\U)^2 \Vert_{L^{1}(B_{r+\rho})} + \delta_0^{-2} \rho^{-\alpha} \left\Vert \left( \frac{g(\U)}{G'(\U)} \right)^2 \right\Vert_{L^1(B_{r+\rho})}\\
&\qquad\qquad\qquad+ \delta_0^{-2} [C(\delta_1) + \delta_1 \rho^{-\alpha}] \Vert \tau^2 \U^2G'(\U)^2 \Vert_{L^{1}(B_{r+\rho})}\\
&\qquad\qquad\qquad+ \delta_0^{-1}C(\delta_2)\Vert \tau^2 g(\U)\U \Vert_{L^{1}(B_{r+\rho})} + \delta_0^{-1}(\delta_2 +1)\rho^{-\alpha} \Vert g(\U)\U \Vert_{L^{1}(B_{r+\rho})}\Bigg).
\end{align*}
Fixing $p > 0$ with $p \neq 1$, and defining $g,G$ as in \eqref{eq:gGdef}, we deduce the desired result from \autoref{lemma:gGidentities}, choosing $\gamma = \frac{d}{d-\theta \alpha}$, $\delta_0 = (1-\kappa^{-1}) \frac{|p-1|}{p} \le 1$, $\delta_1 = c_4\frac{1}{p^2}$ and $\delta_2 = c_5\frac{1}{|p-1|}$ for some small enough constants $c_4,c_5 > 0$, using that $\frac{|p-1|}{p} \le (1-\kappa^{-1})^{-1}$ by assumption on $p$.
\end{proof}

The following lemma is a counterpart to \autoref{lemma:MSgen2} for the dual form.

\begin{lemma}
\label{lemma:MSgendual2}
Assume that \eqref{K1glob}, \eqref{K2} and \eqref{cutoff} hold true for some $\theta \in [\frac{d}{\alpha},\infty]$. Moreover, assume \eqref{Sob} if $\theta < \infty$. Then there are $c_1,c_2 > 0$ such that for every $0 < \rho \le r \le 1$ and every nonnegative function $u \in V(B_{r+\rho}|\R^d) \cap L^{2\theta'}(\R^d)$, and every $\eps > 0$
\begin{align*}
c_1\int_{B_{r+\rho}} \int_{B_{r+\rho}} \hspace*{-2ex} (\tau^2(x) \wedge \tau^2(y)) \left(\log\frac{\U(x)}{\tau(x)}-\log\frac{\U(y)}{\tau(y)}\right)^2 K_s(x,y) \d y \d x \le \widehat{\cE}(u,-\tau^2 \U^{-1}) + c_2 \rho^{-\alpha}\vert B_{r+\rho}\vert,
\end{align*}
where $B_{2r} \subset \Omega$, $\tau = \tau_{r,\frac{\rho}{2}}$, and $\U = u + \eps$.
\end{lemma}

\begin{proof}
The proof is a combination of the proofs of \autoref{lemma:MSgen2} and \autoref{lemma:MSgendual}.
We will explain how to prove that there are $c_1, c_2 > 0$ such that
\begin{align}
\label{eq:genCacc2dualstep1}
\begin{split}
\widehat{\cE}_{B_{r+\rho}}(u,-\tau^2 \U^{-1}) &\ge c_1 \int_{B_{r+\rho}}\int_{B_{r+\rho}} (\log \U(x) - \log \U(y))^2 (\tau^2(x) \wedge \tau^2(y))K_s(x,y) \d y \d x\\
&- c_2[C+\rho^{-\alpha}]|B_{r+\rho}|,
\end{split}
\end{align}
where $C = C(1) > 0$ is the constant from \eqref{eq:K1consequence}.
Note that by the same argument as in Step 2 of the proof of \autoref{lemma:MSgen}, one can prove that
\begin{align*}
 - \widehat{\cE}_{(B_{r+\rho} \times B_{r+\rho})^{c}}(u,-\tau^{2} \U^{-1}) = - \widehat{\cE}_{(B_{r+\rho} \times B_{r+\rho})^{c}}(u,-\tau^{2} g(\U)) \le c \rho^{-\alpha} \Vert \U g(\U) \Vert_{L^1(B_{r+\rho})} = c \rho^{-\alpha} |B_{r+\rho}|.
\end{align*}
By combining these two estimates, one easily deduces the desired result using \eqref{eq:logaux4} and \eqref{cutoff}, as in the proof of \autoref{lemma:MSgendual}.

It remains to prove \eqref{eq:genCacc2dualstep1}. First, we estimate
\begin{align*}
\widehat{\cE}_{B_{r+\rho}}^{K_s}(u,-\tau^{2}g(\U)) = I_s + J_s, \qquad \widehat{\cE}_{B_{r+\rho}}^{K_a}(u,-\tau^{2}g(\U)) \ge I_a + M_a + N_a,
\end{align*}
as in the proof of \autoref{lemma:MSgendual}. The estimates of $J_s$ and $I_s + I_a$ work as in the proof of \autoref{lemma:MSgendual} and yield that there are $c,B > 0$ such that for every $\delta > 0$
\begin{align*}
I_s + I_a &\ge B \int_{B_{r+\rho}} \int_{B_{r+\rho}} \hspace*{-2ex} (G(\U(x)) - G(\U(y)))^2 (\tau^2(x) \wedge \tau^2(y) ) K_s(x,y) \d y \d x,
\end{align*}
\begin{align*}
J_s &\ge - 2\delta \int_{B_{r+\rho}}\int_{B_{r+\rho}} \hspace*{-2ex} (G(\U(x)) - G(\U(y)))^2 \left(\frac{g(\U(y))}{G'(\U(y))} \vee \frac{g(\U(y))}{G'(\U(y))} \right)^2( \tau^2(x) \wedge \tau^2(y))K_s(x,y) \d y \d x \\
& \quad - c \delta^{-1} \rho^{-\alpha} |B_{r+\rho}| - c \delta \rho^{-\alpha} \Vert \U^2 g^2(\U)\Vert_{L^1(B_{r+\rho})},
\end{align*}
Recall that the estimate of $I_s + I_a$ requires \eqref{K2}.\\
For $M_a$, we observe that there exists $c > 0$ such that for every $\delta > 0$:
\begin{align*}
M_a &\ge - \delta\iint_M \left( \U(x) |g(\U(x)) - g(\U(y))| \right)^2 (\tau^2(x) \vee \tau^2(y)) J(x,y) \d y \d x\\
& \quad - c\delta^{-1} \iint_M (\tau^2(x) \vee \tau^2(y)) \frac{|K_a(x,y)|^2}{J(x,y)} \d y \d x.
\end{align*}
By \eqref{K1glob}, \eqref{eq:K1consequence} applied with $\delta = 1$ and \eqref{cutoff}, we have
\begin{align}
\label{eq:K1logu}
\iint_M (\tau^2(x) \vee \tau^2(y)) \frac{|K_a(x,y)|^2}{J(x,y)} \d y \d x &\le c \cE^{K_s}_{B_{r+\rho}}(\tau,\tau) + c \rho^{-\alpha} \int_{B_{r+\rho}} \tau^2(x) \d x \le c [C+\rho^{-\alpha}] |B_{r+\rho}|,
\end{align}
where $C = C(1) > 0$ denotes the constant from \eqref{eq:K1consequence}. Therefore, by \autoref{lemma:logusaviour}:
\begin{align*}
M_a &\ge -\delta \iint_M \left( \U(x) |g(\U(x)) - g(\U(y))| \right)^2 (\tau^2(x) \vee \tau^2(y)) J(x,y) \d y \d x - c(\delta) \rho^{-\alpha} |B_{r+\rho}|\\
&\ge -2\delta\iint_M \left( \U(x) |g(\U(x)) - g(\U(y))| \right)^2 (\tau^2(x) \wedge \tau^2(y))J(x,y) \d y \d x\\
& \quad -2\delta \iint_M \left( \U(x) |g(\U(x)) - g(\U(y))| \right)^2 (\tau(x) - \tau(y))^2 J(x,y) \d y \d x -c \delta^{-1} [C+\rho^{-\alpha}] |B_{r+\rho}|.
\end{align*}
We deduce using \eqref{eq:genaux3}, \autoref{lemma:improvedK2help} and \eqref{cutoff}:
\begin{align*}
M_a &\ge -2\delta \iint_M (G(\U(x)) - G(\U(y)))^2 \left(\U(x)G'(\U(x)) \vee \U(y)G'(\U(y)) \right) (\tau^2(x) \wedge \tau^2(y)) J(x,y) \d y \d x\\
& \quad - 2\delta \rho^{-\alpha}\Vert \U^2 G'(\U)^2 \Vert_{L^1(B_{r+\rho})} - c\delta^{-1} [C+\rho^{-\alpha}] \vert B_{r+\rho} \vert.
\end{align*}

For $N_a$, we compute using the same argument as in the proof of \autoref{lemma:MSgendual} that:
\begin{align*}
N_a &\ge -c \rho^{-\alpha} \Vert g(\U)\U \Vert_{L^{1}(B_{r+\rho})}- c\int_{B_{r+\rho}} \int_{B_{r+\rho}} \tau^2(x)g(\U(x))\U(x) \frac{|K_a(x,y)|^2}{J(x,y)} \d y \d x\\
&\ge- c [C + \rho^{-\alpha}] \Vert g(\U)\U \Vert_{L^{1}(B_{r+\rho})} - c \cE^{K_s}_{B_{r+\rho}}(\tau g(\U)^{1/2}\U^{1/2}, \tau g(\U)^{1/2}\U^{1/2}),
\end{align*}
where we applied \eqref{eq:K1consequence} to estimate the second summand in the last step and $C = C(1)$ is the constant from \eqref{eq:quantifiedK1consequence}.

Altogether, we have shown
\begin{align*}
&\widehat{\cE}_{B_{r+\rho}}(u,-\tau^{2}g(\U)) \ge B \int_{B_{r+\rho}}\int_{B_{r+\rho}} (G(\U(x)) - G(\U(y)))^2 (\tau^2(x) \wedge \tau^2(y)) K_s(x,y) \d y \d x\\
&-2\delta\int_{B_{r+\rho}}\int_{B_{r+\rho}} \hspace*{-3ex} (G(\U(x)) - G(\U(y)))^2 \left(\U(x)G'(\U(x)) \vee \U(y)G'(\U(y)) \right) (\tau^2(x) \wedge \tau^2(y)) K_s(x,y) \d y \d x\\
&-2\delta\int_{B_{r+\rho}}\int_{B_{r+\rho}} \hspace*{-3ex} (G(\U(x)) - G(\U(y)))^2 \left(\U(x)G'(\U(x)) \vee \U(y)G'(\U(y)) \right) (\tau^2(x) \wedge \tau^2(y)) J(x,y) \d y \d x\\
&- c \left([C+\rho^{-\alpha}] \Big[ |B_{r+\rho}| + \Vert \U g(\U) \Vert_{L^1(B_{r+\rho})} \Big] + \delta \rho^{-\alpha}  \Big[ \Vert \U^2 g(\U)^2 \Vert_{L^1(B_{r+\rho})} + \Vert \U^2 G'(\U)^2 \Vert_{L^1(B_{r+\rho})} \Big] \right)\\
&- c \cE^{K_s}_{B_{r+\rho}}(\tau g(\U)^{1/2}\U^{1/2}, \tau g(\U)^{1/2}\U^{1/2}).
\end{align*}

Finally, choosing $\delta > 0$ small enough, we deduce the desired result after choosing $g,G$ as in \eqref{eq:gGdef} with $p = 1$ and applying \autoref{lemma:gGidentities} and \autoref{lemma:improvedK2help}.
\end{proof}

In order to prove H\"older estimates for solutions to \eqref{PDEdual}, we need to show a weak parabolic Harnack inequality for weak supersolutions to \eqref{PDEdualext}. For this purpose, it remains to prove estimates for quantities of the form $\widehat{\cE}(u,-\tau^2g(\U)) + \widehat{\cE}^{K_a}(d,-\tau^2 g(\U))$, where $g(t) = t^{-p}$. The following two lemmas distinguish between the cases $p \neq 1$ and $p = 1$.

\begin{lemma}
\label{lemma:MSgendualext}
Assume that \eqref{K1glob} and \eqref{cutoff} hold true for some $\theta \in (\frac{d}{\alpha},\infty]$. Moreover, assume \eqref{Sob} if $\theta < \infty$. Then, there is $\gamma \ge 1$ such that for every $\delta \in (0,1)$, there are $c_1,c_2 > 0$, with the property that for every $0 < \rho \le r \le 1$ and every nonnegative function $u \in V(B_{r+\rho}|\R^d)$,  and $p \ge 1 - \kappa^{-1}$ with $p \neq 1$, $\eps > 0$, $d \in L^{\infty}(\R^d)$:
\begin{align*}
-\delta \cE^{K_s}_{B_{r+\rho}}(\tau \U^{\frac{-p+1}{2}},\tau \U^{\frac{-p+1}{2}}) &\le c_1 |p-1| \widehat{\cE}^{K_a}(d,-\tau^{2} \U^{-p}) + c_2 \left( 1 \vee p^{\gamma} \right)\rho^{-\alpha} \Vert \U^{-p+1}\Vert_{L^1(B_{r+\rho})},
\end{align*}
where $B_{2r} \subset \Omega$, $\tau = \tau_{r,\frac{\rho}{2}}$, and $\U = u + \eps + r^{\frac{1}{2}\left( \alpha - \frac{d}{\theta} \right)}\Vert d \Vert_{\infty}$.
\end{lemma}

\begin{proof}
Throughout the proof, we define $\eta:= \frac{1}{2}\left(\alpha - \frac{d}{\theta}\right)$. We will often use that by definition of $\U$: $\Vert d \Vert_{\infty} \le \rho^{-\eta} \U$. Moreover, we fix a pair of function $g,G$ satisfying the assumptions of \autoref{lemma:genaux} and such that $G'$ is decreasing. Later, we will fix $p \ge 1 - \kappa^{-1}$ with $p \neq 1$ and define $g,G$ as in \eqref{eq:gGdef}. We define $M$ as in the proof of \autoref{lemma:MSgen}.

Step 1: We claim that there exists $c > 0$ such that for every $\delta_0 \in (0,1)$, $\delta_2,\delta_3 > 0$:
\begin{align}
\label{eq:MSdualextstep1}
\begin{split}
&\widehat{\cE}^{K_a}_{B_{r+\rho}}(d,-\tau^{2}g(\U)) \\
&\ge -\delta_0 \cE^{K_s}_{B_{r+\rho}}(\tau G(\U),\tau G(\U)) - \delta_1 \cE^{K_s}_{B_{r+\rho}}(\tau \U G'(\U) , \tau \U G'(\U)) - \delta_2 \cE^{K_s}_{B_{r+\rho}}(\tau g(\U)^{1/2}\U^{1/2} , \tau g(\U)^{1/2}\U^{1/2})\\
&-c \Big( \delta_0 \rho^{-\alpha}\Vert G(\U)^2 \Vert_{L^1(B_{r+\rho})}  +\left[C(\delta_0\delta_1) + \delta_0 \delta_1\right]\rho^{-\alpha} \Vert \tau^2 \U^2 G'(\U)^2 \Vert_{L^1(B_{r+\rho})}\\
&\qquad\qquad+ C(\delta_2)\rho^{-\alpha} \Vert \tau^2 g(\U)\U\Vert_{L^1(B_{r+\rho})} + [1 + \delta_2]\rho^{-\alpha} \Vert g(\U)\U\Vert_{L^1(B_{r+\rho})} \Big).
\end{split}
\end{align}
Here, $C(\delta_0\delta_1), C(\delta_2) > 0$ denote respective constants from \eqref{eq:quantifiedK1consequence}.
We compute
\begin{align*}
\widehat{\cE}^{K_a}_{B_{r+\rho}}(d,-\tau^2 g(\U)) &\ge - 2\Vert d \Vert_{\infty} \iint_{M} |\tau^2 g(\U(x)) - \tau^2 g(\U(y))| |K_a(x,y)| \d y \d x\\
&= - 2\Vert d \Vert_{\infty} \iint_{M} \tau^2(x) |g(\U(x)) - g(\U(y))| |K_a(x,y)| \d y \d x\\
& \quad - 2\Vert d \Vert_{\infty} \iint_{M} g(\U(y)) |\tau^2(x) - \tau^2(y)| |K_a(x,y)| \d y \d x\\
&=: I_1 + I_2.
\end{align*}
Note that on $M$ it holds: $(\U(x) \wedge \U(y))(G'(\U(x)) \vee G'(\U(y))) = \U(x)G'(\U(x))$ since $G',g$ are decreasing. Thus, using \eqref{eq:genaux3} we obtain that for every $\delta_0 \in (0,1)$:
\begin{align*}
I_1 &\ge -2\rho^{-\eta} \iint_M \left|G(\U(x)) - G(\U(y)) \right| \U(x)G'(\U(x)) \tau^2(x) |K_a(x,y)| \d y \d x\\
&\ge - \delta_0 \iint_M (G(\U(x)) - G(\U(y)) )^2 (\tau^2(y) \vee \tau^2(x))  J(x,y) \d y \d x\\
& \quad - c \delta_0^{-1} \rho^{-2\eta} \iint_M \tau^2(x)\U^2(x)G'(\U(x))^2 \frac{|K_a(x,y)|^2}{J(x,y)} \d y \d x.
\end{align*}
With the help of \eqref{eq:genaux5}, \eqref{K1glob} and \eqref{cutoff}, we estimate the first summand from below by
\begin{align*}
 - 2\delta_0 \cE^{K_s}_{B_{r+\rho}}(\tau G(\U) , \tau G(\U)) - c \delta_0 \rho^{-\alpha} \Vert G(\U)^2 \Vert_{L^{1}(B_{r+\rho})},
\end{align*}
Moreover, by definition of $\eta$, \eqref{K1glob} and \eqref{eq:quantifiedK1consequence} applied with $\delta = \frac{\delta_0\delta_1}{c} \rho^{2\eta}$, $C(\delta) \le c C(\delta_0\delta_1) \rho^{-\frac{d}{\theta}}$ for some $\delta_1 \in (0,1)$, where $C(\delta), C(\delta_0\delta_1) > 0$ denote the constants from \eqref{eq:quantifiedK1consequence}:
\begin{align}
\label{eq:MSdualexthelp1}
\begin{split}
c \delta_0^{-1} &\rho^{-2\eta} \iint_M \tau^2(x)\U^{2}(x)G'(\U(x))^2 \frac{|K_a(x,y)|^2}{J(x,y)} \d y \d x\\
&\le \delta_1 \cE^{K_s}_{B_{r+\rho}}(\tau \U G'(\U) , \tau \U G'(\U)) + c \left[C(\delta_0\delta_1)\rho^{-\alpha} + \delta_0 \delta_1  \rho^{-\alpha}\right] \Vert \tau^2 \U^2 G'(\U)^2 \Vert_{L^1(B_{r+\rho})}.
\end{split}
\end{align}
This provides the estimate for $I_1$. To estimate $I_2$, we use \eqref{K1glob}, \eqref{cutoff}, and the fact that $g(\U(y))(\U(x) \wedge \U(y)) \le g(\U(x))\U(x) \wedge g(\U(y))\U(y)$ to show that for every $\delta_2 > 0$:
\begin{align*}
I_2 &\ge -2 \rho^{-\eta} \int_{B_{r+\rho}} \int_{B_{r+\rho}} (\U(x) \wedge \U(y))g(\U(y)) |\tau^2(x) - \tau^2(y)| |K_a(x,y)| \d y \d x\\
&\ge -2 \int_{B_{r+\rho}} \int_{B_{r+\rho}}\U(y) g(\U(y)) (\tau(x) - \tau(y))^2 J(x,y) \d y \d x\\
& \quad -2 \rho^{-2\eta} \int_{B_{r+\rho}} \int_{B_{r+\rho}}(\U(x) \wedge \U(y)) g(\U(y)) (\tau^2(x) + \tau^2(y)) \frac{|K_a(x,y)|^2}{J(x,y)} \d y \d x\\
&\ge -c\rho^{-\alpha}\Vert g(\U)\U \Vert_{L^1(B_{r+\rho})} - c\rho^{-2\eta} \int_{B_{r+\rho}} \int_{B_{r+\rho}} \tau^2(y)g(\U(y))\U(y) \frac{|K_a(x,y)|^2}{J(x,y)} \d y \d x\\
&\ge -c \rho^{-\alpha} \left(C(\delta_2) \Vert \tau^2 g(\U)\U\Vert_{L^1(B_{r+\rho})} + [1 + \delta_2]\Vert g(\U)\U\Vert_{L^1(B_{r+\rho})} \right) \\
& \quad - \delta_2 \cE^{K_s}_{B_{r+\rho}}(\tau g(\U)^{1/2}\U^{1/2} , \tau g(\U)^{1/2}\U^{1/2}).
\end{align*}
Note that in the last step, we applied a similar argument as in \eqref{eq:MSdualexthelp1}, this time applying \eqref{eq:quantifiedK1consequence} with $\delta = \frac{\delta_2}{c} \rho^{2\eta}$, $C(\delta) \le c C(\delta_2) \rho^{-\frac{d}{\theta}}$.
This proves the desired estimate.

Step 2: Note that for every $\delta_2 > 0$:
\begin{align}
\label{eq:MSdualextstep2}
\begin{split}
&\widehat{\cE}^{K_a}_{(B_{r+\rho} \times B_{r+\rho})^{c}}(d,-\tau^2 g(\U)) \ge - 2\Vert d \Vert_{\infty} \int_{B_{r+\rho}} \int_{B_{r+\rho}^c} \tau^2(x) g(\U(x)) |K_a(x,y)| \d y \d x\\
&\ge -2 \rho^{-\eta}\int_{B_{r+\rho}} \int_{B_{r+\rho}^c} g(\U(x))\U(x) |\tau^2(x) - \tau^2(y)||K_a(x,y)| \d y \d x\\
&\ge -2\int_{B_{r+\rho}} \int_{B_{r+\rho}^c} g(\U(x))\U(x)(\tau(x) - \tau(y))^2 J(x,y) \d y \d x\\
& \quad - 2\rho^{-2\eta} \int_{B_{r+\rho}}\tau^2(x)g(\U(x))\U(x) \int_{B_{r+\rho}^c}  \frac{|K_a(x,y)|^2}{J(x,y)} \d y \d x\\
&\ge -c \rho^{-\alpha} \left(C(\delta_2) \Vert \tau^2 g(\U)\U\Vert_{L^1(B_{r+\rho})} + \delta_2\Vert g(\U)\U\Vert_{L^1(B_{r+\rho})} \right) - \delta_2 \cE^{K_s}_{B_{r+\rho}}(\tau g(\U)^{1/2}\U^{1/2} , \tau g(\U)^{1/2}\U^{1/2}).
\end{split}
\end{align}
Here, we used the same arguments as in the estimate of $I_2$ and \eqref{K1glob}.

By combining \eqref{eq:MSdualextstep1} and \eqref{eq:MSdualextstep2}, we obtain the desired result, upon defining $g,G$ as in \eqref{eq:gGdef}, using \autoref{lemma:gGidentities} and setting, for some arbitrary $\delta \in (0,1)$: $\delta_0 = \delta (1-\kappa^{-1})\frac{|p-1|}{p} \le 1$, $\delta_1 = \frac{\delta}{p|p-1|}$, $\delta_2 = \frac{\delta}{|p-1|}$.
\end{proof}

\begin{lemma}
\label{lemma:MSgendualext2}
Assume that \eqref{K1glob}, \eqref{K2} and \eqref{cutoff} hold true for some $\theta \in (\frac{d}{\alpha},\infty]$. Moreover, assume \eqref{Sob} if $\theta < \infty$. Then for every $\delta > 0$, there is $c> 0$ such that for every $0 < \rho \le r \le 1$ and every nonnegative function $u \in V(B_{r+\rho}|\R^d)$, $\eps > 0$, $d \in L^{\infty}(\R^d)$:
\begin{align*}
-\delta\int_{B_{r+\rho}} \int_{B_{r+\rho}} (\tau^2(x) \wedge \tau^2(y)) &\left(\log\frac{\U(x)}{\tau(x)}-\log\frac{\U(y)}{\tau(y)}\right)^2 K_s(x,y) \d y \d x\\
&\qquad\qquad\qquad \le \widehat{\cE}^{K_a}(d,-\tau^2 \U^{-1}) + c \rho^{-\alpha}\vert B_{r+\rho}\vert,
\end{align*}
where $B_{2r} \subset \Omega$, $\tau = \tau_{r,\frac{\rho}{2}}$, and $\U = u + \eps + r^{\frac{1}{2}\left( \alpha - \frac{d}{\theta} \right)}\Vert d \Vert_{\infty}$.
\end{lemma}

\begin{proof}
The proof follows by combination of the respective proofs of \autoref{lemma:MSgendual2} and \autoref{lemma:MSgendualext}. As before, let us denote $\eta:= \frac{1}{2}\left(\alpha - \frac{d}{\theta}\right)$ and note that by definition, $\Vert d \Vert_{\infty} \le \rho^{-\eta} \U$. Moreover, we define $M$ as in the proof of \autoref{lemma:MSgen}. We  will explain how to prove that for every $\delta > 0$ there exists $c > 0$ such that
\begin{align}
\label{eq:MSgendualext2help1}
\begin{split}
\widehat{\cE}_{B_{r+\rho}}(u,-\tau^2 \U^{-1}) &\ge -\delta \int_{B_{r+\rho}}\int_{B_{r+\rho}} (\log \U(x) - \log \U(y))^2 (\tau^2(x) \wedge \tau^2(y))K_s(x,y) \d y \d x\\
&-c [C + \rho^{-\alpha}]|B_{r+\rho}|,
\end{split}
\end{align}
where $C = C(1) > 0$ is the constant from \eqref{eq:K1consequence}.
By the same arguments as in Step 2 of the proof of \autoref{lemma:MSgendualext}, one can prove that
\begin{align*}
\widehat{\cE}^{K_a}_{(B_{r+\rho} \times B_{r+\rho})^{c}} & (d,-\tau^2 \U^{-1}) = \widehat{\cE}^{K_a}_{(B_{r+\rho} \times B_{r+\rho})^{c}}(d,-\tau^2 g(\U))\\
&\ge -c [C\rho^{-\alpha} + \rho^{-\alpha}]\Vert g(\U)\U\Vert_{L^1(B_{r+\rho})} - c\cE^{K_s}_{B_{r+\rho}}(\tau g(\U)^{1/2}\U^{1/2} , \tau g(\U)^{1/2}\U^{1/2})\\
&= -c \rho^{-\alpha}|B_{r+\rho}|,
\end{align*}
choosing $g,G$ as in \eqref{eq:gGdef} with $p=1$ and applying \autoref{lemma:gGidentities}, where we used that $\rho \le 1$.
Combining these two estimates, the desired result directly follows using \eqref{eq:logaux4} and \eqref{cutoff}, as in the proof of \autoref{lemma:MSgen2}.\\
It remains to prove \eqref{eq:MSgendualext2help1}. As in the proof of \autoref{lemma:MSgendualext}, we decompose
\begin{align*}
\widehat{\cE}^{K_a}_{B_{r+\rho}}(d,-\tau^2 g(\U)) \ge I_1 + I_2.
\end{align*}
To estimate $I_1$, we compute that there exists $c > 0$ such that for every $\delta > 0$:
\begin{align*}
I_1 &\ge -2 \rho^{-\eta} \iint_M \tau^2(x) \U(x)|g(\U(x)) - g(\U(y))| |K_a(x,y)| \d y \d x\\
&\ge - \delta \iint_M (\tau^2(x) \vee \tau^2(y)) \left(\U(x)|g(\U(x)) - g(\U(y))|\right)^2 J(x,y) \d y \d x\\
& \quad - c \delta^{-1} \rho^{-2\eta} \iint_{M} (\tau^2(x) \vee \tau^2(y)) \frac{|K_a(x,y)|^2}{J(x,y)} \d y \d x.
\end{align*}
As in the estimate for $M_a$ in the proof of \autoref{lemma:MSgendual2}, we estimate the first summand as follows:
\begin{align*}
\delta &\iint_M (\tau^2(x) \vee \tau^2(y)) \left(\U(x)|g(\U(x)) - g(\U(y))|\right)^2 J(x,y) \d y \d x \le 2\delta \rho^{-\alpha} \Vert \U^2 G'(\U)^2\Vert_{L^1(B_{r+\rho})} \\
&+ 2\delta \iint_M (G(\U(x)) - G(\U(y)))^2 \left(\U(x)G'(\U(x)) \vee \U(y)G'(\U(y)) \right)(\tau^2(x) \wedge \tau^2(y)) J(x,y) \d y \d x,
\end{align*}
using \eqref{eq:genaux3}, \eqref{cutoff}, \autoref{lemma:logusaviour} and \autoref{lemma:improvedK2help}.
Moreover, by definition of $\eta$, and \eqref{K1}:
\begin{align*}
\delta^{-1} \rho^{-2\eta} \iint_{M} (\tau^2(x) \vee \tau^2(y)) \frac{|K_a(x,y)|^2}{J(x,y)} \d y \d x \le c\delta^{-1} \rho^{\frac{d}{\theta}-\alpha} \Vert \tau^2 \Vert_{L^{\theta'}(B_{r+\rho})} \le c\delta^{-1} \rho^{-\alpha} |B_{r+\rho}|.
\end{align*}
This provides the desired estimate for $I_1$.\\
The estimate of $I_2$ follows along the lines of the corresponding estimate in the proof of \autoref{lemma:MSgendualext} and yields that there exists $c> 0$ such that:
\begin{align*}
I_2 &\ge -c [C\rho^{-\alpha} + \rho^{-\alpha}]\Vert g(\U)\U\Vert_{L^1(B_{r+\rho})} - c \cE^{K_s}_{B_{r+\rho}}(\tau g(\U)^{1/2}\U^{1/2} , \tau g(\U)^{1/2}\U^{1/2}).
\end{align*}

Altogether, we have shown:
\begin{align*}
&\widehat{\cE}^{K_a}_{B_{r+\rho}}(d,-\tau^2 g(\U)) \ge -c [C\rho^{-\alpha} + \rho^{-\alpha}] \Vert g(\U)\U\Vert_{L^1(B_{r+\rho})} - c \cE^{K_s}_{B_{r+\rho}}(\tau g(\U)^{1/2}\U^{1/2} , \tau g(\U)^{1/2}\U^{1/2})\\
& -2\delta \iint_M (G(\U(x)) - G(\U(y)))^2 \left(\U(x)G'(\U(x)) \vee \U(y)G'(\U(y)) \right)(\tau^2(x) \wedge \tau^2(y)) J(x,y) \d y \d x\\
&-2\delta \rho^{-\alpha} \Vert \U^2 G'(\U)^2\Vert_{L^1(B_{r+\rho})} -c\delta^{-1} \rho^{-\alpha} |B_{r+\rho}| 
\end{align*}

Finally, let us choose $g,G$ as in \eqref{eq:gGdef} with $p=1$. This implies \eqref{eq:MSgendualext2help1} upon using \autoref{lemma:gGidentities} and \autoref{lemma:improvedK2help}, as desired.

\end{proof}

\section{Weak parabolic Harnack inequalities}
\label{sec:wpHI}

In this section we establish our main results \autoref{thm:mainthmPDE}(i) and \autoref{thm:mainthmPDEdual}(i). Moreover, we will prove a weak parabolic Harnack inequality for nonnegative weak supersolutions to \eqref{PDEdualext}. \\
The Caccioppoli-type estimates for nonsymmetric forms established in \autoref{sec:Caccioppoli} are tailored in such a way that it is possible to perform the well-known Moser iteration technique for solutions to \eqref{PDE} and \eqref{PDEdual} (respectively \eqref{PDEdualext}) in the same way as for parabolic equations connected to symmetric forms. We follow the the proof in \cite{FeKa13}, which is based on Moser's original arguments, see \cite{Mos64}, \cite{Mos71}.

First, we set up Moser iteration schemes for negative and small positive exponents (see \autoref{thm:MI}, \autoref{thm:MI2}), in order to prove estimates for the infimum and the $L^1$-norm of a nonnegative supersolution. These results require the Caccioppoli-type estimates \autoref{lemma:MSgen}, \autoref{lemma:MSgendual} and \autoref{lemma:MSgendualext}.\\

\begin{theorem}[Moser iteration I: negative exponents]
\label{thm:MI}
Assume \eqref{cutoff} and \eqref{Sob}. 
\vspace{-0.2cm}
\begin{itemize}
\item[(i)] Assume \eqref{K1} holds true for some $\theta \in [\frac{d}{\alpha}, \infty]$. Then, there exist $c > 0$ and $\delta > 0$ such that for every $p^* \in (0,1]$,  for every $0 < R \le 1$, every $\sigma \in [1/2,1)$ and every nonnegative, weak supersolution $u$ to \eqref{PDE} in $I^{\ominus}_R(t_0) \times B_{2R}$ and every $\eps > 0$, it holds:
\begin{align}
\label{eq:MI}
\inf_{I^{\ominus}_{\sigma R}(t_0) \times B_{\sigma R}} \U \ge \left[c (1- \sigma)^{\delta}\right]^{1/p^*} \left(\dashint_{I^{\ominus}_R(t_0)} \dashint_{B_R} \U^{-p^*}(t,x) \d x \d t\right)^{-1/p^*},
\end{align}
where $B_{2R} \subset \Omega$ and $\U = u + \eps + R^{\alpha}\Vert f \Vert_{\infty}$.

\item[(ii)] Assume \eqref{K1glob} holds true for some $\theta \in (\frac{d}{\alpha}, \infty]$. Then, there exist $c > 0$ and $\delta > 0$ such that for every $p^* \in (0,1]$, every $0 < R \le 1$, every $\sigma \in [1/2,1)$ and every  nonnegative, weak supersolution $u$ to \eqref{PDEdualext} in $I^{\ominus}_R(t_0) \times B_{2R}$ and every $\eps > 0$ the estimate \eqref{eq:MI} holds true with $\U = u + \eps + R^{\alpha}\Vert f \Vert_{L^{\infty}} + R^{\frac{1}{2}\left(\alpha - \frac{d}{\theta} \right)}\Vert d \Vert_{L^{\infty}}$.
\end{itemize}
\end{theorem}

\begin{proof}
We first explain how to prove (i). The proof works as in \cite{FeKa13}. First, let $0 < \rho \le r \le r + \rho < R$, and $q \ge p^{\ast}$ for some $p^{\ast} \in (0,1]$.
Writing $q = p-1$, we observe that \autoref{lemma:MSgen} implies that for any nonnegative $u \in V(B_{r+\rho}|\R^d)$:
\begin{align}
\label{eq:negexpCacc}
\cE^{K_s}_{B_{r+\rho}}(\tau \U^{-\frac{q}{2}},\tau \U^{-\frac{q}{2}}) \le c_1 q\cE(u,-\tau^{2}\U^{-q-1}) + c_2 (1 \vee q) \rho^{-\alpha} \Vert \U^{-q}\Vert_{L^{1}(B_{r+\rho})},
\end{align}
where $c_1, c_2 > 0$ are some constants.
We observe that $\partial_t(\U^{-q}) = -q\U^{-q-1} \partial_t u$. For any nonnegative weak supersolution $u$ to \eqref{PDE}, we obtain using \eqref{eq:negexpCacc} with $q = p-1$:
\begin{align}
\label{eq:MIhelp1}
\begin{split}
c_1&\int_{B_{r+\rho}}\tau^{2}(x)\partial_t(\U^{-q})(t,x) \d x + \cEs_{B_{r+\rho}}(\tau \U^{-\frac{q}{2}}(t),\tau \U^{-\frac{q}{2}}(t))\\
&\le c_1 q \left[(\partial_t u(t) , -\tau^{2}\U^{-q-1}(t)) +  \cE(u(t),-\tau^{2}\U^{-q-1}(t)) \right] + c_2 (1 \vee q) \rho^{-\alpha} \Vert \U^{-q}(t)\Vert_{L^{1}(B_{r+\rho})}\\
&\le c_3 q\left[ (f(t),-\tau^{2}\U^{-q-1})\right] + c_2 (1 \vee q)\rho^{-\alpha} \Vert \U^{-q}(t)\Vert_{L^{1}(B_{r+\rho})}\\
&\le c_4 (1 \vee q) \rho^{-\alpha} \Vert \U^{-q}(t)\Vert_{L^{1}(B_{r+\rho})},
\end{split}
\end{align}
where $c_3,c_4 > 0$ are constants and we tested the equation with $-\tau^{2}\U^{-q-1}$, where $\tau = \tau_{r,\frac{\rho}{2}}$. Moreover, we used that by definition of $\U$:
\begin{align*}
(f(t),-\tau^{2}\U^{-q-1}) \le c_5 \rho^{-\alpha} \Vert \U^{-q}(t)\Vert_{L^{1}(B_{r+\rho})}
\end{align*}
for some $c_5 > 0$.
Let now $\chi \in C^1(\R)$ be a nonnegative function such that $\chi \equiv 1$ in $I^{\ominus}_r$, $\Vert \chi \Vert_{\infty} \le 1$, $\Vert \chi'\Vert_{\infty} \le 2((r+\rho)^{\alpha}-r^{\alpha})^{-1}$ and $\chi(t_0 - (r+\rho)^{\alpha}) = 0$. We multiply \eqref{eq:MIhelp1} with $\chi^2$ and integrate from $t_0 - (r+\rho)^{\alpha}$ to $t$ for some arbitrary $t \in I^{\ominus}_{r}(t_0)$. This yields
\begin{align*}
\int_{B_{r+\rho}}& \chi^2(t)\tau^{2}(x) \U^{-q}(t,x) \d x + \int^t_{t_0 - (r+\rho)^{\alpha}} \chi^2(s) \cEs_{B_{r+\rho}}(\tau \U^{-\frac{q}{2}}(s),\tau \U^{-\frac{q}{2}}(s)) \d s\\
&\le c_6 (1 \vee q) \rho^{-\alpha} \int_{t_0 - (r+\rho)^{\alpha}}^t \chi^2(s) \Vert \U^{-q}(s)\Vert_{L^{1}(B_{r+\rho})} \d s\\
& \quad + c_6\int_{t_0 - (r+\rho)^{\alpha}}^t \vert\chi'(s)\vert\chi(s) \int_{B_{r+\rho}} \tau^{2}(x) \U^{-q}(s,x) \d x \d s,
\end{align*}
where $c_6 > 0$.
Using the properties of $\chi$ we obtain
\begin{align*}
\sup_{t \in I^{\ominus}_r} &\int_{B_{r}} \U^{-q}(t,x) \d x + \int_{I^{\ominus}_r} \cEs_{B_{r+\rho}}(\tau \U^{-\frac{q}{2}}(s),\tau \U^{-\frac{q}{2}}(s)) \d s\\
&\le c_7 (1 \vee q) \left(\rho^{-\alpha} \vee ((r+\rho)^{\alpha}-r^{\alpha})^{-1}\right) \int_{I^{\ominus}_{r+\rho}} \Vert \U^{-q}(s)\Vert_{L^{1}(B_{r+\rho})} \d s
\end{align*}
for some $c_7 > 0$.
Let us define $\kappa = \frac{\alpha}{d} = 1 + \frac{\alpha}{d} > 1$. By H\"older interpolation with $\frac{1}{\kappa} = \frac{(\kappa-1)/\kappa}{1} + \frac{1/\kappa}{d/(d-\alpha)}$ and Sobolev inequality \eqref{Sob} we obtain
\begin{align}
\label{eq:MIhelp2}
\begin{split}
\Vert \U^{-q} \Vert_{L^{\kappa}(I_r^{\ominus} \times B_r)} &\le \left(\sup_{t \in I_r^{\ominus}}\Vert \U^{-q}(t) \Vert_{L^{1}(B_r)}^{\kappa-1} \int_{I_r^{\ominus}}\Vert \U^{-q}(s) \Vert_{L^{\frac{d}{d-\alpha}}(B_r)} \d s\right)^{1/\kappa}\\
&\qquad\le c_8 (1 \vee q) \left(\rho^{-\alpha} \vee ((r+\rho)^{\alpha}-r^{\alpha})^{-1}\right) \Vert \U^{-q}\Vert_{L^{1}(I^{\ominus}_{r+\rho} \times B_{r+\rho})}
\end{split}
\end{align}
for some $c_8 > 0$.
Next, we define sequences $(q_i)_i, (\rho_i)_i$ with $\rho_i,q_i > 0$, $q_i \nearrow \infty$ and $\rho_i \searrow 0$ such that $r_i  := r_{i-1}-\rho_i \searrow \sigma R$ and $r_0 = R$. Namely, we set $\rho_i = 2^{-i}(1-\sigma)R$, $q_i = \kappa q_{i-1} = \kappa^i q_0$ and $q_0 = p^*$. Then we take \eqref{eq:MIhelp2} to the power $1/q_0 = 1/p^*$, iterate, and use that as $i \to \infty$:
\begin{align*}
\left( \dashint_{I^{\ominus}_{r_i}} \dashint_{B_{r_i}} \U^{-q_i}(t,x) \d x \d t\right)^{-1/q_i} \to \left(\inf_{I_{\sigma R}^{\ominus} \times B_{\sigma R}} \U\right)^{-1}.
\end{align*}
Further details can be found in \cite{FeKa13}.\\
The proof of (ii) works in a similar way. In fact, any nonnegative weak supersolution to \eqref{PDEdualext} satisfies 
\begin{align}
\label{eq:MIhelp3}
c_9&\int_{B_{r+\rho}}\tau^{2}(x)\partial_t(\U^{-q})(t,x) \d x + \cEs_{B_{r+\rho}}(\tau \U^{-\frac{q}{2}}(t),\tau \U^{-\frac{q}{2}}(t)) \le c_{10} (1 \vee q^{\gamma}) \rho^{-\alpha} \Vert \U^{-q}(t)\Vert_{L^{1}(B_{r+\rho})}
\end{align}
for some $c_9,c_{10} > 0$ and $\gamma \ge 1$. This is due to the fact that \autoref{lemma:MSgendual} and \autoref{lemma:MSgendualext} imply:
\begin{align*}
\cE^{K_s}_{B_{r+\rho}}(\tau \U^{-\frac{q}{2}},\tau \U^{-\frac{q}{2}}) \le c_{11} q\left[\widehat{\cE}(u,-\tau^{2}\U^{-q-1}) + \widehat{\cE}^{K_a}(d,-\tau^2 \U^{-q-1}) \right] + c_{12} (1 \vee q^{\gamma})\rho^{-\alpha} \Vert \U^{-q}\Vert_{L^{1}(B_{r+\rho})}
\end{align*}
for some $c_{11},c_{12} > 0$, $\gamma \ge 1$. With the help of \eqref{eq:MIhelp3}, we can establish (ii) via the same arguments as in the proof of (i).
\end{proof}

\begin{theorem}[Moser iteration II: small positive exponents]
\label{thm:MI2}
Assume \eqref{cutoff} and \eqref{Sob}. 
\vspace{-0.2cm}
\begin{itemize}
\item[(i)] Assume \eqref{K1} holds true for some $\theta \in [\frac{d}{\alpha} , \infty]$. Then, there exist $c > 0$ and $\delta > 0$ such that for every $p^* \in (0,\kappa^{-1})$, every $0 < R \le 1$, every $\sigma \in [1/2,1)$ and every nonnegative, weak supersolution $u$ to \eqref{PDE} in $I^{\oplus}_R(t_0) \times B_{2R}$ and every $\eps > 0$, it holds:
\begin{align}
\label{eq:MI2}
\dashint_{I^{\oplus}_{\sigma R}(t_0) \times B_{\sigma R}} \U(t,x) \d x \d t \le \left[c (1- \sigma)^{-\delta}\right]^{\left(\frac{1}{p^*}-1\right)} \left( \dashint_{I^{\oplus}_R(t_0)} \dashint_{B_R} \U^{p^*}(t,x) \d x \d t \right)^{1/p^*},
\end{align}
where $B_{2R} \subset \Omega$, $\kappa = 1 + \frac{\alpha}{d}$ and $\U = u + \eps + R^{\alpha}\Vert f \Vert_{\infty}$.

\item[(ii)] Assume \eqref{K1glob} holds true for some $\theta \in (\frac{d}{\alpha} , \infty]$. Then, there exist $c > 0$ and $\delta > 0$ such that for every $p^* \in (0,\kappa^{-1})$, every $0 < R \le 1$, every $\sigma \in [1/2,1)$ and every nonnegative, weak supersolution $u$ to \eqref{PDEdualext} in $I^{\oplus}_R(t_0) \times B_{2R}$ and every $\eps > 0$, the estimate \eqref{eq:MI2} holds true with $\U = u + \eps + R^{\alpha}\Vert f \Vert_{L^{\infty}} + R^{\frac{1}{2}\left(\alpha - \frac{d}{\theta} \right)}\Vert d \Vert_{L^{\infty}}$.
\end{itemize}
\end{theorem}

\begin{proof}
We first explain how to prove (i). The proof works as in \cite{FeKa13}. Let $0 < \rho \le r \le r + \rho < R$, and $q \in [p^{\ast},\kappa^{-1})$ for some $p^{\ast} \in (0,\kappa^{-1})$. Writing $q = 1-p$, we observe that \autoref{lemma:MSgen} implies that for any nonnegative $u \in V(B_{r+\rho}|\R^d)$:
\begin{align}
\label{eq:posexpCacc}
\cE^{K_s}_{B_{r+\rho}}(\tau \U^{\frac{q}{2}},\tau \U^{\frac{q}{2}}) \le c_1 q\cE(u,-\tau^{2}\U^{q-1}) + c_2 (1 \vee q) \rho^{-\alpha} \Vert \U^{q}\Vert_{L^{1}(B_{r+\rho})},
\end{align}
where $c_1, c_2 > 0$ are some constants. We observe that $-\partial_t (\U^q) = -q \U^{q-1} \partial_t u$. For any nonnegative weak supersolution $u$ to \eqref{PDE}, we obtain using \eqref{eq:posexpCacc}:
\begin{align}
\label{eq:MI2help1}
\begin{split}
-c_1 &\int_{B_{r+\rho}} \tau^2(x)\partial_t(\U^q)(t,x)\d x + \cEs_{B_{r+\rho}}(\tau \U^{\frac{q}{2}}(t),\tau \U^{\frac{q}{2}}(t))\\
&\le c_1 q \left[ (\partial_t u(t) , -\tau^2 \U^{q-1}(t)) + \cE(u(t),-\tau^2 \U^{q-1}(t))\right] + c_2(1 \vee q) \rho^{-\alpha} \Vert \U^q(t) \Vert_{L^{1}(B_{r+\rho})}\\
&\le c_2 q \left[(f(t) , -\tau^2 \U^{q-1}(t))\right] + c_3 q \rho^{-\alpha}\Vert \U^q(t) \Vert_{L^{1}(B_{r+\rho})}\\
&\le c_4 \rho^{-\alpha}\Vert \U^q(t) \Vert_{L^{1}(B_{r+\rho})},
\end{split}
\end{align}
where $c_3,c_4 > 0$ are constants, we used that $q \le 1$, and we tested the equation with $-\tau^2 \U^{q-1}$ with $\tau = \tau_{r,\frac{\rho}{2}}$. Moreover, we used that by definition of $\U$:
\begin{align*}
(f(t) , -\tau^2 \U^{q-1}(t)) \le c_5 \rho^{-\alpha} \Vert \U^q(t) \Vert_{L^{1}(B_{r+\rho})}
\end{align*}
for some $c_5 > 0$.
The difference between \eqref{eq:MI2help1} and the estimate \eqref{eq:MIhelp1} in the proof of \autoref{thm:MI} is the negative sign in front of the term involving the time derivative.
Therefore, we take a nonnegative function $\chi \in C^1(\R)$ such that $\chi \equiv 1$ in $I^{\oplus}_r(t_0)$, $\Vert \chi \Vert_{\infty} \le 1$, $\Vert \chi'\Vert_{\infty} \le 2((r+\rho)^{\alpha} - r^{\alpha})^{-1}$, $\chi(t_0 + (r+\rho)^{\alpha}) = 0$  and multiply \eqref{eq:MI2help1} with $\chi^2$ and integrate from $t$ to $t_0 + (r+\rho)^{\alpha}$ for an arbitrary $t \in I^{\oplus}_r(t_0)$. This yields
\begin{align*}
\int_{B_{r+\rho}} &\chi^2(t) \tau^2(x) \U^q(t,x) \d x + \int_{t}^{t_0 + (r+\rho)^{\alpha}} \chi^2(s) \cEs_{B_{r+\rho}}(\tau \U^{\frac{q}{2}}(s),\tau \U^{\frac{q}{2}}(s)) \d s\\
&\le c_6 \rho^{-\alpha} \int_{t}^{t_0 + (r+\rho)^{\alpha}} \Vert \U^q(s) \Vert_{L^{1}(B_{r+\rho})} \d s\\
& \quad + c_6 \int_{t}^{t_0 + (r+\rho)^{\alpha}} \vert \chi'(s)\vert \chi(s) \int_{B_{r+\rho}} \tau^2(x) \U^q(s,x) \d x \d s
\end{align*}
for some $c_6 > 0$. Consequently:
\begin{align*}
\sup_{t \in I^{\oplus}_r} \int_{B_{r}} \U^q(t,x) \d x &+ \int_{I^{\oplus}_r} \cEs_{B_{r+\rho}}(\tau \U^{\frac{q}{2}}(s),\tau \U^{\frac{p}{2}}(s)) \d s\\
&\le c_7 \left(\rho^{-\alpha} \vee ((r+\rho)^{\alpha}-r^{\alpha})^{-1}\right)  \int_{I^{\oplus}_{r+\rho}} \Vert \U^q(s) \Vert_{L^{1}(B_{r+\rho})} \d s
\end{align*}
for some $c_7 > 0$. Recall that $\kappa = 1 + \frac{\alpha}{d} > 1$ and observe that using H\"older interpolation and Sobolev inequality \eqref{Sob} we can derive:
\begin{align}
\label{eq:MI2help2}
\begin{split}
\Vert \U^{q} \Vert_{L^{\kappa}(I_r^{\oplus} \times B_r)} &\le \left(\sup_{t \in I_r^{\oplus}}\Vert \U^{q}(t) \Vert_{L^{1}(B_r)}^{\kappa-1} \int_{I_r^{\oplus}}\Vert \U^{q}(s) \Vert_{L^{\frac{d}{d-\alpha}}(B_r)} \d s\right)^{1/\kappa}\\
&\le c_8 \left(\rho^{-\alpha} \vee ((r+\rho)^{\alpha}-r^{\alpha})^{-1}\right)\Vert \U^{q}\Vert_{L^{1}(I^{\oplus}_{r+\rho} \times B_{r+\rho})}
\end{split}
\end{align}
for some $c_8 > 0$. Next, again an iteration argument yields the desired result (see \cite{FeKa13}).\\
The proof of (ii) works in the same way. In fact, one can prove, as in the proof of \autoref{thm:MI}, that any nonnegative weak supersolution to \eqref{PDEdualext} satisfies 
\begin{align*}
-c_{10}&\int_{B_{r+\rho}}\tau^{2}(x)\partial_t(\U^{q})(t,x) \d x + \cEs_{B_{r+\rho}}(\tau \U^{\frac{q}{2}}(t),\tau \U^{\frac{q}{2}}(t)) \le c_{11} \rho^{-\alpha} \Vert \U^{q}(t)\Vert_{L^{1}(B_{r+\rho})}
\end{align*}
for some $c_{10}, c_{11} > 0$, using \autoref{lemma:MSgendual} and \autoref{lemma:MSgendualext}. From this line, (ii) follows by the same arguments as before.
\end{proof}

Next, we establish weak $L^1$-estimates for the logarithm of a nonnegative supersolution to \eqref{PDE}, respectively \eqref{PDEdualext}, using the Caccioppoli-type estimates derived in  \autoref{lemma:MSgen2}, \autoref{lemma:MSgendual2}, and \autoref{lemma:MSgendualext2}.

\begin{theorem}[weak $L^1$-estimates for $\log \U$]
\label{thm:BMO}
Assume \eqref{K2}, \eqref{cutoff} and \eqref{Poinc}. 
\vspace{-0.2cm}
\begin{itemize}
\item[(i)] Assume \eqref{K1} holds true for some $\theta \in [\frac{d}{\alpha}, \infty]$. Moreover, assume \eqref{Sob} if $\theta < \infty$. Then there exists $c > 0$ such that for every $0 < R \le 1$ and every nonnegative, weak supersolution $u$ to \eqref{PDE} in $I_R(t_0) \times B_{2R}$ and every $\eps > 0$, it holds:
\begin{align}
\label{eq:BMO1}
& \left\vert I^{\oplus}_R(t_0) \times B_R \cap \{ \log \U < -s-a\}\right\vert \le c\frac{\left\vert I^{\oplus}_R(t_0) \times B_R \right\vert}{s}, ~~s > 0,\\
\label{eq:BMO2}
& \left\vert I^{\ominus}_R(t_0) \times B_R \cap \{ \log \U > s-a\}\right\vert \le c\frac{\left\vert I^{\ominus}_R(t_0) \times B_R \right\vert}{s}, ~~s > 0,
\end{align}
where $B_{2R} \subset \Omega$, $a = a(\U) \in \R$ is a constant, $\U = u + \eps + R^{\alpha}\Vert f \Vert_{L^{\infty}}$.

\item[(ii)] Assume \eqref{K1glob} holds true for some $\theta \in (\frac{d}{\alpha}, \infty]$. Moreover, assume \eqref{Sob} if $\theta < \infty$. Then, there exists $c > 0$ such that for every $0 < R \le 1$ and every any nonnegative, weak supersolution $u$ to \eqref{PDEdualext} in $I_R(t_0) \times B_{2R}$ and every $\eps > 0$ the estimates \eqref{eq:BMO1}, \eqref{eq:BMO2} hold true with $\U = u + \eps + R^{\alpha}\Vert f \Vert_{L^{\infty}} + R^{\frac{1}{2}\left(\alpha - \frac{d}{\theta} \right)}\Vert d \Vert_{L^{\infty}}$.
\end{itemize}
\end{theorem}

\begin{proof}
We first explain how to prove (i). The proof works as in \cite{FeKa13}. We apply \autoref{lemma:MSgen2} with $r = R, \rho = R/2$ and define $v(t,x) = -\log\frac{\U(t,x)}{\tau(x)}$, where $\tau = \tau_{r,\frac{\rho}{2}}$. Observe that $\partial_t v = - \U^{-1}\partial_t u$. For any nonnegative weak supersolution $u$ to \eqref{PDE} and every $t \in I_R$, we obtain using \autoref{lemma:MSgen2}:
\begin{align*}
\int_{B_{3R/2}} \tau^2(x) \partial_t v(t,x) \d x &+ c_1\int_{B_{3R/2}} \int_{B_{3R/2}} (\tau^2(x) \wedge \tau^2(y)) \left(v(t,x)-v(t,y)\right)^2 K_s(x,y) \d y \d x\\
&\le \left[(\partial_t u(t),-\tau^2 \U^{-1}(t)) + \cE(u(t),-\tau^2 \U^{-1}(t))\right] + c_2 R^{-\alpha}\vert B_{R}\vert\\
&\le (f , -\tau^{2}\U^{-1}(t)) + c_2 R^{-\alpha}\vert B_{R}\vert \le c_3 R^{-\alpha}\vert B_{R}\vert,
\end{align*}
for some constants $c_1,c_2,c_3 > 0$, where we tested the equation with $-\tau^2 \U^{-1}$ and used that
\begin{align*}
(f(t) , -\tau^{2}\U^{-1}(t)) \le c R^{-\alpha} \vert B_{R}\vert
\end{align*}
by the definition of $\U$. 
We define $V(t) = \frac{\int_{B_{3R/2}} v(t,x)\tau^2(x) \d x}{\int_{B_{3R/2}}\tau^2(x) \d x}$ and apply a weighted Poincar\'e inequality as it is derived from \eqref{Poinc} in \cite{DyKa13} (see Proposition 4 in \cite{DyKa13}). This yields
\begin{align}
\label{eq:BMOhelp1}
\int_{B_{3R/2}} \tau^2(x) \partial_t v(t,x) \d x + c_4 R^{-\alpha} \int_{B_{2R/3}} (v(t,x) - V(t))^2 \tau^2(x) \d x \le c_2 R^{-\alpha} \vert B_R \vert,
\end{align}
where $c_4 > 0$.
Next, we integrate \eqref{eq:BMOhelp1} in time over $[t_1,t_2] \subset I_R$ divide by $\int_{B_{3R/2}} \tau^2(x) \d x$ and restrict the domain of integration for the second integral to $B_R$:
\begin{align*}
V(t_2) - V(t_1) + c_5 R^{-d-\alpha} \int_{t_1}^{t_2}  \int_{B_{R}} (v(t,x) - V(t))^2 \d x \le c_6 (t_2-t_1) R^{-\alpha},
\end{align*}
where $c_5, c_6 > 0$.
From here, one proceeds as in \cite{CKW19} or \cite{FeKa13}.\\
The proof of (ii) follows in the same way. Using \autoref{lemma:MSgendual2} and \autoref{lemma:MSgendualext2}, it becomes apparent that any nonnegative weak supersolution to \eqref{PDEdualext} satisfies
\begin{align*}
\int_{B_{3R/2}} \tau^2(x) &\partial_t v(t,x) \d x + c_1\int_{B_{3R/2}} \int_{B_{3R/2}} (\tau^2(x) \wedge \tau^2(y)) \left(v(t,x)-v(t,y)\right)^2 K_s(x,y) \d y \d x\\
&\le \left[(\partial_t u(t),-\tau^2 \U^{-1}(t)) + \widehat{\cE}(u(t),-\tau^2 \U^{-1}(t)) +  \widehat{\cE}^{K_a}(d,-\tau^2 \U^{-1}(t))\right] + c_7 R^{-\alpha}\vert B_{R}\vert\\
&\le c_8 R^{-\alpha}\vert B_{R}\vert
\end{align*}
for some $c_7,c_8 > 0$. With the help of this estimate, the remaining proof of (ii) goes via the same arguments as the proof of (i).
\end{proof}

Finally, we are in the position to deduce the desired weak parabolic Harnack inequality for weak supersolutions to \eqref{PDE}, respectively \eqref{PDEdual}, using a well-known lemma by Bombieri and Giusti.

\begin{proof} (of \autoref{thm:mainthmPDE} (i), \autoref{thm:mainthmPDEdual} (i))
The proof works as in \cite{FeKa13} or \cite{CKW19}. The idea is to make use of a lemma by Bombieri and Giusti (Lemma 2.2.6 in \cite{Sal02}).\\
For the application of this lemma, we need \autoref{thm:BMO}, \autoref{thm:MI} and \autoref{thm:MI2} applied with $\U = u + \eps + R^{\alpha}\Vert f \Vert_{L^{\infty}}$, where $u$ is a weak supersolution to \eqref{PDE}, respectively \eqref{PDEdual} and $\eps > 0$ is arbitrary. This yields:
\begin{align*}
\inf_{(t_0 + R^{\alpha} - (R/2)^{\alpha}, t_0 + R^{\alpha}) \times B_{R/2}} \U \ge c \left( \dashint_{(t_0 - R^{\alpha}, t_0 - R^{\alpha} + (R/2)^{\alpha}) \times B_{R/2}} \U(t,x) \d x \d t - R^{\alpha}\Vert f \Vert_{L^{\infty}}\right).
\end{align*}
The desired result follows upon taking the limit $\eps \searrow 0$.
\end{proof}

\begin{remark*}
Note that the proof of \autoref{thm:mainthmPDEdual}(i) relies on \autoref{thm:BMO}(ii), \autoref{thm:MI}(ii) and \autoref{thm:MI2}(ii) applied with $d \equiv 0$. Therefore, the results in \autoref{lemma:MSgendualext} and \autoref{lemma:MSgendualext2} are not required for the weak parabolic Harnack inequality for \eqref{PDEdual}.
\end{remark*}

Moreover, thanks to \autoref{lemma:MSgendualext} and \autoref{lemma:MSgendualext2}, we are able to prove a weak parabolic Harnack inequality for nonnegative supersolutions to \eqref{PDEdualext}:

\begin{theorem}
\label{thm:wHIdualext}
Assume that \eqref{K1glob}, \eqref{K2}, \eqref{cutoff}, \eqref{Poinc}, \eqref{Sob}, and \eqref{cutoff2dual} hold true for some $\alpha \in (0,2)$, $\theta \in (\frac{d}{\alpha},\infty]$.
Then there is $c > 0$ such that for every $0 < R \le 1$, every $d \in L^{\infty}(\R^d)$ and every nonnegative, weak supersolution $u$ to \eqref{PDEdualext} in $I_R(t_0) \times B_{2R}$:
\begin{align}
\label{eq:genwPHI}
\inf_{\left(t_0 + R^{\alpha} - (\frac{R}{2})^{\alpha}, t_0 + R^{\alpha}\right)\times B_{\frac{R}{2}}} u \ge c \left( \dashint_{\left(t_0 - R^{\alpha}, t_0 - R^{\alpha} + (\frac{R}{2})^{\alpha}\right) \times B_{\frac{R}{2}}} u(t,x) \d x \d t - R^{\alpha}\Vert f \Vert_{L^{\infty}} - R^{\eta} \Vert d \Vert_{L^{\infty}}\right).
\end{align}
\end{theorem}

\begin{proof}
By the same argument as in the proof of \autoref{thm:mainthmPDEdual} (i), we can use the lemma by Bombieri and Giusti after applying \autoref{thm:BMO} (ii), \autoref{thm:MI} (ii) and \autoref{thm:MI2} (ii) with $\U = u + \eps + R^{\alpha}\Vert f \Vert_{L^{\infty}} + R^{\eta} \Vert d \Vert_{L^{\infty}(\R^d)}$, setting $\eta = \frac{1}{2}\left(\alpha - \frac{d}{\theta} \right) > 0$, and $\eps > 0$ is arbitrary.
\end{proof}

\section{Pointwise regularity estimates}
\label{sec:proofs}

In this section we explain how to establish H\"older estimates for weak solutions to \eqref{PDE}, respectively \eqref{PDEdual}, proving \autoref{thm:mainthmPDE}(ii) and \autoref{thm:mainthmPDEdual}(ii). The main idea is to use the weak parabolic Harnack inequalities established in the previous section. The underlying argument for weak solutions to \eqref{PDE} works similar to the symmetric case which is by now standard. Therefore, we will restrict ourselves to pointing out the most important steps.\\
However, establishing H\"older estimates for solutions to the dual equation is not straightforward and requires a careful adaptation of the existing arguments. This is due to the fact that the solution property is not invariant under addition of constants, i.e. the family of solutions to the dual equation does not satisfy the properties of the sets $\mathcal{S}_{x,r}$ introduced in \cite{DyKa20}.

We set up the following notation:
\begin{align*}
D(t_0,R) &= (t_0 - 2R^{\alpha} , t_0) \times B_{2R},\\
\widehat{D}(t_0,R) &= (t_0 - 2 R^{\alpha} , t_0 ) \times B_{3R},\\
D_{\ominus}(t_0,R) &= (t_0 - 2 R^{\alpha} , t_0 - 2R^{\alpha} + (R/2)^{\alpha}) \times B_{R/2},\\
D_{\oplus}(t_0,R) &= (t_0 - (R/2)^{\alpha} , t_0) \times B_{R/2}.
\end{align*}

The first step is to deduce an increase of infimum-estimate for weak supersolutions to \eqref{PDE}.

\begin{lemma}
\label{lemma:ioinf}
Assume that \eqref{K1}, \eqref{K2}, \eqref{cutoff}, \eqref{Poinc}, \eqref{Sob}, and \eqref{cutoff2} hold true with $\theta \in [\frac{d}{\alpha},\infty]$.
Then there exist $\beta \in (0,1)$, $\delta \in (0,1)$ such that for every $0 < R \le 1$ and every function $u$ with
\vspace{-0.2cm}
\begin{itemize}
\item $u\ge 0$ in $(t_0 - 2R^{\alpha} , t_0) \times B_{3 R}$,
\item $u$ is a weak supersolution to \eqref{PDE} with $f \equiv 0$ in $D(t_0,R)$,
\item $|D_{\ominus}(t_0,R) \cap \{ u \ge 1/2 \}| \ge \frac{1}{2} |D_{\ominus}(t_0,R)|$,
\item $u \ge 1 - 3^{j\beta}$ in $(t_0 - 2R^{\alpha} , t_0) \times B_{3^j R}$, for every $j \ge 1$,
\end{itemize}
\vspace{-0.2cm}
it holds that $u \ge \delta$ in $D_{\oplus}(t_0,R)$, where $B_{2R} \subset \Omega$.
\end{lemma}

\begin{proof}
First, observe that $u_+$ is a weak supersolution to \eqref{PDE} with $f(t,x) = \int_{\R^d} u_-(t,y) K(x,y) \d y$. This is due to the fact that for every $\phi \in C_c(B_{2R})$
\begin{align*}
\cE(u_- , \phi) = \int_{\R^d} \int_{\R^d} (u_-(x)-u_-(y))\phi(x) K(x,y) \d y \d x = \int_{B_{2R}} \phi(x) \left(\int_{B_{3R}^{c}} u_-(y) K(x,y)\d y \right) \d x.
\end{align*}
Note that by decomposing $B_{3r}^{c} = \bigcup_{j \in \N} B_{3^{j+1} r}\setminus B_{3^{j}r}$ and using the observation that $B_{3^jR}^{c} \subset B_{3^{j-1}R}(x)^{c}$ for every $x \in B_{2R}$, and \eqref{cutoff}, \eqref{cutoff2}
\begin{align*}
\Vert f \Vert_{L^{\infty}(B_{2R})} &\le \sum_{j \in \N} (3^{j\beta}-1) \left\Vert \int_{B_{3^{j+1}R} \setminus B_{3^{j}R}}  K(\cdot,y)\d y \right\Vert_{L^{\infty}(B_{2R})}\\
&\le \sum_{j \in \N}(3^{j\beta}-1) \left\Vert \int_{B_{3^{j-1}R}(\cdot)^{c}}  K(\cdot,y)\d y \right\Vert_{L^{\infty}(B_{2R})} \le cR^{-\alpha}\sum_{j \in \N}(3^{j\beta}-1)3^{-j(\sigma \wedge \alpha)},
\end{align*}
where $\sigma > 0$ is the constant from \eqref{cutoff2}. Since this expression tends to zero, as $\beta \searrow 0$, we can deduce the desired result from the weak Harnack inequality \autoref{thm:mainthmPDE} (i) via standard arguments. For a more detailed proof we refer to \cite{FeKa13}. 
\end{proof}

\begin{proof}[Proof of \autoref{thm:mainthmPDE} (ii)]
The proof of the H\"older estimate for solutions to \eqref{PDE} is well-known. It uses \autoref{lemma:ioinf} and follows along by the same arguments as the respective proofs in \cite{FeKa13}, \cite{CKW19}.
\end{proof}

Now, we prove H\"older estimates for weak solutions to \eqref{PDEdual}. This goes via an increase of infimum-estimate \autoref{lemma:ioinfdual} which is used to establish oscillation decay for weak solutions, like in the classical case. However, since the family of solutions to the dual equation \eqref{PDEdual} is not invariant under addition of constants, we need to work with a larger class of functions, namely the family of solutions to \eqref{PDEdualext}, where $d \in L^{\infty}(\R^d)$. The main ingredient for \autoref{lemma:ioinfdual} is the weak Harnack inequality for weak supersolutions to \eqref{PDEdualext}, see \autoref{thm:wHIdualext}. 

\begin{lemma}
\label{lemma:ioinfdual}
Assume that \eqref{K1glob}, \eqref{K2}, \eqref{cutoff}, \eqref{Poinc}, \eqref{Sob}, and \eqref{cutoff2dual} hold true with $\theta \in (\frac{d}{\alpha},\infty]$.
Then there exist $\beta \in (0,1)$, $\delta \in (0,1)$ and $\nu > 1$ such that for every $0 < R \le 1$ and every function $u$ and every $d \in \R$ with
\vspace{-0.2cm}
\begin{itemize}
\item $u\ge 0$ in $(t_0 - 2R^{\alpha} , t_0) \times B_{3\nu R}$,
\item $u$ is a weak supersolution to \eqref{PDEdualext} with $f \equiv 0$ in $D(t_0,R)$,
\item $|D_{\ominus}(t_0,R) \cap \{ u \ge 1/2 \}| \ge \frac{1}{2} |D_{\ominus}(t_0,R)|$,
\item $u \ge - 3\nu^{2j\beta}$ in $(t_0 - 2R^{\alpha} , t_0) \times B_{3\nu^{2j}R}$, for every $j \ge 1$,
\end{itemize}
\vspace{-0.2cm}
it holds that $u \ge \delta - |d| R^{\frac{1}{2}\left(\alpha - \frac{d}{\theta} \right)}$ in $D_{\oplus}(t_0,R)$, where $B_{2R} \subset \Omega$.
\end{lemma}

\begin{proof}
Let $v = u_+$ and observe that $v$ is a supersolution to \eqref{PDEdualext} with right-hand side $f$ given by $f(t,x) =  \int_{\R^d} u_-(t,y) K(y,x) \d y$, see the proof of \autoref{lemma:ioinf}. Using \eqref{cutoff}, \eqref{cutoff2dual} and the assumptions on $u$, we estimate:
\begin{align*}
\Vert f \Vert_{L^{\infty}(B_{2R})} &\le \left\Vert \int_{B_{3\nu R}^{c}} u_-(y) K(y,\cdot)\d y \right\Vert_{L^{\infty}(B_{2R})}\\
&\le 3 \nu^{2\beta} \left\Vert \int_{B_{3 \nu^{2}R} \setminus B_{3 \nu R}}  \hspace*{-3ex} K(y,\cdot)\d y \right\Vert_{L^{\infty}(B_{2R})} + 3\sum_{j = 2}^{\infty} \nu^{2j\beta} \left\Vert \int_{B_{3 \nu^{2j}R} \setminus B_{3 \nu^{2(j-1)}R}}  \hspace*{-5ex} K(y,\cdot)\d y \right\Vert_{L^{\infty}(B_{2R})}\\
&\le 3 \nu^{2\beta} \left\Vert \int_{B_{\nu R}(\cdot)^c}  \hspace*{-3ex} K(y,\cdot)\d y \right\Vert_{L^{\infty}(B_{2R})} + 3\sum_{j = 2}^{\infty}\nu^{2j\beta} \left\Vert \int_{B_{\nu^{2(j-1)}R}(\cdot)^c}  \hspace*{-3ex} K(y,\cdot)\d y \right\Vert_{L^{\infty}(B_{2R})}\\
&\le c R^{-\alpha} \left( \nu^{2 \beta - (\alpha \wedge \sigma)}  + \sum_{j = 2}^{\infty} \nu^{[2\beta - 2(\alpha \wedge \sigma)] j + 2(\alpha \wedge \sigma)}\right).
\end{align*}
Here, $\sigma > 0$ comes from \eqref{cutoff2dual}. The above quantity converges to zero if $\beta < \frac{\alpha \wedge \sigma}{4}$ is chosen small enough upon sending $\nu \nearrow \infty$. We deduce from \eqref{eq:genwPHI} applied with $v$ that for some $c_1 > 0$:
\begin{align*}
\inf_{D_{\oplus}(R)} u &\ge c \left(\dashint_{D_{\ominus}(R)} u(t,x) \d x \d t - R^{\alpha}\Vert f \Vert_{L^{\infty}} - R^{\eta} |d| \right)\\
&\ge \frac{c}{4} - c_2 \left( \nu^{2 \beta - (\alpha \wedge \sigma)}  + \sum_{j = 2}^{\infty} \nu^{[2\beta - 2(\alpha \wedge \sigma)] j + 2(\alpha \wedge \sigma)}\right) - c R^{\eta} |d|,
\end{align*}
where we used the assumptions on $u$. By choosing $\beta < \frac{\alpha \wedge \sigma}{4}$ and $\nu > 1$ large enough , we deduce that $u \ge \delta - 8\delta|d| R^{\eta}$ in $D_{\oplus}(R)$, where $\delta = \frac{c}{8} > 0$, as desired.
\end{proof}

We are now in the position to give the proof of the H\"older estimate for weak solutions to \eqref{PDEdual}:

\begin{proof}[Proof of \autoref{thm:mainthmPDEdual} (ii)]
Let us define $\nu = 6 \vee 2 \cdot 2^{\frac{1}{\alpha}} \vee 2^{\frac{1}{\eta}} \vee \nu'$, and $\delta = (1 - \nu^{-(\beta \wedge 1 \wedge \frac{\eta}{2})}) \wedge \delta'$, where $\delta', \beta, \nu'$ are the constants from \autoref{lemma:ioinfdual}, $\eta = \frac{1}{2}\left(\alpha - \frac{d}{\theta}\right)$. Moreover, we set $\gamma = -\frac{1}{2}\log_{\nu}(1-\delta) > 0$. This yields
\begin{align}
(1-\delta) = \nu^{-2\gamma}, ~~ \gamma \le \beta \wedge 1 \wedge \frac{\eta}{2}, ~~ \widehat{D}(t_0,\nu^{-1}R) \subset D_{\oplus}(t_0,R).
\end{align}
The goal of the proof is to construct an increasing sequence $\{m_k\}$ and a decreasing sequence $\{M_k\}$ such that
\begin{align}
\label{eq:HolderGoal1}
m_k &\le u \le M_k, ~~ \text{ in } \widehat{D}(t_0, \nu^{-2k}R),\\
\label{eq:HolderGoal2}
M_k - m_k &= 2\Vert u \Vert_{\infty}\nu^{-2\gamma k} \left(1 + \left[ 1 + \nu^{2(\gamma - \eta)} + \dots + \nu^{2(k-1)(\gamma - \eta)} \right] K R^{\eta}\right),
\end{align}
where $K = \nu^{2\gamma - \eta} \le 1$.\\
Let us prove \eqref{eq:HolderGoal1}, \eqref{eq:HolderGoal2} by induction. By defining $m_0 = - \Vert u \Vert_{\infty}$ and $M_0 = \Vert u \Vert_{\infty}$, we have proved the desired result for $k = 0$. We suppose that the desired result holds true for $j \le k-1$ and intend to deduce that it holds true for $k$.\\
Note that always exactly one of the following two options holds true:
\begin{align*}
\left| D_{\ominus}(\nu^{-2(k-1)-1}R) \cap \left\{ u \ge \frac{M_{k-1} + m_{k-1}}{2} \right\} \right| &\ge \frac{|D_{\ominus}(\nu^{-2(k-1)-1}R)|}{2}\\
\left| D_{\ominus}(\nu^{-2(k-1)-1}R) \cap \left\{ u \le \frac{M_{k-1} + m_{k-1}}{2} \right\} \right| &\ge \frac{|D_{\ominus}(\nu^{-2(k-1)-1}R)|}{2}.
\end{align*}
In the first case, we define $v = \frac{u - m_{k-1}}{M_{k-1} - m_{k-1}}$. The verification of the desired result \eqref{eq:HolderGoal1}, \eqref{eq:HolderGoal2} in the second case goes analogously by defining $v = \frac{M_{k-1} - u}{M_{k-1} - m_{k-1}}$.\\
Let us assume that we are in the first case. Our goal is to apply \autoref{lemma:ioinfdual} to $v$. Let us therefore verify the assumptions:\\
First, note that $v \ge 0$ in $\widehat{D}(\nu^{-2(k-1)-1}R)$ by induction hypothesis.\\
Second, we observe that if $u$ solves \eqref{PDEdual} and $D \in \R$, then $u - D$ solves \eqref{PDEdualext}, where $d \equiv -D$. Therefore, $v$ is a supersolution to \eqref{PDEdualext} in $D(\nu^{-2(k-1)-1}R)$, where $d \equiv \frac{-m_{k-1}}{M_{k-1} - m_{k-1}}$.\\
Third, it clearly holds 
\begin{align*}
|D_{\ominus}(\nu^{-2(k-1)-1}R) \cap \{ v \ge 1/2 \}| \ge \frac{1}{2} |D_{\ominus}(\nu^{-2(k-1)-1}R)|
\end{align*}
by definition of $v$ and by assumption.
Fourth, for every $j \in \N$ with $j \le k-1$ and $(t,x) \in (t_0 - 2(\nu^{-2(k-1)-1}R)^{\alpha}, t_0) \times B_{3\nu^{-2(k-j-1)}}$ it holds by the induction hypothesis:
\begin{align*}
v(t,x) &= \frac{u(t,x) - m_{k-1}}{M_{k-1} - m_{k-1}} \ge \frac{m_{k-j-1} - m_{k-1}}{M_{k-1} - m_{k-1}}\\
&\ge \frac{M_{k-1} - M_{k-j-1} + m_{k-j-1} - m_{k-1}}{M_{k-1} - m_{k-1}} = 1 - \frac{M_{k-j-1} -m_{k-j-1}}{M_{k-1} - m_{k-1}}\\
&= 1 - \frac{2\Vert u \Vert_{\infty}\nu^{-2\gamma (k-j-1)} \left(1 + \left[ 1 + \nu^{2(\gamma - \eta)} + \dots + \nu^{2(k-j-2)(\gamma - \eta)} \right] K R^{\eta}\right)}{2\Vert u \Vert_{\infty} \nu^{-2\gamma (k-1)} \left(1 + \left[ 1 + \nu^{2(\gamma - \eta)} + \dots + \nu^{2(k-2)(\gamma - \eta)} \right] K R^{\eta}\right)}\\
&= 1 - \nu^{2\gamma j} \left(\frac{1 + \left[ 1 + \nu^{2(\gamma - \eta)} + \dots + \nu^{2(k-j-2)(\gamma - \eta)} \right] K R^{\eta}}{1 + \left[ 1 + \nu^{2(\gamma - \eta)} + \dots + \nu^{2(k-2)(\gamma - \eta)} \right] K R^{\eta}} \right)\\
&\ge 1 - \nu^{2\gamma j} \left(1 + (1 - \nu^{-\eta})^{-1} \right)\\
&\ge 1 - 3 \nu^{2\gamma j},
\end{align*}
where we used that $R,K \le 1$, $\gamma \le \frac{\eta}{2}$, and the assumption $\nu \ge 2^{\frac{1}{\eta}}$ which yields $ (1 - \nu^{-\eta})^{-1} \le 2$.\\
For $j > k-1$, we compute
\begin{align*}
v(t,x) &= \frac{u(t,x) - m_{k-1}}{M_{k-1} - m_{k-1}} \ge \frac{(M_{k-1} - m_{k-1}) - (M_0 - m_0)}{M_{k-1} - m_{k-1}} = 1 - \frac{2 \Vert u \Vert_{\infty}}{M_{k-1} - m_{k-1}}\\
&= 1 - \frac{1}{\nu^{-2\gamma (k-1)} \left(1 + \left[ 1 + \nu^{2(\gamma - \frac{\alpha}{2})} + \dots + \nu^{2(k-2)(\gamma - \frac{\alpha}{2})} \right] K R^{\frac{\alpha}{2}}\right)} \\
&\ge 1 - \nu^{2\gamma (k-1)} \ge 1 - \nu^{2\gamma j}.
\end{align*}

Consequently, all assumptions of \autoref{lemma:ioinfdual} are satisfied by $v$ and we deduce that 
\begin{align*}
v \ge \delta - |d|(\nu^{-2(k-1)-1}R)^{\eta} \text{ in } D_{\oplus}(\nu^{-2(k-1)-1}R).
\end{align*}
Recall that by definition of $\nu > 0$, it holds $\widehat{D}(\nu^{-2k}R) \subset D_{\oplus}(\nu^{-2(k-1)-1}R)$, so we have by $m_{k-1} \le 2\Vert u \Vert_{\infty}$:
\begin{align*}
u &\ge m_{k-1} + \delta(M_{k-1} - m_{k-1}) - m_{k-1}(\nu^{-2(k-1)-1} R)^{\eta}\\
&\ge m_{k-1} + \delta(M_{k-1} - m_{k-1}) - 2\Vert u \Vert_{\infty}(\nu^{-2(k-1)-1} R)^{\eta}, ~~ \text{ in } I_{\nu^{-2k}R} \times B_{3 \nu^{-2k}R}.
\end{align*}
We define $m_k := m_{k-1} + \delta(M_{k-1} - m_{k-1}) - 2\Vert u \Vert_{\infty}(\nu^{-2(k-1)-1}R)^{\eta}$ and $M_k := M_{k-1}$. \\
Then, it follows by the induction hypothesis, using $(1-\delta) = \nu^{-2\gamma}$:
\begin{align*}
M_k - m_k &= (M_{k-1} - m_{k-1})(1-\delta) + 2\Vert u \Vert_{\infty}(\nu^{-2(k-1)-1}R)^{\eta}\\
&= 2\Vert u \Vert_{\infty} \nu^{-2\gamma (k-1)} \left(1 + \left[ 1 + \nu^{2(\gamma - \eta)} + \dots + \nu^{2(k-2)(\gamma - \eta)} \right] K R^{\eta}\right)(1-\delta) \\
& \quad + 2\Vert u \Vert_{\infty}(\nu^{-2(k-1)-1}R)^{\eta}\\
&= 2\Vert u \Vert_{\infty}\nu^{-2\gamma k}\left(1 + \left[ 1 + \nu^{2(\gamma - \eta)} + \dots + \nu^{2(k-2)(\gamma - \eta)} \right] K R^{\eta}\right)\\
& \quad + 2\Vert u \Vert_{\infty} \nu^{-2\gamma k} \nu^{2(k-1)(\gamma - \eta)} \nu^{2\gamma - \eta} R^{\eta}\\
&= 2\Vert u \Vert_{\infty} \nu^{-2\gamma k}\left(1 + \left[ 1 + \nu^{2(\gamma - \eta)} + \dots + \nu^{2(k-1)(\gamma - \eta)} \right] K R^{\eta} \right),
\end{align*}
where we used that $K = \nu^{2\gamma - \eta}$.
We have proved \eqref{eq:HolderGoal1}, \eqref{eq:HolderGoal2}.\\
Consequently, we have shown that for every $k \in \N$:
\begin{align*}
\osc_{\widehat{D}(\nu^{-2k}R)} u &\le 2 \Vert u \Vert_{\infty} \nu^{-2\gamma k}\left(1 + \left[ 1 + \nu^{2(\gamma - \eta)} + \dots + \nu^{2(k-1)(\gamma - \eta)} \right] K R^{\eta} \right)\\
&\le 2 \Vert u \Vert_{\infty} \nu^{-2\gamma k}\left(1 + (1 - \nu^{-\eta})^{-1} K R^{\eta} \right),
\end{align*}
where we used the fact that $\gamma \le \frac{\eta}{2}$. In fact, since $K, R \le 1$, we obtain that there exists some $C > 0$ such that
\begin{align*}
\osc_{\widehat{D}(\nu^{-2k}R)} u \le C \Vert u \Vert_{\infty} \nu^{-2\gamma k}.
\end{align*}
From here, the desired result is an immediate consequence.
\end{proof}

\section{Extension to time-inhomogeneous jumping kernels}
\label{sec:time}

In this section, we discuss how to extend \autoref{thm:mainthmPDE} and \autoref{thm:mainthmPDEdual} to equations of type \eqref{PDE} and \eqref{PDEdual}, where $L_t$ is defined as in \eqref{eq:op-time} for some jumping kernel $k : I \times \R^d \times \R^d \to [0,\infty]$ which might depend on time. Let us fix $k$ and assume that there exists $K_s$ such that \eqref{eq:symmetric-time-dependence} holds true. Moreover, assume that $K_s$ satisfies \eqref{Poinc}, \eqref{Sob} and \eqref{cutoff}.

First, we extend the weak solution concept for \eqref{PDE}, \eqref{PDEdual} and \eqref{PDEdualext} to time-inhomogeneous jumping kernels.

\begin{definition}
We say that $u \in L_{loc}^2(I;V(\Omega|\R^d)) \cap L_{loc}^{\infty}(I;L^2(\Omega))$ is a supersolution to \eqref{PDE} in $I \times \Omega$ if $\partial_t u \in L^1_{loc}(I;L^2(\Omega))$ and 
\begin{align*}
(\partial_t u(t),\phi) + \cE^{k(t)}(u(t),\phi) \le (f(t),\phi), \qquad \forall t \in I, ~~ \forall \phi \in H_{\Omega}(\R^d) \text{ with } \phi \le 0.
\end{align*}
Weak subsolutions and solutions, as well as the corresponding notions for \eqref{PDEdual} and \eqref{PDEdualext} are defined analogously.
\end{definition}

In the following, we will interpret $\frac{\theta\alpha}{\theta\alpha-d} = 1$ for $\theta = \infty$.
The main auxiliary result is a time-inhomogeneous analog to \eqref{eq:K1consequence} and \eqref{eq:quantifiedK1consequence}. 

\begin{lemma}
Let $\mu \ge 1$ and $\theta \in [\frac{d}{\alpha},\infty]$. There is $c > 0$ such that for every $\delta > 0$ and $Z \in L^{\mu}(I_r)$  there is $D(\delta) > 0$ such that for every $w \in L^1(I_r)$, where $I_r \subset I$ is an interval with $r \le 1$, it holds:
\begin{align}
\label{eq:K1consequence-time}
\int_{I_r} w(t) Z(t)^{\frac{\theta\alpha}{\theta\alpha - d}} \d t \le \delta^{\frac{\theta\alpha}{\theta\alpha-d}} \sup_{t \in I_r} w(t) + c D(\delta)^{\frac{\theta\alpha}{\theta\alpha-d}} \Vert w \Vert_{L^1(I_r)}.
\end{align}
Moreover, if $\mu > 1$, the constant $D(\delta)$ has the following form:
\begin{align}
\label{eq:quantifiedK1consequence-time}
D(\delta) = \begin{cases}
\Vert Z \Vert_{L^{\infty}(I_r)}, &~~ \mu = \infty\\
\Vert Z \Vert_{L^{\mu}(I_r)}^{\frac{\mu}{\mu-1}} \delta^{\frac{1}{1-\mu}}, &~~ \mu \in (1,\infty)
\end{cases}
\end{align}
\end{lemma}

\begin{proof}
As in the proof of \eqref{eq:K1consequence}, we will split $Z(t) = Z_1(t) + Z_2(t)$, where $Z_1(t) = Z(t) \mathbbm{1}_{\{|Z(t)| > M\}}$ and choose $M > 0$ large enough, such that $\Vert Z_1 \Vert_{L^1(I_r)} < \delta$. This proves \eqref{eq:K1consequence-time}. To prove \eqref{eq:quantifiedK1consequence-time}, by a straightforward computation:
\begin{align*}
\Vert Z_1 \Vert_{L^1(I_r)} \le \Vert Z \Vert_{L^{\mu}(I_r)} |\{ Z \ge M \}|^{1-\frac{1}{\mu}} \le \Vert Z \Vert_{L^{\mu}(I_r)} \left( \frac{\Vert Z \Vert_{L^{\mu}(I_r)} }{M} \right)^{\mu -1} = \Vert Z \Vert_{L^{\mu}(I_r)}^{\mu} M^{1-\mu}.
\end{align*}
In order for $\Vert Z_1\Vert_{L^1(I_r)} \le \delta$ to hold true, we can choose any $M > \delta^{\frac{1}{1-\mu}} \Vert Z \Vert_{L^{\mu}(I_r)}^{\frac{\mu}{\mu-1}}$, which yields the desired result.
\end{proof}

We are now ready to provide a proof of  \autoref{thm:mainthm-time}(i).

\begin{proof}[Proof of \autoref{thm:mainthm-time}(i)]
We assume that \eqref{eq:K1time} holds true. We will use the short notation
\begin{align*}
Z(t) := \left\Vert \int_{B_{r+\rho}} \frac{|k_a(t;\cdot,y)|^2}{J(\cdot,y)} \d y \right\Vert_{L^{\theta}(B_{r+\rho})}
\end{align*}

The first step of the proof is to establish \autoref{thm:MI}(i) for weak supersolutions $u$ to \eqref{PDE} for time-inhomogeneous jumping kernels.
By carefully following the proof of \autoref{lemma:MSgen} and reading off the precise dependence of the constant $C(\delta) > 0$ from \eqref{eq:quantifiedK1consequence} on $Z(t)$, it becomes apparent that the following counterpart of \autoref{lemma:MSgen} holds true for $u \in V(B_{r+\rho}|\R^d)$, given any $p \ge 1 - \kappa^{-1}$, with $p \neq 1$, $\eps > 0$, $t \in I_{R}(t_0)$:
\begin{align}
\label{eq:MSgen-time}
\begin{split}
\cE^{K_s}_{B_{r+\rho}}&(\tau \U^{\frac{-p+1}{2}},\tau \U^{\frac{-p+1}{2}})
\le c_1|p-1| \cE^{k(t)}(u,-\tau^{2}\U^{-p})\\
&\qquad\qquad+ c_2 \left(1 \vee |p-1| \right) \rho^{-\alpha} \Vert \U^{-p+1}\Vert_{L^{1}(B_{r+\rho})} + c_2 Z(t)^{\frac{\theta \alpha}{\theta \alpha - d}} \Vert \tau^2\U^{-p+1}\Vert_{L^{1}(B_{r+\rho})}.
\end{split}
\end{align}
We test the weak formulation of the equation for $u$ with $\phi = -\tau^2 \U^{q-1}$ for some $q > 0$, multiply with an adequate cut-off function in time $\chi$, and integrate in time, as in the proof of \autoref{thm:MI}. By using \eqref{eq:MSgen-time}, we obtain:
\begin{align}
\label{eq:timeMShelp1}
\begin{split}
\sup_{t \in I_{r+\rho}^{\ominus}(t_0)}\int_{B_{r+\rho}}& \chi^2(t) \tau^2(x) \U^{-q}(t,x) \d x + \int_{I_{r+\rho}^{\ominus}(t_0)} \chi^2(s) \cEs_{B_{r+\rho}}(\tau \U^{-\frac{q}{2}}(s),\tau \U^{-\frac{q}{2}}(s)) \d s\\
&\le c_3 (1 \vee q)(\rho^{-\alpha} \vee ((r+\rho)^{\alpha} - r^{\alpha})^{-1})  \Vert \U^{-q} \Vert_{L^{1}(I_{r+\rho}^{\ominus} \times B_{r+\rho})}\\
& \quad + c_3 \int_{I_{r+\rho}^{\ominus}(t_0)} \chi^2(s)  Z(s)^{\frac{\theta \alpha}{\theta \alpha - d}}  \Vert \tau^2\U^{-q}(s) \Vert_{L^1( B_{r+\rho})} \d s.
\end{split}
\end{align}
An application of \eqref{eq:K1consequence-time} with $\delta = \delta_0 c_3^{-1}$ for some sufficiently small $\delta_0 \in (0,1)$ yields
\begin{align*}
c_3 &\int_{I_{r+\rho}^{\ominus}(t_0)} \chi^2(s)  Z(s)^{\frac{\theta \alpha}{\theta \alpha - d}}  \Vert \U^{-q}(s) \Vert_{L^1( B_{r+\rho})} \d s\\
&\le \frac{1}{2}\sup_{t \in I_{r+\rho}^{\ominus}(t_0)}\int_{B_{r+\rho}}\chi^2(t) \tau^2(x) \U^{-q}(t,x) \d x + c_4 \Vert \U^{-q} \Vert_{L^{1}(I_{r+\rho}^{\ominus} \times B_{r+\rho})},
\end{align*}
where $c_4 = c_3D(\delta)> 0$. Therefore, we have shown that 
\begin{align}
\label{eq:MS1timefinal}
\begin{split}
\sup_{t \in I^{\ominus}_r} &\int_{B_{r}} \U^{-q}(t,x) \d x + \int_{I^{\ominus}_r} \cEs_{B_{r+\rho}}(\tau \U^{-\frac{q}{2}}(s),\tau \U^{-\frac{q}{2}}(s)) \d s\\
&\le c_5 (1 \vee q^{\gamma}) \left(\rho^{-\alpha} \vee ((r+\rho)^{\alpha}-r^{\alpha})^{-1}\right) \int_{I^{\ominus}_{r+\rho}} \Vert \U^{-q}(s)\Vert_{L^{1}(B_{r+\rho})} \d s.
\end{split}
\end{align}
From here, the proof of \autoref{thm:MI}(i) goes through without any further change. Moreover, note that an analog to \autoref{thm:MI2}(i) can be proved along the following arguments.

It remains to show \autoref{thm:BMO}(i) for weak supersolutions $u$ to \eqref{PDE} for time-inhomogeneous jumping kernels. From here, one can deduce the desired result by following the arguments from the proof of \autoref{thm:mainthmPDE}.\\
By following the proof of \autoref{lemma:MSgen2}, and applying H\"older's inequality instead of \eqref{eq:K1consequence} in the estimate of $J_a$, we immediately obtain that for every $t \in I_{R}(t_0)$
\begin{align}
\begin{split}
c_1\int_{B_{3R/2}} \int_{B_{3R/2}} &(\tau^2(x) \wedge \tau^2(y)) \left(\log\frac{\U(x)}{\tau(x)}-\log\frac{\U(y)}{\tau(y)}\right)^2 K_s(x,y) \d y \d x\\
&\le \cE^{k(t)}(u,-\tau^2 \U^{-1}) + c_2 R^{-\alpha}\vert B_{R}\vert + c_2 Z(t)|B_{R}|^{\frac{1}{\theta'}}.
\end{split}
\end{align}
by testing the weak formulation of the equation with $\phi = -\tau^2 \U^{-1}$ and defining 
\begin{align*}
v(t,x) = -\log\frac{\U(t,x)}{\tau(x)}, \quad V(t) = \frac{\int_{B_{3R/2}} v^2(t,x) \tau^2(x) \d x}{\int_{B_{3R/2}} \tau^2(x) \d x},
\end{align*}
we obtain the following estimate by applying the same arguments as in the proof of \autoref{thm:BMO}:
\begin{align}
\label{eq:BMOhelp-time}
V(t_2) - V(t_1) + c R^{-d-\alpha} \int_{t_1}^{t_2} \int_{B_R} (v(t,x) - V(t))^2 \d x \le \int_{t_1}^{t_2} B(t) 
d t,
\end{align}
where $[t_1,t_2] \subset I_R$ and
\begin{align*}
B(t) = c_2 R^{-\alpha} + c_2 Z(t) |B_{R}|^{-\frac{1}{\theta}}.
\end{align*}
Next, we define
\begin{align*}
w(t,x) = v(t,x) - \int_{t_0}^{t} B(\tau) \d \tau, \quad W(t) = V(t) - \int_{t_0}^{t} B(\tau) \d \tau.
\end{align*}
The following estimate is a standard consequence of the definitions of $w,W$ and the estimate \eqref{eq:BMOhelp-time}, see \cite{FeKa13}, \cite{CKW19} and p.109 in \cite{ArSe67}:
\begin{align*}
|I_{r}^{\oplus}(t_0) \times B_R \cap \{w \ge a + s \}| \le c_3\frac{R^{d+\alpha}}{s}.
\end{align*}
In order to deduce the desired result, it therefore remains to prove that
\begin{align*}
\left| I_{r}^{\oplus}(t_0) \times B_R \cap \left\{ \int_{t_0}^t c_2 Z(t) |B_R|^{-\frac{1}{\theta}} \d \tau \ge \frac{s}{4} \right\}\right| \le c \frac{R^{d+\alpha}}{s}.
\end{align*}
By application of \eqref{K1}, we can deduce
\begin{align*}
\int_{t_0}^t c_2 Z(t) |B_R|^{-\frac{1}{\theta}} \d \tau \le c_4(t-t_0)^{\frac{1}{\mu'}} \Vert Z \Vert_{L^{\mu}(I_r^{\oplus}(t_0))} R^{-\frac{d}{\theta}} \le c_5 (t-t_0)^{\frac{1}{\mu'}}R^{-\frac{d}{\theta}}.
\end{align*}
Since by \eqref{eq:compatibility} it holds that $\frac{d\mu'}{\theta} < \alpha$, we deduce that
\begin{align*}
\left| I_{r}^{\oplus}(t_0) \times B_R \cap \left\{ \int_{t_0}^t c_2 Z(t) |B_R|^{-\frac{1}{\theta}} \d \tau \ge \frac{s}{4} \right\}\right| &\le \left| I_{r}^{\oplus}(t_0) \times B_R \cap \left\{ t-t_0 \ge c_6 R^{\frac{d\mu'}{\theta}} s^{\mu'} \right\} \right|\\
&\le c_7\left[(1-c_6 s^{\mu'}) \vee 0 \right]R^{d+\alpha} \le c_8 \frac{R^{d+\alpha}}{s} \,,
\end{align*}
where we applied the elementary estimate $a - b^{\mu'} \le c \frac{a^{2 - \frac{1}{\mu}}}{b}$ in the last step. These arguments suffice to deduce \eqref{eq:BMO1} and \eqref{eq:BMO2} in our setting. This yields the desired result.
\end{proof}

We end this section by providing a proof of \autoref{thm:mainthm-time}(ii). It follows the same structure as the proof of \autoref{thm:mainthm-time}(i) but some additional care is required when dealing with weak solutions to \eqref{PDEdualext}.

\begin{proof}[Proof of \autoref{thm:mainthm-time}(ii)]
We assume \eqref{eq:K1globtime} and denote
\begin{align*}
Z(t) := \left\Vert \int_{\R^d} \frac{|k_a(t;\cdot,y)|^2}{J(\cdot,y)} \d y \right\Vert_{L^{\theta}(\R^d)}.
\end{align*}
First of all, we can read off from the proof of \autoref{lemma:MSgendual} that for $u \in V(B_{r+\rho}|\R^d)$, given any $p \ge 1 - \kappa^{-1}$, with $p \neq 1$, $\eps > 0$, $t \in I_R(t_0)$:
\begin{align}
\label{eq:MSgendual-time}
\begin{split}
\cE^{K_s}_{B_{r+\rho}}&(\tau \U^{\frac{-p+1}{2}},\tau \U^{\frac{-p+1}{2}})
\le c_1|p-1| \widehat{\cE}^{k(t)}(u,-\tau^{2}\U^{-p})\\
&\qquad\qquad\qquad+ c_2 \left(1 \vee p^{\gamma} \right) \left[\rho^{-\alpha}\Vert \U^{-p+1}\Vert_{L^{1}(B_{r+\rho})} + Z(t)^{\frac{\theta \alpha}{\theta \alpha - d}}\Vert \tau^2 \U^{-p+1}\Vert_{L^{1}(B_{r+\rho})}\right] .
\end{split}
\end{align}
Moreover, the proof of \autoref{lemma:MSgendualext} yields that for any $\delta \in (0,1)$:
\begin{align}
\label{eq:MSgendualext-time}
\begin{split}
-\delta \cE^{K_s}_{B_{r+\rho}}&(\tau \U^{\frac{-p+1}{2}},\tau \U^{\frac{-p+1}{2}}) \le c_3 |p-1| \widehat{\cE}^{k_a(t)}(d,-\tau^{2} \U^{-p})\\
&\qquad+ c_4 \left( 1 \vee p^{\gamma} \right)\left[\rho^{-\alpha}\Vert \U^{-p+1}\Vert_{L^1(B_{r+\rho})} + [\rho^{-2\eta} Z(t)]^{\frac{\theta\alpha}{\theta\alpha-d}}\Vert \tau^2 \U^{-p+1}\Vert_{L^1(B_{r+\rho})} \right].
\end{split}
\end{align}
Here, $\U = u + \eps + R^{\alpha}\Vert f \Vert_{L^{\infty}} + R^{\eta} \Vert d \Vert_{L^{\infty}}$, where $\eta = \frac{1}{2} (\alpha -\frac{d}{\theta}-\frac{\alpha}{\mu}) > 0$. The term $[\rho^{-2\eta} Z(t)]^{\frac{\theta\alpha}{\theta\alpha-d}}$ arises from an application of \eqref{eq:K1consequence} with $\delta = \rho^{2\eta}$.
Moreover, an application of \eqref{eq:K1consequence-time} with $\delta = \delta_0 c^{-1}(1 \vee q^{\gamma})^{-1} \rho^{2\eta}$ for some sufficiently small $\delta_0 \in (0,1)$ yields for any $q > 0$:
\begin{align}
\label{eq:MSdualexttimehelp}
\begin{split}
c_4 (1 \vee q^{\gamma}) &\int_{I_{r+\rho}^{\ominus}(t_0)} \chi^2(s) [\rho^{-2\eta} Z(s)]^{\frac{\theta \alpha}{\theta \alpha - d}}  \Vert \tau^2\U^{-q}(s) \Vert_{L^1( B_{r+\rho})} \d s\\
&\le \frac{1}{2}\sup_{t \in I_{r+\rho}^{\ominus}(t_0)}\int_{B_{r+\rho}}\chi^2(t) \tau^2(x) \U^{-q}(t,x) \d x + c_5 (1 \vee q^{\gamma_2}) \rho^{-\alpha} \Vert \U^{-q} \Vert_{L^{1}(I_{r+\rho}^{\ominus} \times B_{r+\rho})}
\end{split}
\end{align}
where $c_1 > 0$ is some constant, $\gamma_2 = \frac{\gamma}{\mu - 1} \frac{\theta \alpha}{\theta \alpha - d}$ and we used that $2\eta \frac{\mu}{\mu-1}\frac{\theta\alpha}{\theta\alpha-d} \le \alpha$. Note that \eqref{eq:MSdualexttimehelp} would follow for any value $0 < \eta \le \frac{1}{2} \frac{(\mu-1)(\theta\alpha-d)}{\mu\theta} = \frac{1}{2}(\alpha - \frac{d}{\theta} - \frac{\alpha}{\mu} + \frac{d}{\mu\theta})$. By testing the equation for $u$ with $\phi = -\tau^2 \U^{q-1}$, and proceeding as in the proof of \autoref{thm:mainthm-time}, we obtain \eqref{eq:MS1timefinal} after combining \eqref{eq:MSgendual-time}, \eqref{eq:MSgendualext-time} and \eqref{eq:MSdualexttimehelp}. This proves \autoref{thm:MI}(ii) for weak supersolutions to \eqref{PDEdualext} with time-dependent jumping kernels. The proof of \autoref{thm:MI2}(ii) follows along the same arguments.\\
It remains to show \autoref{thm:BMO}(ii) for weak supersolutions $u$ to \eqref{PDE} for time-inhomogeneous jumping kernels. By following the proof of \autoref{lemma:MSgendual2} and applying H\"older's inequality instead of \eqref{eq:K1consequence}, we obtain that for every $t \in I_R(t_0)$:
\begin{align}
\begin{split}
c_1\int_{B_{3R/2}} \int_{B_{3R/2}} &(\tau^2(x) \wedge \tau^2(y)) \left(\log\frac{\U(x)}{\tau(x)}-\log\frac{\U(y)}{\tau(y)}\right)^2 K_s(x,y) \d y \d x\\
&\le \widehat{\cE}^{k(t)}(u,-\tau^2 \U^{-1}) + c_2 R^{-\alpha}\vert B_{R}\vert + c_2 Z(t)|B_{R}|^{\frac{1}{\theta'}}.
\end{split}
\end{align}
By the same argument, we can deduce from the proof of \autoref{lemma:MSgendualext2} that for any $\delta \in (0,1)$:
\begin{align*}
-\delta\int_{B_{2R/2}} \int_{B_{3R/2}} &(\tau^2(x) \wedge \tau^2(y)) \left(\log\frac{\U(x)}{\tau(x)}-\log\frac{\U(y)}{\tau(y)}\right)^2 K_s(x,y) \d y \d x\\
&\le \widehat{\cE}^{k_a(t)}(d,-\tau^2 \U^{-1}) + c_2 R^{-\alpha}\vert B_{R}\vert + c_2 R^{-2\eta} Z(t)|B_{R}|^{\frac{1}{\theta'}}.
\end{align*}
In order to deduce \autoref{thm:BMO}, we proceed as in the proof of (i). It remains to estimate
\begin{align*}
&\left| I_{r}^{\oplus}(t_0) \times B_R \cap \left\{ \int_{t_0}^t c_2 R^{-2\eta} Z(t) |B_R|^{-\frac{1}{\theta}} \d \tau \ge \frac{s}{4} \right\}\right| \\
&\le \left| I_{r}^{\oplus}(t_0) \times B_R \cap \left\{ \int_{t_0}^t c_3 R^{-\alpha + \frac{\alpha}{\mu}} \ge \frac{s}{4} \right\}\right| = \left| I_{r}^{\oplus}(t_0) \times B_R \cap \left\{ t-t_0 \ge c_4 R^{\alpha} s^{\mu'} \right\} \right|\\
&\le c_5\left[(1-c_4 s^{\mu'}) \vee 0 \right]R^{d+\alpha} \le c_6 \frac{R^{d+\alpha}}{s} \,.
\end{align*}
This proves \autoref{thm:BMO}(ii). As in the classical case, from here we can deduce the desired results as in the proof of \autoref{thm:mainthmPDEdual}. However, note that due to the time-inhomogeneity, the corresponding proofs have to be carried out with $\eta = \frac{1}{2} (\alpha - \frac{d}{\theta} - \frac{\alpha}{\mu})$.
\end{proof}

\section{Approximation of local objects}
\label{sec:approx}

In this section we demonstrate that in the limit $\alpha \nearrow 2$ the operators given by \eqref{eq:op} approximate second order divergence form operators with a drift term if an assumption reminiscent of \eqref{K1glob} is satisfied uniformly in $\alpha$. The main achievement of this section is a  convergence result in the spirit of Mosco convergence for nonsymmetric forms, see \autoref{thm:Mosco}.

Let $B$ be a bounded open set $B \subset \R^d$ with a Lipschitz boundary and $\alpha_0 \in (0,2)$. Consider a family of kernels $(K^{(\alpha)})_{\alpha \in (\alpha_0,2)}$, $K^{(\alpha)} : \R^d \times \R^d \to [0,\infty]$ satisfying
\begin{align}
\label{eq:globalK1}
\sup_{\alpha \in (\alpha_0,2)}\left\Vert \int_{B} \frac{\vert K_a^{(\alpha)}(\cdot,y)\vert^{2}}{K^{(\alpha)}_s(\cdot,y)} \d y \right\Vert_{L^{\theta}(B)} =:C < \infty,
\end{align}
where $\theta \in [\frac{d}{\alpha},\infty]$ is allowed to depend on $\alpha$.
Furthermore, we assume the following pointwise upper and lower bound on the symmetric part:
\begin{align}
\label{eq:kernelcomp}
\Lambda^{-1}(2-\alpha)\vert x-y \vert^{-d-\alpha} \le K^{(\alpha)}_s(x,y) \le \Lambda(2-\alpha)\vert x-y \vert^{-d-\alpha}
\end{align}
for some constant $\Lambda \ge 1$ and every $x,y \in B$. We define $\cE^{(\alpha)}_B$ via 
\begin{equation*}
\cE^{(\alpha)}_B(u,v) = \int_{B}\int_{B}(u(x)-u(y))v(x) K^{(\alpha)}(x,y) \d y \d x.
\end{equation*} 
Note that the following Sobolev inequality holds true by \eqref{eq:kernelcomp}:
\begin{align}
\label{eq:MoscoSob}
\Vert v^2 \Vert_{L^{\frac{d}{d-\alpha}}(B)} \le c\cE^{K_s^{(\alpha)}}_{B}(v,v) + c \vert B \vert^{-\frac{d}{\alpha}}\int_{B} v^2(x) \d x.
\end{align}
Let us prove that the conditions \eqref{eq:globalK1} and \eqref{eq:kernelcomp} are sufficient for $(\cE^{(\alpha)}_B,H^{\alpha/2}(B))$ to be regular lower bounded semi-Dirichlet forms on $L^2(B)$ (see \cite{Osh13}, \cite{MaRo92}, and also \cite{Sta65}).\\
First, in \cite{ScWa15}, Schilling and Wang proved this fact in the case $\theta = \infty$ (see \autoref{prop:SchillingWang}). The proof amounts to the establishment of a sector condition and a G\r{a}rding-type inequality for $(\cE^{(\alpha)}_B,H^{\alpha/2}(B))$. The remaining properties follow by the same arguments as in \cite{FuUe12}.\\
In analogy to the case $\theta = \infty$, also when $\theta \in [\frac{d}{\alpha},\infty)$ it suffices to prove a sector condition and a G\r{a}rding-type inequality. We restrict ourselves to the situation of $B$ being bounded. Note that due to \eqref{eq:globalK1}, for every $\eps > 0$ there exist $c(\eps) > 0$, $W_1^{(\alpha)} \in L^{\frac{d}{\alpha}}(B)$, $W_2^{(\alpha)} \in L^{\infty}(B)$ such that
\begin{align}
\label{eq:Gardinghelp}
\Vert W_1^{(\alpha)} \Vert_{L^{\frac{d}{\alpha}}(B)} < \eps,~~ \Vert W_2^{(\alpha)} \Vert_{L^{\infty}(B)} < c(\eps),~~ W_1^{(\alpha)}(x) + W_2^{(\alpha)}(x) = \int_{B} \frac{\vert K_a^{(\alpha)}(\cdot,y)\vert^{2}}{K^{(\alpha)}_s(\cdot,y)} \d y.
\end{align}
In this case, the sector condition is a direct consequence of the following estimate:
\begin{align}
\label{eq:SCGIhelp}
\begin{split}
\cE^{K_a^{(\alpha)}}_B(u,v)^2 &= \left(\int_B\int_B (u(x)-u(y))v(x) K^{(\alpha)}_a(x,y) \d y \d x\right)^2 \\
&\le \cE^{K_s^{(\alpha)}}_B(u,u)\int_B v^2(x) \left(\int_B \frac{\vert K^{(\alpha)}_a(x,y)\vert^2}{K^{(\alpha)}_s(x,y)}\d y\right) \d x\\
&\le \cE^{K^{(\alpha)}_s}_B(u,u) \left( \eps \Vert v^2 \Vert_{L^{\frac{d}{d-\alpha}}(B)} + c(\eps)\Vert v \Vert^2_{L^{2}(B)}\right)\\
&\le c\cE^{K_s^{(\alpha)}}_B(u,u)\left(c \eps \cE^{K_s^{(\alpha)}}_B(v,v) + (c(\eps)+ \vert B \vert^{-\frac{d}{\alpha}})\Vert v\Vert_{L^2(B)}^2 \right),
\end{split}
\end{align}
where we applied \eqref{eq:MoscoSob}.
Moreover, by application of Young's inequality on \eqref{eq:SCGIhelp} and choosing $\eps > 0$ small enough in \eqref{eq:Gardinghelp} we get that the following G\r{a}rding-type inequality holds uniformly in $\alpha$, i.e. there exists $\lambda > 1$ (depending on $\vert B \vert$), such that 
\begin{align}
\label{eq:unifGarding}
\cE^{(\alpha)}_B(u,u) \ge \frac{1}{2}\cE^{K_s^{(\alpha)}}_B(u,u) - (\lambda-1) \Vert u \Vert^2_{L^2(B)}.
\end{align}
Therefore, also when $\theta \in [\frac{d}{\alpha},\infty)$ and $B$ is bounded, $(\cE^{(\alpha)}_B,H^{\alpha/2}(B))$ are regular lower bounded semi-Dirichlet forms.

Our goal is to prove the following theorem:

\begin{theorem}
\label{thm:Mosco}
Let $B \subset \R^d$ be an open bounded set with Lipschitz boundary. Assume  \eqref{eq:globalK1} and \eqref{eq:kernelcomp}. Then the sequence of regular lower bounded semi-Dirichlet forms $(\cE^{(\alpha)}_B,H^{\alpha/2}(B))$ converges to the regular lower bounded semi-Dirichlet form $(\cE,H^1(B))$ in the Mosco-Hino-sense in $L^2(B)$ as $\alpha \nearrow 2$, where we define for $\delta > 0$ and $x \in B$:
\begin{align}
\label{eq:aij}
a_{i,j}(x) &:= \lim_{\alpha \nearrow 2} \int_{B_{\delta}(0)} (-h_i)(-h_j) K^{(\alpha)}_s(x,x+h) \d h,\\
\label{eq:bi}
b_i(x) &:= \lim_{\alpha \nearrow 2} \int_{B_{\delta}(0)} (-h_i) K^{(\alpha)}_a(x,x+h) \d h,
\end{align}
if the limits exist and $\cE$ is defined as
\begin{equation*}
\cE(u,v) = \int_{B} a_{i,j}(x) \partial_i u(x) \partial_j u(x) \d x + 2\int_{B} b_i(x) \partial_i u(x) v(x) \d x.
\end{equation*}
\end{theorem}

\begin{remark*}
The limits in \eqref{eq:aij}, \eqref{eq:bi} have to be understood in a pointwise a.e.-sense. Both limits do not depend on $\delta$.
\end{remark*}

The concept of Mosco-convergence was developed in \cite{Mos94} and extended to nonsymmetric forms in \cite{Hin98}. The fact that $(\cE^{(\alpha)}_B,H^{\alpha/2}(B)) \to (\cE,H^1(B))$ in the Mosco-Hino-sense implies that the corresponding resolvents $(G^{(\alpha)}_{\lambda})$, as well as the semigroups $(T^{(\alpha)}_t)$ converge in $L^2(B)$ towards $(G_{\lambda})$ respectively $(T_t)$. Moreover, the corresponding dual resolvents $(\widehat{G}^{(\alpha)}_{\lambda})$, and dual semigroups $(\widehat{T}^{(\alpha)}_t)$ converge weakly in $L^2(B)$ towards $(\widehat{G}_{\lambda})$ respectively $(\widehat{T}_t)$, see Theorem 3.1 in \cite{Hin98} and Remark 7.17 in \cite{Toe10}. Therefore, Mosco-convergence is a considerably stronger property than mere convergence of $\cE^{(\alpha)}_B(u,v) \to \cE(u,v)$ for arbitrary $u,v$ (see \autoref{lemma:ptwconv}). Nonetheless, the latter is an important ingredient in the proof of Mosco-convergence.

\begin{lemma}
\label{lemma:ptwconv}
Let $B \subset \R^d$ be an open bounded set with Lipschitz boundary. Assume  \eqref{eq:globalK1} and \eqref{eq:kernelcomp}. Then for every $u,v \in H^1(B)$ it holds
\begin{align}
\label{eq:limKs}
\lim_{\alpha \nearrow 2} \cE^{K_s^{(\alpha)}}_{B}(u,v) &= \int_B a_{i,j}(x) \partial_i u(x) \partial_j v(x) \d x := \cE^{a_{i,j}}_B(u,v),\\
\label{eq:limKa}
\lim_{\alpha \nearrow 2} \cE^{K_a^{(\alpha)}}_{B}(u,v) &= 2\int_B b_i(x) \partial_i u(x) v(x) \d x := \cE^{b_{i}}_B(u,v),
\end{align}
where $a_{i,j}, b_i$ are defined as in \eqref{eq:aij}, \eqref{eq:bi}.\\
Moreover, $a(x) := \left(a_{i,j}(x)\right)_{i,j}$ is a symmetric matrix for a.e. $x \in B$, uniformly positive definite and has bounded entries, and $b_i \in L^{2\theta_0}(B)$, where $\theta_0 := \inf_{\alpha \in (\alpha_0,2)} \theta(\alpha)$ and $\alpha_0 \in (0,2)$.
\end{lemma}

\begin{proof}
First, we point out that \eqref{eq:limKs} was already proved in \cite{GKV20} in a more general framework.
We prove \eqref{eq:limKa} for $u,v \in C_c^2(\overline{B})$ and conclude by a density argument. We write
\begin{align*}
\cE^{K_a^{(\alpha)}}_{B}(u,v) &= \int_{B}\int_{B \cap \lbrace \vert x-y \vert > 1\rbrace} (u(x)-u(y))(v(x)+v(y))K_a^{(\alpha)}(x,y) \d y \d x\\
& \quad + \int_{B}\int_{B \cap \lbrace \vert x-y \vert \le 1\rbrace} (u(x)-u(y))(v(x)+v(y))K_a^{(\alpha)}(x,y) \d y \d x.
\end{align*}

Using \eqref{eq:KaKs}, and \eqref{eq:kernelcomp} we obtain
\begin{align}
\label{eq:ptwhelp1}
\begin{split}
\int_{B}&\int_{B \cap \lbrace \vert x-y \vert > 1\rbrace} (u(x)-u(y))(v(x)+v(y))K_a^{(\alpha)}(x,y) \d y \d x\\
&\le \frac{1}{2}\int_{B}\int_{B \cap \lbrace \vert x-y \vert > 1\rbrace} (u(x)-u(y))^2\vert K_a^{(\alpha)}(x,y)\vert \d y \d x\\
& \quad + \frac{1}{2}\int_{B}\int_{B \cap \lbrace \vert x-y \vert > 1\rbrace} (v(x)+v(y))^2\vert K_a^{(\alpha)}(x,y)\vert \d y \d x\\
&\le 2\Lambda (2-\alpha)\int_{B} (u^2(x) + v^2(x))\left(\int_{B \cap \lbrace \vert x-y \vert > 1\rbrace} \vert x-y\vert^{-d-\alpha} \d y\right) \d x \quad \to 0 \text{ as } \alpha \nearrow 2.
\end{split}
\end{align}
On the other hand, by Taylor's formula, for $x,y \in B$ with $\vert x-y\vert \le 1$ there is a bounded remainder $r(x,y)$ such that
\begin{equation*}
u(x)-u(y) = \nabla u(x)(x-y) +r(x,y)\vert x-y\vert^2.
\end{equation*}
Consequently:
\begin{align}
\label{eq:ptwhelp2}
\begin{split}
\int_{B}&\int_{B \cap \lbrace \vert x-y \vert \le 1\rbrace} (u(x)-u(y))(v(x)+v(y))K_a^{(\alpha)}(x,y) \d y \d x\\
& = 2\int_{B}\int_{B \cap \lbrace \vert x-y \vert \le 1\rbrace} \left(\nabla u(x)(x-y) +r(x,y)\vert x-y\vert^2\right)v(x)K_a^{(\alpha)}(x,y) \d y \d x\\
&= 2\int_{B}v(x) \nabla u(x) \left(\int_{B \cap \lbrace \vert x-y \vert \le 1\rbrace}(x-y)K_a^{(\alpha)}(x,y) \d y\right) \d x\\
& \quad + 2\int_{B}v(x)\int_{B \cap \lbrace \vert x-y \vert \le 1\rbrace}r(x,y)\vert x-y\vert^2K_a^{(\alpha)}(x,y) \d y \d x.
\end{split}
\end{align}
The first term converges to $2\int_B v(x) \partial_i u(x) b_i(x) \d x$ by definition of $b$ and dominated convergence. For the second term we have by \eqref{eq:kernelcomp} and \eqref{eq:globalK1}:
\begin{align}
\label{eq:ptwhelp3}
\begin{split}
&\int_B v(x) \int_{B \cap \lbrace \vert x-y \vert \le 1\rbrace}r(x,y)\vert x-y\vert^2K_a^{(\alpha)}(x,y) \d y \d x\\
&\le c_1\Vert r \Vert_{\infty} \sup_{x \in B}\left((2-\alpha)\int_{B \cap \lbrace \vert x-y \vert \le 1\rbrace} \hspace*{-3ex}  \vert x-y \vert^{4-d-\alpha} \d y \right)^{1/2} \int_B v(x)\left(\int_B \frac{\vert K^{(\alpha)}_a(x,y)\vert^{2}}{K^{(\alpha)}_s(x,y)}\d y\right)^{1/2}\d x\\
&\le c_2 \sup_{x \in B} \left((2-\alpha)\int_{B \cap \lbrace \vert x-y \vert \le 1\rbrace} \vert x-y \vert^{4-d-\alpha} \d y \right)^{1/2} \left(\Vert v \Vert_{L^{2}(B)} + \cE_B^{K_s^{(\alpha)}}(v,v) \right)\\
&\to 0 \text{ as } \alpha \nearrow 2,
\end{split}
\end{align}
where $c_1, c_2 > 0$ are constants and we used a similar argument as in \eqref{eq:SCGIhelp} in the last estimate. \\
It remains to prove the second part of the assertion. The desired properties of $a_{i,j}$ follow from Proposition 3.1 in \cite{GKV20}. Next, we estimate:
\begin{align*}
b_i(x) = \lim_{\alpha \nearrow 2} \int_{B_{\delta}(0)} (-h_i) K^{(\alpha)}_a(x,x+h) \d h \le \vert a_{i,i}(x)\vert^{1/2} \left( \sup_{\alpha \in (\alpha_0,2)}\int_{B_{\delta}(0)} \frac{\vert K_a^{(\alpha)}(x,x+h)\vert^2}{K_s^{(\alpha)}(x,x+h)} \d h \right)^{1/2},
\end{align*}
which, by \eqref{eq:globalK1} and boundedness of $a_{i,j}$, implies
\begin{align*}
\Vert b_i \Vert_{L^{2\theta_0}(B)} \le \Vert a_{i,i}\Vert_{\infty}^{1/2} \left(\sup_{\alpha \in (\alpha_0,2)}\int_{B} \left(\int_{B_{\delta}(0)} \frac{\vert K_a^{(\alpha)}(x,x+h)\vert^2}{K_s^{(\alpha)}(x,x+h)} \d h \right)^{\theta_0} \d x\right)^{1/(2\theta_0)} < \infty.
\end{align*}
This concludes the proof.
\end{proof}

We are ready to prove \autoref{thm:Mosco}. We point out that related results were proved in \cite{GKV20}, \cite{Gou20} for symmetric Dirichlet forms. Although Mosco-Hino convergence (see \cite{Hin98}) requires the verification of slightly stronger properties than for the symmetric analog (see \cite{Mos94}), the arguments in our proof are reminiscent of those in \cite{Gou20}.

\begin{proof}(of \autoref{thm:Mosco})
Throughout this proof, we introduce the notation $\cE^s(u,v) = \frac{1}{2}(\cE(u,v) + \cE(v,u))$ for the symmetric part of a bilinear form $\cE$.\\
First, we prove the Mosco convergence according to \cite{Hin98}. Let $(\alpha_n)_n \subset (0,2)$ be a sequence such that $\alpha_n \nearrow 2$ as $n \to \infty$. It suffices to establish the following two properties:
\vspace{-0.2cm}
\begin{itemize}
\item[(i)] For every sequence $(v_n) \subset L^2(B)$ with $v_n \in H^{\alpha_n/2}(B)$ and $\sup_n \cE_B^{(\alpha_n)}(v_n,v_n) < \infty$ and every $v \in H^1(B)$ such that $v_n \rightharpoonup v$ in $L^2(B)$ and every $u \in C^{2}_c(\overline{B})$ it holds
\begin{equation*}
\liminf_{n \to \infty} \cE^{(\alpha_n)}_B(u,v_n) \le \cE(u,v).
\end{equation*}

\item[(ii)] For every sequence $(u_n) \subset L^2(B)$ with $\liminf_{n \to \infty} \cE_B^{(\alpha_n)}(u_n,u_n) < \infty$ and every $u \in L^2(B)$ such that $u_n \rightharpoonup u$ in $L^2(B)$ it holds
\begin{equation*}
\cE(u,u) \le \liminf_{n \to \infty} \cE^{(\alpha_n)}_B(u_n,u_n).
\end{equation*}
\end{itemize}

First, we prove (i). Note that $\sup_n \Vert v_n\Vert_{L^2(B)} < \infty$ by the weak convergence $v_n \rightharpoonup v$. Thus, due to G\r{a}rding's inequality \eqref{eq:unifGarding} and the fact that $\sup_n \cE^{(\alpha_n)}_B(v_n,v_n) < \infty$, we conclude
\begin{align}
\label{eq:Moscohelp0}
\sup_n \cE^{K_s^{(\alpha_n)}}_B(v_n,v_n) + \Vert v_n\Vert_{L^2(B)}^2 < \infty.
\end{align}
Since $\cE^{K_s^{(\alpha_n)}}_B(u,u) \to \cE^{a_{i,j}}(u,u)$ for every $u \in C^{2}_c(\overline{B})$ by \autoref{lemma:ptwconv}, we conclude that there exists a subsequence $(v_{n_k}) \subset (v_n)$ and $v' \in H^1(B)$ such that $v_{n_k} \rightharpoonup v'$ in $H^{\alpha/2}(B)$. In particular, for every $u \in C^{2}_c(\overline{B})$ (see Theorem 5.59 in \cite{Gou20}, resp. Lemma 2.2 in \cite{KuSh03}):
\begin{align}
\label{eq:Moscohelp1}
\cE^{K_s^{(\alpha_{n_k})}}_B(u, v_{n_k}) \to \cE^{a_{i,j}}(u,v').
\end{align}
The weak convergence $v_n \rightharpoonup v$ in $L^2(B)$ yields that $v' = v$.\\
Furthermore, it is easy to see that upon replacing $v$ by $v_n$ the proof of \eqref{eq:limKa} yields
\begin{align}
\label{eq:Moscohelp2}
\cE^{K_a^{(\alpha_n)}}_B(u,v_n) \to \cE^{b_i}(u,v), ~~ u \in C_c^{2}(\overline{B}).
\end{align}
Indeed, \eqref{eq:ptwhelp1} remains valid since $\sup_x \left((2-\alpha_n)\int_{\lbrace \vert x-y \vert > 1\rbrace} \vert x-y\vert^{-d-\alpha_n} \d y\right) \to 0$ and due to \eqref{eq:Moscohelp0}. Also, the first summand in \eqref{eq:ptwhelp2} tends to $2\int_B v(x) \nabla u(x) b(x) \d x$ since the convergence in \eqref{eq:bi} is pointwise, $b \in L^{2\theta_0}(B)$, $\nabla u$ is bounded and therefore $b\nabla u \in L^2(B)$, $$\left(\int_{B \cap \lbrace \vert \cdot-y \vert \le 1\rbrace}(y-\cdot)K_a^{(\alpha)}(\cdot,y) \d y\right)\nabla u \to b\nabla u  \text{ in } L^2(B) \,.$$ The proof of \eqref{eq:ptwhelp3} works in the same way using \eqref{eq:Moscohelp0}. Combining \eqref{eq:Moscohelp1} and \eqref{eq:Moscohelp2} yields
\begin{equation*}
\liminf_{n \to \infty} \cE^{(\alpha_n)}_B(u,v_n) \le \lim_{k \to \infty}\cE^{(\alpha_{n_k})}_B(u, v_{n_k}) = \cE(u,v), ~~ u \in C_c^{2}(\overline{B}).
\end{equation*}

Now, we prove (ii). Let us assume that $u_n \rightharpoonup u$, $\liminf_{n \to \infty} \cE^{(\alpha_n)}_B(u_n,u_n) < \infty$. By G\r{a}rding's inequality \eqref{eq:unifGarding} and uniform boundedness of $\Vert u_n \Vert_{L^2(B)}$, which is due to the weak convergence, 
\begin{align}
\label{eq:Moscohelp3}
\liminf_{n \to \infty} \cE^{K_s^{(\alpha_n)}}_B(u_n,u_n) < \infty.
\end{align}

Then according to Proposition 4.2 in \cite{Pon04} there exists a subsequence $(u_{n_k}) \subset (u_n)$ with $u_{n_k} \to u$ in $L^2(B)$ due to \eqref{eq:Moscohelp3}. Consequently it must be that already $u_n \to u$ in $L^2(B)$.\\
Note that by the weak sector condition and G\r{a}rding's inequality \eqref{eq:unifGarding}, $H^{\alpha_n/2}(B)$ with scalar product $\cE^{(\alpha_n),s}_{B,\lambda}(\cdot,\cdot) := \cE^{(\alpha_n),s}_B(\cdot,\cdot) + \lambda (\cdot,\cdot)_{L^2(B)}$ is a Hilbert space since the norm induced by $\cE^{(\alpha_n),s}_{B,\lambda}(\cdot,\cdot)$ is equivalent to the standard $H^{\alpha_n/2}(B)$-norm. From \autoref{lemma:ptwconv}, it follows that one can also prove a G\r{a}rding's inequality for $\cE$. Thus, we see that also $H^1(B)$ with scalar product $\cE^s_{\lambda}(\cdot,\cdot) := \cE^s(\cdot,\cdot) + \lambda (\cdot,\cdot)_{L^2(B)}$ is a Hilbert space. Due to \autoref{lemma:ptwconv}, we know that 
\begin{align}
\label{eq:Moscohelp4}
\cE^{(\alpha_n),s}_B(u,u) = \cE^{(\alpha_n)}_B(u,u) \to \cE(u,u) = \cE^{s}(u,u), ~~\text{for every } u \in H^1(B).
\end{align}
Furthermore, since $\liminf_{n \to \infty} \cE^{(\alpha_n)}_B(u_n,u_n) < \infty$, it is also clear that there exists a subsequence $(u_{n_k}) \subset (u_n)$ such that $\lim_{k \to \infty} \cE^{(\alpha_{n_k})}_B(u_{n_k},u_{n_k}) = \liminf_{n \to \infty} \cE^{(\alpha_n)}_B(u_n,u_n)$. Note that $(u_{n_k})$ is bounded with respect to $\cE^{(\alpha_{n_k}),s}_{B,\lambda}(\cdot,\cdot)$. 
Due to \eqref{eq:Moscohelp4} this implies according to Lemma 2.2 in \cite{KuSh03} that there exists a further subsequence $(u_{n_{k_l}}) \subset (u_{n_k})$ that we will simply also denote by $(u_{n_k})$, and $u' \in H^1(B)$ such that for every $v \in C_c^{\infty}(B)$:
\begin{equation*}
\lim_{k \to \infty} \cE^{(\alpha_{n_k}),s}_B(u_{n_k},v) + \lambda (u,v) = \lim_{k \to \infty} \cE^{(\alpha_{n_k}),s}_{B,\lambda}(u_{n_k},v) = \cE^{s}(u',v) + \lambda (u',v),
\end{equation*}
where we used that $u_n \to u$ in $L^2(B)$. Therefore, $u' = u$ and $u_{n_k} \rightharpoonup u$ with respect to $\cE_{B,\lambda}^{(\alpha_{n_k}),s}$, $\cE_{\lambda}^s$. By Lemma 2.3 in \cite{KuSh03}, and $u_n \to u$ in $L^2(B)$, we deduce 
\begin{align*}
 \cE(u,u) + \lambda \Vert u \Vert_{L^2(B)}^2 = \cE_{\lambda}(u,u)&\le \liminf_{k \to \infty} \cE^{(\alpha_n)}_{B,\lambda}(u_{n_k},u_{n_k})\\
 &= \liminf_{k \to \infty} \cE^{(\alpha_n)}_B(u_{n_k},u_{n_k}) + \lambda \Vert u \Vert_{L^2(B)}^2
\end{align*}
and from the definition of $(u_{n_k})$ we obtain that even
\begin{equation*}
\cE(u,u) \le \liminf_{k \to \infty} \cE^{(\alpha_{n_k})}_B(u_{n_k},u_{n_k}) = \liminf_{n \to \infty} \cE^{(\alpha_n)}_B(u_{n},u_{n}),
\end{equation*}
as desired.\\
Finally, we prove that $(\cE,H^1(B))$ is a regular lower bounded semi-Dirichlet form. Recall that due to \autoref{lemma:ptwconv}, it holds that $a(x)$ is uniformly positive definite and bounded, and $b_i \in L^d(B)$ $(\cE,H^1(B))$. In particular, $(\cE,H^1(B))$ satisfies the sector condition and a G\r{a}rding-type inequality as in \eqref{eq:unifGarding}. Consequently, $(\cE,H^1(B))$ is indeed a regular lower bounded semi-Dirichlet form.
\end{proof}

\begin{remark*}
Note that also in the case $B = \R^d$, it is true that $a_{i,j}$ is uniformly elliptic and bounded and $b_i \in L^{2\theta_0}(\R^d)$ under the condition that $\theta > d/2$ is independent of $\alpha$, where $a_{i,j}(x), b_i(x)$ are defined as in \eqref{eq:aij}, \eqref{eq:bi} for $x \in \R^d$.\\
If $\theta = \infty$, \autoref{lemma:ptwconv} and property (i) in the proof of \autoref{thm:Mosco} also hold true with $B = \R^d$.
However, property (ii) can only be verified for every strongly converging sequence $u_n \to u$ in $L^2(\R^d)$ because \eqref{eq:Moscohelp3} and boundedness of $\Vert u_n \Vert_{L^2(\R^d)}$ only yield the existence of a subsequence $(u_{n_k}) \subset (u_n)$ converging in $L^2_{loc}(\R^d)$. The combination of (i) and this weakened version of (ii) can be regarded as a nonsymmetric analog to Gamma-convergence in $L^2(\R^d)$.
\end{remark*}

\begin{remark*}
Mosco-Hino convergence to $(\cE,H^1(B))$ in $L^2(B)$ also holds true for the sequence $(\cE_{B,\R^d}^{(\alpha)},V^{\alpha}(B|\R^d))$, where
\begin{equation*}
\cE_{B,\R^d}^{(\alpha)}(u,v) = \int_B \int_{\R^d} (u(x)-u(y))v(x) K^{(\alpha)}(x,y) \d y \d x,
\end{equation*}
and $V^{\alpha}(B|\R^d) = \{ u : \R^d \to \R : v \mid_B \in L^2(B) : \cE_{B,\R^d}^{K_s^{(\alpha)}}(u,u) < \infty  \}$, equipped with $\Vert u \Vert^2_{V^{\alpha}(B|\R^d)} = \Vert u \Vert_{L^2(B)} + \cE_{B,\R^d}^{(\alpha)}(u,u)$ using the same arguments, as before. Note that the computation \eqref{eq:SCGIhelp} also works for $(\cE_{B,\R^d}^{(\alpha)},V^{\alpha}(B|\R^d))$, yielding the sector condition and G\r{a}rding's inequality. Moreover, it is easy to see that $\cE_{B,\R^d}^{(\alpha)}(u,v) \to \cE(u,v)$ as $\alpha \nearrow 2$ for every $u,v \in H^1(\R^d)$.
\end{remark*}

\section{Examples}
\label{sec:examples}

The goal of this section is to discuss several classes of examples and to investigate the validity of the key assumptions \eqref{K1} and \eqref{K2}. The examples include the following cases:
\begin{align*}
K(x,y) &= g(x,y)\vert x-y\vert^{-d-\alpha},\\
K(x,y) &= \vert x-y\vert^{-d-\alpha} + ( V(x) - V(y)) \vert x-y\vert^{-d-\alpha},\\
K(x,y) &= |x-y|^{-d-\alpha}\mathbbm{1}_{D}(x-y) + |x-y|^{-d-\beta}\mathbbm{1}_C(x-y),
\end{align*}
for appropriate functions $g : \R^d \times \R^d \to [\lambda,\Lambda]$, $V : \R^d \to \R$, sets $C, D \subset \R^d$, and $\alpha,\beta \in (0,2)$.
The corresponding classes are considered in \autoref{sec:fcase}, \autoref{sec:ex1} and \autoref{sec:cone}.
Note that \autoref{sec:ex1} generalizes \autoref{ex:intro} and connects this class of operators to second order differential operators in divergence form through \autoref{thm:Mosco}. In \autoref{sec:cone}, we discuss jumping kernels whose nonsymmetric part, as in the third example above, may live on certain cones centered at the origin, giving rise to an example which satisfies \eqref{K1}, \eqref{K2} with $j = J = |x-y|^{-d-\alpha}$, but not with $j = J = K_s$.

Before we discuss the influence of nonsymmetry, let us comment on the assumptions on the symmetric part $K_s$:

\begin{example}
\label{ex:symm}
A large class of symmetric kernels $K_s$ satisfying \eqref{cutoff}, \eqref{Poinc}, \eqref{Sob}, \eqref{cutoff2} are those for which \eqref{eq:suffcutoff} holds for every $\zeta > 0$ and there is $\Lambda \ge 1$ such that \eqref{eq:coercivity} holds true.
This does not only include kernels that are pointwise comparable to the $\alpha$-stable kernel but also anisotropic kernels that are not fully supported on $\R^d \times \R^d$ (see \cite{BKS19}) as
\begin{align*}
K_s(x,y) = \vert x-y\vert^{-d-\alpha} \left(\mathbbm{1}_{C(x)}(y) + \mathbbm{1}_{C(y)}(x) \right),
\end{align*}
where $\{C(x)\}_{x \in \R^d}$ are certain configurations of double cones $C(x) \subset \R^d$ centered at $x \in \R^d$.
\end{example}

As assumptions \eqref{cutoff}, \eqref{Poinc}, \eqref{Sob} only affect the symmetric part and are well-known in the literature, we focus on the effect of \eqref{K1}, \eqref{K2} on the class of admissible kernels. Trivially, given any symmetric $K_s$ satisfying \eqref{cutoff}, \eqref{Poinc}, \eqref{Sob}, \eqref{cutoff2} one can construct an admissible kernel $K$ by adding any antisymmetric $K_a$ as long as \eqref{K1}, \eqref{K2} and $\vert K_a(x,y)\vert \le \vert K_s(x,y)\vert$ hold true in order to obtain an admissible kernel satisfying all assumptions of \autoref{thm:mainthmPDE}, \autoref{thm:mainthmPDEdual}.

We close this section by giving a sufficient condition for \eqref{K1} in a general setup. Consider a symmetric kernel $J^{\alpha} : \R^d \times \R^d \to [0,\infty]$ and a function $g : \R^d \times \R^d \to [0,\infty]$. We define 
\begin{equation*}
K(x,y) = g(x,y)J^{\alpha}(x,y).
\end{equation*}
Then
\begin{align*}
K_a(x,y) = \frac{g(x,y)-g(y,x)}{2}J^{\alpha}(x,y), ~~K_s(x,y) = \frac{g(x,y)+g(y,x)}{2} J^{\alpha}(x,y).
\end{align*}

Our goal is to discuss assumptions on $g$ under which \eqref{K1} is  satisfied for $K$ with $J = K_s$.\\
\eqref{K1} reads as follows for $K$: There exists $C > 0$ such that for every ball $B_{2r} \subset \Omega$ with $r \le 1$:
\begin{align}
\label{eq:equivK1}
\left\Vert \int_{B_{2r}} (g(\cdot,y)-g(y,\cdot))^2 J^{\alpha}(\cdot,y) \d y  \right\Vert_{L^{\theta}(B_{2r})} \le C < \infty.
\end{align}

The subsequent lemma provides a suitable criterion for the verification of \eqref{K1}.
\begin{lemma}
\label{lemma:suffK1}
Let $V : \R^d \to \R$ and $\theta \in [\frac{d}{\alpha},\infty]$. Assume that $J^{\alpha}$ satisfies \eqref{eq:suffcutoff} and that there is a constant $c > 0$ such that for every ball $B_{2r} \subset \Omega$ with $r \le 1$ it holds either
\begin{align}
\label{eq:suffK1}
\exists \gamma \in (\alpha/2,1]:~~ \Vert[V(\cdot)]_{C^{0,\gamma}(B_{2r})}\Vert_{L^{2\theta}(B_{2r})} \le c,
\end{align}
where $[V(x)]_{C^{0,\gamma}(B_{2r})} := \inf \left\lbrace A > 0 ~|~ \forall y \in B_{2r}: \vert V(x)-V(y)\vert \le A \vert x-y\vert^{\gamma}\right\rbrace$, or
\begin{align}
\label{eq:suffK12}
\Vert \nabla V\Vert_{L^{2\theta}(B_{2r})} + \Vert \sup_{y \in B_{2r}} r(\cdot,y)\Vert_{L^{2\theta}(B_{2r}))} \le c,
\end{align}
where $r(x,y) = \left[V(x) - V(y) - (\nabla V(x) , x-y)\right]\vert x -y \vert^{-1}$.
Then
\begin{align*}
\left\Vert \int_{B_{2r}} (V(\cdot)-V(y))^2 J^{\alpha}(\cdot,y) \d y  \right\Vert_{L^{\theta}(B_{2r})} \le C < \infty.
\end{align*}
\end{lemma}

\begin{proof}
Let $B_{2r} \subset \Omega$.
First, we claim that for every $\eps > 0$ there exists $c_1 > 0$ independent of $B_{2r}$ such that
\begin{align}
\label{eq:suffK1help}
\left\Vert \int_{B_{2r}} \vert \cdot - y\vert^{\alpha+\eps} J^{\alpha}(\cdot,y) \d y  \right\Vert_{L^{\infty}(B_{2r})} \le c_1.
\end{align}
The claim follows directly by decomposing $ B_{4r}(x) = \bigcup_{k=0}^{\infty} B_{4^{1-k}r}(x) \setminus B_{4^{-k}r}(x)$ and applying \eqref{eq:suffcutoff} for every $x \in B_{2r}$ after enlarging the domain of integration for the inner integral to $B_{4r}(x) \supset B_{2r}$ for each $x \in B_{2r}$.\\
If \eqref{eq:suffK1} is satisfied, the assertion follows by H\"older's inequality and application of \eqref{eq:suffK1help} with $\eps := 2\gamma - \alpha > 0$. In case \eqref{eq:suffK12} holds, we observe that $(V(x) - V(y))^2 \le 2 \left(\vert\nabla V(x)\vert^2  + r^2(x,y)\right)\vert x-y\vert^2$ and proceed as before, applying \eqref{eq:suffK1help} with $\eps := 2 - \alpha > 0$.
\end{proof}

Let us discuss assumptions \eqref{eq:suffK1} and \eqref{eq:suffK12} from \autoref{lemma:suffK1} and give illustrating examples.

Case $\theta = \infty$: The following functions $V$ satisfy the assumptions of \autoref{lemma:suffK1} with $\theta = \infty$.
\vspace{-0.2cm}
\begin{itemize}
\item If $\gamma > \alpha/2$ then \eqref{eq:suffK1} holds with $(\theta =\infty, \gamma)$ for every $V \in C^{0,\gamma}(\Omega)$. As by Morrey's inequality $W^{s,p}(\Omega) \subset C^{0,\gamma}(\Omega)$ with $\gamma = s - \frac{d}{p}$, \eqref{eq:suffK1} is in particular satisfied for every $V \in W^{s,p}(\Omega)$ with $2s > \alpha$ and $p > \frac{2d}{2s-\alpha}$ if $\Omega$ is smooth.
\item \eqref{eq:suffK12} is satisfied for every $V \in C^1(\Omega)$ with $\Vert \nabla V \Vert_{L^{\infty}(\Omega)} < \infty$ by mean value theorem.
\end{itemize}

Case $\theta < \infty$: It is possible to consider also less smooth functions $V$. This shows the advantage of formulating assumption \eqref{K1} with general $\theta \in [\frac{d}{\alpha},\infty]$.
\vspace{-0.2cm}
\begin{itemize}
\item Let $\Omega = B_2(0)$. Define $V(x) = \vert x \vert^{\gamma_0}$ for some $ 0 < \gamma_0 < \alpha/2$. Then $V \in C^{0,\gamma_0}(B_2(0))$ but \eqref{eq:suffK1} fails for $\theta = \infty$. However, given $x \in B_2(0)$ and $\gamma := \gamma_0 + \eps \le 1$ for some $\eps > 0$: 
\begin{align*}
\sup_{y \in B_2(0)} \frac{\vert V(x)-V(y)\vert}{\vert x-y\vert^{\gamma}} = \vert x \vert^{\gamma_0 - \gamma},
\end{align*}
so $[V(\cdot)]_{C^{0,\gamma}(B_2(0))} = \vert \cdot\vert^{-\eps} \in L^{2\theta}(B_2(0))$ if $\eps < \frac{d}{2\theta}$. Note that $\gamma > \frac{\alpha}{2}$ if $\eps > \frac{\alpha}{2}-\gamma_0$. Thus, \eqref{eq:suffK1} is satisfied for any pair $(\theta,\gamma)$ such that $\theta \in [\frac{d}{\alpha} , \frac{d}{\alpha - 2\gamma_0})$ and $\eps \in (\frac{\alpha}{2}-\gamma_0 , \frac{d}{2\theta})$.
\item \eqref{eq:suffK12} holds with $\theta \in [\frac{d}{\alpha},\infty)$ if $V \in W^{1,2\theta}(\R^d)$. A proof is given in \cite{Spe16}.
\end{itemize}

\subsection{nonsymmetric coefficients}
\label{sec:fcase}

In this section we consider jumping kernels $K$ that are driven by a symmetric kernel $J^{\alpha} : \R^d \times \R^d \to [0,\infty]$ satisfying \eqref{eq:suffcutoff} for every $\zeta > 0$, \eqref{Poinc} and \eqref{Sob} for some $\alpha \in (0,2)$ and a nonsymmetric coefficient function $g : \R^d \times \R^d \to [\lambda,\Lambda]$, where $0 < \lambda \le \Lambda < \infty$. We define
\begin{align}
\label{eq:kernelclassf}
K(x,y) = g(x,y)J^{\alpha}(x,y).
\end{align}
$J^{\alpha}$ can be any kernel from \autoref{ex:symm}, the prototype kernel being $J^{\alpha}(x,y) = \vert x-y\vert^{-d-\alpha}$. 
First, we observe that by the boundedness of $g$, $J^{\alpha}$ inherits the properties \eqref{cutoff}, \eqref{Poinc}, \eqref{Sob}, \eqref{cutoff2} to $K$.
Moreover, \eqref{K2} is satisfied even globally:

\begin{proposition}
\label{prop:K2}
Let $D = \frac{\Lambda - \lambda}{\Lambda + \lambda} < 1$. Then \eqref{K2} is satisfied. In particular,
\begin{align*}
\vert K_a(x,y) \vert \le DK_s(x,y),~~ \forall x,y \in \R^d.
\end{align*}
\end{proposition}

\begin{proof}
By definition of $D$, we have $\frac{\lambda}{\Lambda} = \frac{1-D}{D+1}$.
Let $x,y \in \R^d$ with $g(x,y) \ge g(y,x)$. Then:
\begin{align*}
\frac{g(y,x)}{g(x,y)} \ge \frac{\lambda}{\Lambda} = \frac{1-D}{D+1},
\end{align*}
which implies that $\vert g(x,y) - g(y,x)\vert = g(x,y) - g(y,x) \le D(g(x,y)+g(y,x))$.
\end{proof}

By application of \autoref{lemma:suffK1} it is easy to verify \eqref{eq:equivK1} and therefore \eqref{K1} for certain classes of coefficient functions $g$.

\begin{example}
\label{ex:Vs}
Let $K(x,y) = g(x,y)J^{\alpha}(x,y)$, where $J^{\alpha}$, $g$ are as before. Assume that $V_1,V_2 : \R^d \to \R$ satisfy the assumptions of \autoref{lemma:suffK1} for  $\theta \in [\frac{d}{\alpha},\infty]$. Then \eqref{K1} (with $\theta$) holds for $K$ if one of the following is true:
\vspace{-0.2cm}
\begin{itemize}
\item[(i)] $g(x,y) := V_1(x) + V_2(y) \in [\lambda,\Lambda]$.
\item[(ii)]  $g(x,y) := V_1(x)V_2(y) \in [\lambda,\Lambda]$ and $\Vert V_1 \Vert_{L^{\infty}(\Omega)} + \Vert V_2 \Vert_{L^{\infty}(\Omega)} < \infty$.
\end{itemize}
\vspace{-0.2cm}
If (i) or (ii) are satisfied, then \autoref{thm:mainthmPDE} is applicable to $K$. 
\end{example}

\begin{remark*}
Note that one can carry out all the arguments from this section also if $g : \R^d \times \R^d \to [0,\Lambda]$ with $g\mid_{\Omega \times \Omega} \in [\lambda,\Lambda]$ without any significant changes.
\end{remark*}

\subsection{Carr\'e du champ-type nonlocal drift}
\label{sec:ex1}

In this section, we discuss and generalize \autoref{ex:intro} from the introduction. Let $\alpha \in (0,2)$, $L \in (0,\infty]$, $0 < \lambda \le \Lambda < \infty$ and $j : \R^d \times \R^d \to [\lambda,\Lambda]$ be a symmetric function. Let $V : \R^d \to \R$ be such that 
\begin{align}
\label{eq:Knonneg}
\vert V(x) - V(y) \vert \mathbbm{1}_{\lbrace \vert x-y\vert \le L \rbrace}(x,y) \le \lambda, ~~ \forall x,y \in \R^d.
\end{align}
We define
\begin{align}
\label{eq:kernelclassV}
K(x,y) = j(x,y)c_{d,\alpha}\vert x-y\vert^{-d-\alpha} + ( V(x) - V(y)) \mathbbm{1}_{\lbrace \vert x-y\vert \le L \rbrace}(x,y)c_{d,\alpha}\vert x-y\vert^{-d-\alpha},
\end{align}
where $c_{d,\alpha} := \frac{2^{\alpha}\Gamma\left(\frac{d+\alpha}{2}\right)}{\pi^{d/2}\vert\Gamma\left(-\frac{\alpha}{2}\right)\vert} > 0$. This class of kernels is without further assumptions not contained in the class considered in \autoref{sec:fcase} since the coefficient function $g$ was supposed to be bounded between two positive constants.

From the definition it is already clear that $K \ge 0$ and
\begin{align*}
K_s(x,y) = j(x,y)c_{d,\alpha}\vert x-y\vert^{-d-\alpha},~~ K_a(x,y) = ( V(x) - V(y)) \mathbbm{1}_{\lbrace \vert x-y\vert \le L \rbrace}(x,y)c_{d,\alpha}\vert x-y\vert^{-d-\alpha}.
\end{align*}
Moreover, \eqref{cutoff}, \eqref{cutoff2}, \eqref{Poinc}, \eqref{Sob} trivially hold by boundedness of $j$. \eqref{K1} holds for $K$ with $\theta \in [\frac{d}{\alpha},\infty]$ if $V$ satisfies the condition of \autoref{lemma:suffK1}.

Note that if \eqref{K1} holds true with $\theta = \infty$, \eqref{K2} is satisfied in the sense of \autoref{prop:K1impliesK2}.
The following is another sufficient condition for \eqref{K2}: \\
Assume that there exists $D < 1$ such that for every ball $B_{2r} \subset \Omega$, $0 < r \le 1$:
\begin{align}
\label{eq:suffK2}
\vert V(x) - V(y)\vert\mathbbm{1}_{\lbrace \vert x-y\vert \le L \rbrace}(x,y) \le D\lambda, ~~ \forall x,y \in B_{2r}.
\end{align}
Under these additional assumptions on $V,L$, \autoref{thm:mainthmPDE} is applicable to $K$.

We observe that for a given H\"older continuous function $V$ it is possible to choose $L$ suitably such that \eqref{eq:Knonneg}, \eqref{eq:suffK2} hold true for $K$:
\begin{proposition}
\label{ex:Lcond}
Assume that $V \in C^{0,\gamma}(\R^d)$ for some $\gamma \in (0,1]$.  If $L \le \left(\lambda [V]_{C^{0,\gamma}(\R^d)}^{-1}\right)^{1/\gamma}$, then \eqref{eq:Knonneg} holds true. If  $L \le \left(D\lambda [V]_{C^{0,\gamma}(\R^d)}^{-1}\right)^{1/\gamma}$ for some $D < 1$, then \eqref{K2} holds for $K$.
\end{proposition}

\begin{remark*}
The class of kernels defined in \eqref{eq:kernelclassV} can be generalized as follows:
\vspace{-0.2cm}
\begin{itemize}
\item[(i)] Consider symmetric kernels $K_{1,2}(x,y) \asymp \vert x-y \vert^{-d-\alpha}$. Let $V : \R^d \to \R$, $L > 0$. Define
\begin{align*}
K(x,y) := K_1(x,y) + (V(x)-V(y))\mathbbm{1}_{\lbrace \vert x-y \vert \le L\rbrace}(x,y)K_2(x,y).
\end{align*}
\item[(ii)] Let $K_1$ be as before, $0 < \beta^{(i)} < \alpha/2$ and $K_2^{(i)}(x,y) \asymp \vert x-y\vert^{-d-\beta^{(i)}}$ symmetric, $V^{(i)} : \R \to \R$ and $L^{(i)} > 0$, for $i \in \lbrace 1,\dots,d\rbrace$. Define
\begin{align*}
K(x,y) := K_1(x,y) + \sum_{i=1}^d(V^{(i)}(x_i)-V^{(i)}(y_i))\mathbbm{1}_{\lbrace \vert x_i-y_i \vert \le L^{(i)}\rbrace}(x,y)K_2^{(i)}(x,y).
\end{align*}
\end{itemize}
\vspace{-0.2cm}
Under suitable assumptions on $V,L$, respectively $V^{(i)},L$ one can establish \eqref{K1}, \eqref{K2}.
\end{remark*}

\textbf{Convergence to a diffusion with drift.~}
The class of kernels defined in \eqref{eq:kernelclassV} is of fundamental importance to us since the corresponding operators can be regarded as nonlocal counterparts of second order divergence form operators with a drift term. We define
\begin{align*}
\Gamma^{(\alpha)}_L(u,V)(x) = \int_{\R^d}(u(x)-u(y))(V(x)-V(y))\mathbbm{1}_{\lbrace \vert x-y\vert \le L \rbrace}(x,y)c_{d,\alpha}\vert x-y\vert^{-d-\alpha} \d y
\end{align*}
and observe that $\cEa(u,v) = \int_{\R^d} \Gamma_L(u,V)(x)v(x) \d x$.\\
Therefore $K$ gives rise to the operator $L_{K_s} + \Gamma_L(\cdot,V)$, where $L_{K_s}$ is a symmetric diffusion operator comparable to $(-\Delta)^{\alpha/2}$. We interpret $\Gamma_L(\cdot,V)$ as a nonlocal drift whose direction depends on $V$. A very simple example of $V$ is given by $V(x) = \sum_{i=1}^d b_i x_i$, where $b_i \in \R$.\\
We give a justification of this viewpoint through the approximation results from \autoref{sec:approx}. Given a suitable function $V$ the following theorem is a corollary of \autoref{thm:Mosco} and states that 
\begin{align*}
(-\Delta)^{\alpha/2} + \Gamma^{(\alpha)}_L(\cdot,V) \to -\Delta + (\nabla\cdot, \nabla V).
\end{align*}

\begin{theorem}
\label{thm:approxV}
Let $\alpha_0 \in (0,2)$ and $\theta \in [\frac{d}{\alpha_0},\infty]$. Let $j : \R^d \times \R^d \to [\lambda,\Lambda]$ be symmetric and $V \in W^{1,2\theta}(\R^d) \cap C^1(\R^d)$ and $L \in (0,\infty)$ such that \eqref{eq:Knonneg} holds true. Define
\begin{align*}
K^{(\alpha)}(x,y) = j(x,y) c_{d,\alpha}\vert x-y\vert^{-d-\alpha} + ( V(x) - V(y)) \mathbbm{1}_{\lbrace \vert x-y\vert \le L \rbrace}(x,y)c_{d,\alpha}\vert x-y\vert^{-d-\alpha}.
\end{align*}
Then \eqref{eq:globalK1} and \eqref{eq:kernelcomp} hold true for $(K^{(\alpha)})_{\alpha \in (\alpha_0,2)}$. Let $B \subset \R^d$ be a smooth bounded domain. Then the sequence of forms $(\cE^{(\alpha)}_B,H^{\alpha/2}(B))$, with
\begin{align*}
\cE^{(\alpha)}_B(u,v) = \int_{B}\int_{B}(u(x)-u(y))v(x) K^{(\alpha)}(x,y) \d y \d x
\end{align*} 
converges in the Mosco-Hino-sense in $L^2(B)$ to $(\cE,H^1(B))$, given by
\begin{align*}
\cE(u,v) = \int_{B} a_{i,j}(x) \partial_i u(x) \partial_j u(x) \d x + 2\int_{B} \partial_i V(x) \partial_i u(x) v(x) \d x,
\end{align*}
where $a_{i,j}$ is defined as in \eqref{eq:aij}.
\end{theorem} 

\begin{proof}
First, we observe that for $\theta < \infty $ it holds (see Theorem 1.3 in \cite{Spe16})
\begin{align}
\label{eq:approxVhelp1}
\lim_{\delta \to 0} \int_{\R^d} \left(\sup_{y : \vert x-y \vert \le \delta} \frac{\vert V(x) - V(y) - (\nabla V(x) , x - y)\vert}{\vert x - y\vert}\right)^{2\theta} \d x = 0.
\end{align}
Using the following identity
\begin{align}
\label{eq:approxVhelp2}
V(x)-V(y) = (\nabla V(x) , x-y) + (V(x)-V(y) - (\nabla V(x) , x-y)),
\end{align}
we deduce that \eqref{eq:globalK1} holds true with $\theta$ by the same arguments as in the proof of \autoref{lemma:suffK1}. In case $\theta = \infty$, the proof of \eqref{eq:globalK1} is immediate. Therefore, \autoref{thm:Mosco} is applicable to $(\cE^{(\alpha)}_B,H^{\alpha/2}(B))$ and it only remains to prove that $b = \nabla V$, where $b$ is defined via
\begin{align*}
b_i(x) = \lim_{\alpha \nearrow 2} \int_{\{\vert x-y \vert \le \delta \wedge L\}} (x_i-y_i) (V(x)-V(y)) c_{d,\alpha}\vert x-y\vert^{-d-\alpha} \d y.
\end{align*}
First, we observe that since $\int_{\{\vert h \vert \le \delta\}} h_ih_j c_{d,\alpha}\vert h\vert^{-d-\alpha} \d h \to \delta_{i,j}$ as $\alpha \nearrow 2$ for every $\delta > 0$, it holds for $x \in \R^d$:
\begin{align*}
\lim_{\alpha \nearrow 2} \int_{\{\vert x-y \vert \le \delta \wedge L\}} (x_i-y_i) (\nabla V(x) , x-y) c_{d,\alpha}\vert x-y\vert^{-d-\alpha} \d y = \partial_i V(x).
\end{align*}
Moreover for $x\in\R^d$:
\begin{align*}
\lim_{\alpha \nearrow 2}& \int_{\{\vert x-y \vert \le \delta \wedge L\}} (x_i-y_i)[V(x)-V(y) - (\nabla V(x) , x-y)] c_{d,\alpha}\vert x-y\vert^{-d-\alpha} \d y \\
&\le \lim_{\alpha \nearrow 2} \int_{\{\vert x-y \vert \le \delta \wedge L\}} \frac{\vert V(x)-V(y) - (\nabla V(x) , x-y)\vert}{\vert x-y\vert} c_{d,\alpha}\vert x-y\vert^{2-d-\alpha} \d y \\
&\le c \sup_{y : \vert x-y \vert \le \delta \wedge L} \frac{\vert V(x)-V(y) - (\nabla V(x) , x-y)\vert}{\vert x-y\vert}
\end{align*}
for some $c > 0$ independent of $\delta,\alpha$. Due to $V \in C^1(\R^d)$ we can conclude from Taylor's formula that for every $x \in \R^d$ the above quantity becomes arbitrarily small by choosing $\delta > 0$ small enough. Therefore, in the light of \eqref{eq:approxVhelp2}, it holds that for every $x \in \R^d$:
\begin{align*}
\int_{\{\vert x-y \vert \le \delta \wedge L\}} (x_i-y_i) (V(x)-V(y)) c_{d,\alpha}\vert x-y\vert^{-d-\alpha} \d y \to \partial_i V(x) ~~ \text{ as } \alpha \nearrow 2,
\end{align*}
so in particular, $b = \nabla V$, as desired.
\end{proof}

\begin{remark*}
By choosing $\alpha_0$ suitably it is possible to apply \autoref{thm:approxV} to any function $V \in W^{1,d+\eps}(\R^d) \cap C^1(\R^d)$, $\eps > 0$, if $L$ is such that \eqref{eq:Knonneg} holds true.
\end{remark*}

\subsection{Nonlocal drifts on cones}
\label{sec:cone}

There exist also admissible nonsymmetric jumping measures that are not induced by coefficient functions $g$ or drifts $V$ as in the previous two sections:

\begin{example}
Let $K_s(x,y) \asymp \vert x-y \vert^{-d-\alpha}$ be given. Let $g,h: \R^d \times \R^d \to \R$ be antisymmetric. Let $L > 0$ and $\beta < \alpha/2$. We set
\begin{align*}
K_a(x,y) = g(x,y)\vert x-y\vert^{-d-\beta}\mathbbm{1}_{\lbrace \vert x-y\vert \le L \rbrace}(x,y) + h(x,y) \vert x-y\vert^{-d-\alpha}\mathbbm{1}_{\lbrace \vert x-y\vert > L \rbrace}(x,y).
\end{align*}
Under suitable assumptions on $g,h,L$ one can verify \eqref{K1}, \eqref{K2} so that $K = K_s + K_a$ is an admissible kernel in the sense of \autoref{thm:mainthmPDE}, \autoref{thm:mainthmPDEdual}.\\
A particular instance of this class is the following anisotropic kernel
\begin{align*}
K(x,y) = \vert x-y\vert^{-d-\alpha} + \mathbbm{1}_C (x-y)\vert x-y\vert^{-d-\beta},
\end{align*}
where $0 < \beta < \alpha/2 < 1$ and $C$ is a single cone centered at the origin. (e.g. $C = \R^d_+$).
It is easy to see that \eqref{K1}, \eqref{K2} are satisfied in this case.
\end{example}

All examples mentioned so far have in common that they satisfy \eqref{K1} and \eqref{K2} with $J = K_s$, respectively $j = K_s$. We would now like to give an example where \eqref{K1}, \eqref{K2} hold true with symmetric kernels $j,J$ that are not pointwise comparable to $K_s$, see also \cite{FKV15}.

\begin{proposition}
\label{prop:nldriftoncones}
Let $C \subset \R^d$ be a single cone and $D \subset \R^d$ be a double cone such that $C \cap D = \emptyset$. Let $0 < 2\beta < \alpha < 2$. Consider
\begin{align*}
K(x,y) = |x-y|^{-d-\alpha}\mathbbm{1}_{D}(x-y) + |x-y|^{-d-\beta}\mathbbm{1}_C(x-y).
\end{align*}
Then \eqref{K1}, \eqref{K2} hold true for $K$ with $J(x,y) = |x-y|^{-d-\alpha}$, $j(x,y) = |x-y|^{-d-\alpha}\mathbbm{1}_{D}(x-y)$ with $\theta = \infty$. Moreover, \eqref{Sob}, \eqref{Poinc}, \eqref{cutoff} hold true with $\alpha$ and \eqref{cutoff2} holds true.
\end{proposition}

\begin{proof}
Note that
\begin{align*}
K_s(x,y) &= |x-y|^{-d-\alpha}\mathbbm{1}_{D}(x-y) + \frac{1}{2}|x-y|^{-d-\beta}\left(\mathbbm{1}_{C}(x-y) + \mathbbm{1}_{C}(y-x)\right),\\
K_a(x,y) &= \frac{1}{2}|x-y|^{-d-\beta}\left(\mathbbm{1}_{C}(x-y) - \mathbbm{1}_{C}(y-x)\right).
\end{align*}
\eqref{K1} follows from the computation
\begin{align*}
\int_{B_{2}} \frac{|K_a(x,y)|^2}{J(x,y)} \d y \le c \int_{B_{2}(x)} |h|^{-d-2\beta + \alpha} < \infty,
\end{align*}
using that $\beta < \alpha/2$. To see \eqref{K2}, note that trivially, $K(x,y) \ge j(x,y)$. The comparability of the corresponding localized energy forms
\begin{align*}
\cE^{J}_{B_{r+\rho}}(v,v) \asymp \cE^{K_s}_{B_{r+\rho}}(v,v) \asymp \cE^{j}_{B_{r+\rho}}(v,v)
\end{align*}
is well-known, given $0 < \rho \le r \le 1$ and $v \in L^2(B_{r+\rho})$. Note that  comparability of the latter two forms requires $\beta \le \alpha$. The remaining assertions are easy to check.
\end{proof}


\begin{thebibliography}{DROSV20}
	
	\bibitem[AS67]{ArSe67}
	D.~G. Aronson and James Serrin.
	\newblock Local behavior of solutions of quasilinear parabolic equations.
	\newblock {\em Arch. Rational Mech. Anal.}, 25:81--122, 1967.
	
	\bibitem[BKS19]{BKS19}
	Kai-Uwe Bux, Moritz Kassmann, and Tim Schulze.
	\newblock Quadratic forms and {S}obolev spaces of fractional order.
	\newblock {\em Proc. Lond. Math. Soc. (3)}, 119(3):841--866, 2019.
	
	\bibitem[BL02]{BaLe02b}
	Richard~F. Bass and David~A. Levin.
	\newblock Harnack inequalities for jump processes.
	\newblock {\em Potential Anal.}, 17(4):375--388, 2002.
	
	\bibitem[BOS21]{BOS21}
	Sun-Sig Byun, Jihoon Ok, and Kyeong Song.
	\newblock {H}{\"o}lder regularity for weak solutions to nonlocal double phase
	problems.
	\newblock {\em arXiv:2108.09623}, 2021.
	
	\bibitem[CCV11]{CCV11}
	Luis Caffarelli, Chi~Hin Chan, and Alexis Vasseur.
	\newblock Regularity theory for parabolic nonlinear integral operators.
	\newblock {\em J. Amer. Math. Soc.}, 24(3):849--869, 2011.
	
	\bibitem[CK03]{ChKu03}
	Zhen-Qing Chen and Takashi Kumagai.
	\newblock Heat kernel estimates for stable-like processes on {$d$}-sets.
	\newblock {\em Stochastic Process. Appl.}, 108(1):27--62, 2003.
	
	\bibitem[CK20]{ChKa20}
	Jamil Chaker and Moritz Kassmann.
	\newblock Nonlocal operators with singular anisotropic kernels.
	\newblock {\em Comm. Partial Differential Equations}, 45(1):1--31, 2020.
	
	\bibitem[CK21a]{CK21b}
	Jamil Chaker and Minhyun Kim.
	\newblock Local regularity for nonlocal equations with variable exponents.
	\newblock {\em arXiv:2107.06043}, 2021.
	
	\bibitem[CK21b]{ChKi21}
	Jamil Chaker and Minhyun Kim.
	\newblock Regularity estimates for fractional orthotropic $ p $-{L}aplacians of
	mixed order.
	\newblock {\em arXiv:2104.07507}, 2021.
	
	\bibitem[CKW19]{CKW19}
	Jamil Chaker, Moritz Kassmann, and Marvin Weidner.
	\newblock Robust {H}{\"o}lder estimates for parabolic nonlocal operators.
	\newblock arXiv:1912.09919, 2019.
	
	\bibitem[CKW20]{CKW20}
	Zhen-Qing Chen, Takashi Kumagai, and Jian Wang.
	\newblock Stability of parabolic {H}arnack inequalities for symmetric non-local
	{D}irichlet forms.
	\newblock {\em J. Eur. Math. Soc. (JEMS)}, 22(11):3747--3803, 2020.
	
	\bibitem[CKW21]{CKW21}
	Jamil Chaker, Minhyun Kim, and Marvin Weidner.
	\newblock Regularity for nonlocal problems with non-standard growth.
	\newblock {\em arXiv:2111.09182}, 2021.
	
	\bibitem[Coz17]{Coz17}
	Matteo Cozzi.
	\newblock Regularity results and {H}arnack inequalities for minimizers and
	solutions of nonlocal problems: a unified approach via fractional {D}e
	{G}iorgi classes.
	\newblock {\em J. Funct. Anal.}, 272(11):4762--4837, 2017.
	
	\bibitem[CS09]{CaSi09}
	Luis Caffarelli and Luis Silvestre.
	\newblock Regularity theory for fully nonlinear integro-differential equations.
	\newblock {\em Comm. Pure Appl. Math.}, 62(5):597--638, 2009.
	
	\bibitem[CS20]{ChSi20}
	Jamil Chaker and Luis Silvestre.
	\newblock Coercivity estimates for integro-differential operators.
	\newblock {\em Calc. Var. Partial Differential Equations}, 59(4):Paper No. 106,
	20, 2020.
	
	\bibitem[D{\'a}v20]{Dav20}
	Gonzalo D{\'a}vila.
	\newblock Comparison principles for nonlocal {H}amilton-{J}acobi equations.
	\newblock {\em arXiv:2011.13312}, 2020.
	
	\bibitem[DCKP14]{DKP14}
	Agnese Di~Castro, Tuomo Kuusi, and Giampiero Palatucci.
	\newblock Nonlocal {H}arnack inequalities.
	\newblock {\em J. Funct. Anal.}, 267(6):1807--1836, 2014.
	
	\bibitem[DCKP16]{DKP16}
	Agnese Di~Castro, Tuomo Kuusi, and Giampiero Palatucci.
	\newblock Local behavior of fractional {$p$}-minimizers.
	\newblock {\em Ann. Inst. H. Poincar\'{e} Anal. Non Lin\'{e}aire},
	33(5):1279--1299, 2016.
	
	\bibitem[DK13]{DyKa13}
	Bart{\l}omiej Dyda and Moritz Kassmann.
	\newblock On weighted {P}oincar\'{e} inequalities.
	\newblock {\em Ann. Acad. Sci. Fenn. Math.}, 38(2):721--726, 2013.
	
	\bibitem[DK20]{DyKa20}
	Bart{\l}omiej Dyda and Moritz Kassmann.
	\newblock Regularity estimates for elliptic nonlocal operators.
	\newblock {\em Anal. PDE}, 13(2):317--370, 2020.
	
	\bibitem[DROSV20]{DRSV20}
	Serena Dipierro, Xavier Ros-Oton, Joaquim Serra, and Enrico Valdinoci.
	\newblock Non-symmetric stable operators: regularity theory and integration by
	parts.
	\newblock {\em arXiv:2012.04833}, 2020.
	
	\bibitem[DT20]{DaTo20}
	Gonzalo D{\'a}vila and Erwin Topp.
	\newblock The nonlocal inverse problem of {D}onsker and {V}aradhan.
	\newblock {\em arXiv:2011.13295}, 2020.
	
	\bibitem[DZZ21]{DZZ21}
	Mengyao Ding, Chao Zhang, and Shulin Zhou.
	\newblock Local boundedness and {H}{\"o}lder continuity for the parabolic
	fractional p-{L}aplace equations.
	\newblock {\em Calculus of Variations and Partial Differential Equations},
	60(1):1--45, 2021.
	
	\bibitem[FG20]{Gou20}
	Guy~Fabrice Foghem~Gounoue.
	\newblock L2-theory for nonlocal operators on domains.
	\newblock {\em PhD Thesis, Bielefeld University}, 2020.
	
	\bibitem[FGKV20]{GKV20}
	Guy~Fabrice Foghem~Gounoue, Moritz Kassmann, and Paul Voigt.
	\newblock Mosco convergence of nonlocal to local quadratic forms.
	\newblock {\em Nonlinear Anal.}, 193:111504, 22, 2020.
	
	\bibitem[FK13]{FeKa13}
	Matthieu Felsinger and Moritz Kassmann.
	\newblock Local regularity for parabolic nonlocal operators.
	\newblock {\em Comm. Partial Differential Equations}, 38(9):1539--1573, 2013.
	
	\bibitem[FKV15]{FKV15}
	Matthieu Felsinger, Moritz Kassmann, and Paul Voigt.
	\newblock The {D}irichlet problem for nonlocal operators.
	\newblock {\em Math. Z.}, 279(3-4):779--809, 2015.
	
	\bibitem[FU12]{FuUe12}
	Masatoshi Fukushima and Toshihiro Uemura.
	\newblock Jump-type {H}unt processes generated by lower bounded
	semi-{D}irichlet forms.
	\newblock {\em Ann. Probab.}, 40(2):858--889, 2012.
	
	\bibitem[GT01]{GiTr01}
	David Gilbarg and Neil~S. Trudinger.
	\newblock {\em Elliptic partial differential equations of second order}.
	\newblock Classics in Mathematics. Springer-Verlag, Berlin, 2001.
	\newblock Reprint of the 1998 edition.
	
	\bibitem[Hin98]{Hin98}
	Masanori Hino.
	\newblock Convergence of non-symmetric forms.
	\newblock {\em J. Math. Kyoto Univ.}, 38(2):329--341, 1998.
	
	\bibitem[IS20]{ImSi20}
	Cyril Imbert and Luis Silvestre.
	\newblock The weak {H}arnack inequality for the {B}oltzmann equation without
	cut-off.
	\newblock {\em J. Eur. Math. Soc. (JEMS)}, 22(2):507--592, 2020.
	
	\bibitem[Kas09]{Kas09}
	Moritz Kassmann.
	\newblock A priori estimates for integro-differential operators with measurable
	kernels.
	\newblock {\em Calc. Var. Partial Differential Equations}, 34(1):1--21, 2009.
	
	\bibitem[Kim20]{Kim20}
	Yong-Cheol Kim.
	\newblock Local properties for weak solutions of nonlocal heat equations.
	\newblock {\em Nonlinear Anal.}, 192:111689, 30, 2020.
	
	\bibitem[KL13]{KiLe13b}
	Yong-Cheol Kim and Ki-Ahm Lee.
	\newblock Regularity results for fully nonlinear integro-differential operators
	with nonsymmetric positive kernels: subcritical case.
	\newblock {\em Potential Anal.}, 38(2):433--455, 2013.
	
	\bibitem[KS03]{KuSh03}
	Kazuhiro Kuwae and Takashi Shioya.
	\newblock Convergence of spectral structures: a functional analytic theory and
	its applications to spectral geometry.
	\newblock {\em Comm. Anal. Geom.}, 11(4):599--673, 2003.
	
	\bibitem[KS14]{KaSc14}
	Moritz Kassmann and Russell~W. Schwab.
	\newblock Regularity results for nonlocal parabolic equations.
	\newblock {\em Riv. Math. Univ. Parma (N.S.)}, 5(1):183--212, 2014.
	
	\bibitem[LD12]{ChD12}
	H\'{e}ctor~Chang Lara and Gonzalo D\'{a}vila.
	\newblock Regularity for solutions of nonlocal, nonsymmetric equations.
	\newblock {\em Ann. Inst. H. Poincar\'{e} Anal. Non Lin\'{e}aire},
	29(6):833--859, 2012.
	
	\bibitem[LD14]{ChD14}
	H\'{e}ctor~Chang Lara and Gonzalo D\'{a}vila.
	\newblock Regularity for solutions of non local parabolic equations.
	\newblock {\em Calc. Var. Partial Differential Equations}, 49(1-2):139--172,
	2014.
	
	\bibitem[LD16]{ChDa16}
	H\'{e}ctor~Chang Lara and Gonzalo D\'{a}vila.
	\newblock H\"{o}lder estimates for non-local parabolic equations with critical
	drift.
	\newblock {\em J. Differential Equations}, 260(5):4237--4284, 2016.
	
	\bibitem[LSU68]{LSU68}
	O.~A. Lady\v{z}enskaja, V.~A. Solonnikov, and N.~N. Ural{'}ceva.
	\newblock {\em Linear and quasilinear equations of parabolic type}.
	\newblock Translations of Mathematical Monographs, Vol. 23. American
	Mathematical Society, Providence, R.I., 1968.
	\newblock Translated from the Russian by S. Smith.
	
	\bibitem[Mos64]{Mos64}
	J\"{u}rgen Moser.
	\newblock A {H}arnack inequality for parabolic differential equations.
	\newblock {\em Comm. Pure Appl. Math.}, 17:101--134, 1964.
	
	\bibitem[Mos71]{Mos71}
	J.~Moser.
	\newblock On a pointwise estimate for parabolic differential equations.
	\newblock {\em Comm. Pure Appl. Math.}, 24:727--740, 1971.
	
	\bibitem[Mos94]{Mos94}
	Umberto Mosco.
	\newblock Composite media and asymptotic {D}irichlet forms.
	\newblock {\em J. Funct. Anal.}, 123(2):368--421, 1994.
	
	\bibitem[MR92]{MaRo92}
	Zhi~Ming Ma and Michael R\"{o}ckner.
	\newblock {\em Introduction to the theory of (nonsymmetric) {D}irichlet forms}.
	\newblock Universitext. Springer-Verlag, Berlin, 1992.
	
	\bibitem[Ok21]{Ok21}
	Jihoon Ok.
	\newblock Local {H}{\"o}lder regularity for nonlocal equations with variable
	powers.
	\newblock {\em arXiv:2107.06611}, 2021.
	
	\bibitem[Osh13]{Osh13}
	Yoichi Oshima.
	\newblock {\em Semi-{D}irichlet forms and {M}arkov processes}, volume~48 of
	{\em De Gruyter Studies in Mathematics}.
	\newblock Walter de Gruyter \& Co., Berlin, 2013.
	
	\bibitem[Pon04]{Pon04}
	Augusto~C. Ponce.
	\newblock An estimate in the spirit of {P}oincar\'{e}'s inequality.
	\newblock {\em J. Eur. Math. Soc. (JEMS)}, 6(1):1--15, 2004.
	
	\bibitem[SC02]{Sal02}
	Laurent Saloff-Coste.
	\newblock {\em Aspects of {S}obolev-type inequalities}, volume 289 of {\em
		London Mathematical Society Lecture Note Series}.
	\newblock Cambridge University Press, Cambridge, 2002.
	
	\bibitem[Sil06]{Sil06}
	Luis Silvestre.
	\newblock H\"{o}lder estimates for solutions of integro-differential equations
	like the fractional {L}aplace.
	\newblock {\em Indiana Univ. Math. J.}, 55(3):1155--1174, 2006.
	
	\bibitem[Sil11]{Sil11}
	Luis Silvestre.
	\newblock On the differentiability of the solution to the {H}amilton-{J}acobi
	equation with critical fractional diffusion.
	\newblock {\em Adv. Math.}, 226(2):2020--2039, 2011.
	
	\bibitem[Spe16]{Spe16}
	Daniel Spector.
	\newblock On a generalization of {$L^p$}-differentiability.
	\newblock {\em Calc. Var. Partial Differential Equations}, 55(3):Art. 62, 21,
	2016.
	
	\bibitem[SS16a]{ScSi16}
	Russell~W. Schwab and Luis Silvestre.
	\newblock Regularity for parabolic integro-differential equations with very
	irregular kernels.
	\newblock {\em Anal. PDE}, 9(3):727--772, 2016.
	
	\bibitem[SS16b]{SiSn16}
	Luis Silvestre and Stanley Snelson.
	\newblock An integro-differential equation without continuous solutions.
	\newblock {\em Math. Res. Lett.}, 23(4):1157--1166, 2016.
	
	\bibitem[Sta65]{Sta65}
	Guido Stampacchia.
	\newblock Le probl\`eme de {D}irichlet pour les \'{e}quations elliptiques du
	second ordre \`a coefficients discontinus.
	\newblock {\em Ann. Inst. Fourier (Grenoble)}, 15(fasc. 1):189--258, 1965.
	
	\bibitem[Str19]{Str19a}
	Martin Str\"{o}mqvist.
	\newblock Local boundedness of solutions to non-local parabolic equations
	modeled on the fractional {$p$}-{L}aplacian.
	\newblock {\em J. Differential Equations}, 266(12):7948--7979, 2019.
	
	\bibitem[SW15]{ScWa15}
	Ren\'{e}~L. Schilling and Jian Wang.
	\newblock Lower bounded semi-{D}irichlet forms associated with {L}\'{e}vy type
	operators.
	\newblock In {\em Festschrift {M}asatoshi {F}ukushima}, volume~17 of {\em
		Interdiscip. Math. Sci.}, pages 507--526. World Sci. Publ., Hackensack, NJ,
	2015.
	
	\bibitem[T{\"o}l10]{Toe10}
	Jonas T{\"o}lle.
	\newblock Variational convergence of nonlinear partial differential operators
	on varying {B}anach spaces.
	\newblock {\em PhD Thesis, Bielefeld University}, 2010.
	
	\bibitem[Uem14]{Uem14}
	Toshihiro Uemura.
	\newblock On multidimensional diffusion processes with jumps.
	\newblock {\em Osaka J. Math.}, 51(4):969--992, 2014.
	
\end{thebibliography}

\end{document}